\documentclass[10pt]{amsart}
\usepackage[T1]{fontenc}
\usepackage{mathtools}
\usepackage{xcolor}
\usepackage{amsmath,amssymb}
\usepackage{amsthm}
\usepackage{thmtools}
\usepackage{bm}
\usepackage{tikz}
\usepackage{url}
\usepackage[backend=biber,style=alphabetic, maxbibnames=99, maxalphanames=6, url=false, doi=false]{biblatex}
\newcommand{\titletext}{Grothendieck weights and K-theoretic positivity for matroids}
\usepackage{geometry}
\usepackage{hyperref}
\hypersetup{
  hidelinks,
  colorlinks=true,
  linkcolor=blue,
  citecolor=magenta,
  pdftitle={\titletext},
  pdfauthor={Yiyu Wang}
}
\usepackage{cleveref}

\usepackage{newpxtext}

\title{\titletext}
\author{Yiyu Wang}
\address{Department of Mathematics, The Ohio State University, 231 W. 18th Ave., Columbus, OH 43210}
\email{wang.20315@osu.edu}
\date{August 9, 2026}

\declaretheorem[name=Theorem, numberwithin=section, refname={Theorem,Theorems}]{theorem}
\renewcommand{\thetheorem}{\ifnum\value{section}=1 \Alph{theorem}\else\thesection.\arabic{theorem}\fi}

\declaretheorem[name=Lemma, sibling=theorem, refname={Lemma,Lemmas}]{lemma}
\declaretheorem[name=Proposition, sibling=theorem, refname={Proposition,Propositions}]{proposition}
\declaretheorem[name=Corollary, sibling=theorem, refname={Corollary,Corollaries}]{corollary}
\declaretheorem[name=Conjecture, sibling=theorem, refname={Conjecture,Conjectures}]{conjecture}
\declaretheorem[name=Remark, style=remark, sibling=theorem, refname={Remark,Remarks}]{remark}
\declaretheorem[name=Definition, style=definition, sibling=theorem, refname={Definition,Definitions}]{definition}

\numberwithin{equation}{section}

\crefname{section}{Section}{Sections}
\crefname{subsection}{Subsection}{Subsections}
\crefname{equation}{Equation}{Equations}
\crefname{appendix}{Appendix}{Appendices}
\Crefname{appendix}{Appendix}{Appendices}

\DeclareMathOperator{\Hom}{Hom}
\DeclareMathOperator{\GW}{GW}
\DeclareMathOperator{\rk}{rk}

\DeclareMathOperator{\td}{td}
\DeclareMathOperator{\ch}{ch}

\DeclareMathOperator{\Star}{Star}
\DeclareMathOperator{\cone}{cone}
\DeclareMathOperator{\face}{face}
\DeclareMathOperator{\link}{link}
\DeclareMathOperator{\Loop}{Loop}

\renewcommand{\k}{{\Bbbk}}
\newcommand{\R}{{\mathbb{R}}}
\newcommand{\Q}{{\mathbb{Q}}}
\newcommand{\Z}{{\mathbb{Z}}}
\newcommand{\PP}{{\mathbb{P}}}

\begin{document}

\begin{abstract}
  We introduce a method for studying \(K\)-theoretic positivity on permutohedral toric varieties through the topology of spaces arising in tropical geometry. The key ingredient is the theory of Grothendieck weights developed by the author.

  We prove two positivity results using this method. The first result is the positivity of the Euler characteristics of tautological bundles associated with an arbitrary matroid and twisted by a nef line bundle. This gives numerical evidence for a conjectural vanishing theorem.

  The second result generalizes the external activity complex of Berget--Fink, originally defined for a pair of matroids, to the case of any tuple of matroids with no common loop. We deduce a formula for its \(\Z^k\)-graded \(K\)-polynomial in terms of exterior powers of the dual tautological quotient classes of the matroids. After a change of variables, its coefficients alternate in sign. We also prove the Cohen--Macaulayness of each such complex using the vanishing theorems for combinatorial geometries developed by Eur--Fink--Larson. This proof is new even in the case of a pair of matroids. As an application, we interpret certain Chern numbers of tautological quotient classes as counts of facets, partially answering a question of Berget--Eur--Spink--Tseng.

\end{abstract}

\maketitle
\setcounter{tocdepth}{1}
\tableofcontents

\section{Introduction}
\subsection{Background}

Speyer's 20-year-old tropical \(f\)-vector conjecture \cite{speyer2008tropical} (see also \cite[Conjecture~1.1]{fink2024omega}) states that, if one subdivides a rank \(r\) hypersimplex into matroid base polytopes, the number of interior faces of dimension \(n-i\) is bounded by
\[
  \frac{(n-i-1)!}{(r-i)!(n-r-i)!(i-1)!}.
\]

To bound these face numbers, Fink--Shaw--Speyer \cite{fink2024omega} introduced the \(\omega\)-invariant of a matroid. Let \(E=\{1,2,\ldots,n\}\) denote the ground set of \(\mathsf{M}\), \(X_E\) denote the permutohedral toric variety, and \(\mathcal{Q}_\mathsf{M}\) denote the tautological quotient bundle defined in \cite{BergetEurSpinkTseng2023}. The \emph{\(\omega\)-invariant} of a rank \(r\) matroid \(\mathsf{M}\) is defined to be \cite[page 7]{berget2025externalactivitycomplexpair}
\[
  \omega(\mathsf{M})=(-1)^{c}\chi\left(X_E,\bigwedge^{n-r} \mathcal{Q}_\mathsf{M}^\vee \cdot \bigwedge^r\mathcal{Q}_\mathsf{M}^\vee\right),
\]
where \(c\) is the number of connected components of \(\mathsf{M}\), and \(\vee\) denotes dualization.

There are two different proofs of the positivity of the \(\omega\)-invariant. Berget--Fink \cite[Theorem~A]{berget2025externalactivitycomplexpair} constructed a Cohen--Macaulay complex, which they called the \emph{external activity complex of a pair of matroids}, whose \(\Z^2\)-graded \(K\)-polynomial is equal to
\[
  \chi\left(X_E,\lambda_{-U_1}(\mathcal{Q}_{\mathsf{M}_1}^\vee) \lambda_{-U_2}(\mathcal{Q}_{\mathsf{M}_2}^\vee)\right).
\]
Here and throughout, \(\lambda_u(\mathcal E)\) denotes the total exterior power of a class \(\mathcal E\).
For \(\mathsf M_1=\mathsf M_2=\mathsf M\), one of the coefficients of this \(K\)-polynomial recovers \(\omega(\mathsf M)\), and Cohen--Macaulayness gives its nonnegativity.

On the other hand, Eur--Fink--Larson \cite[page~4]{eur2025vanishingtheoremscombinatorialgeometries} used vanishing theorems on wonderful varieties to prove the same positivity. They rewrote the \(\omega\)-invariant as
\[
  \omega(\mathsf{M})=(-1)^{r-n+\dim P(\mathsf{M})}\chi\left(X_E,\lambda_{-1}(\mathcal{Q}_\mathsf{M}^\vee) \mathcal{L}_{-P(\mathsf{M})}^{-1}\right),
\]
where \(P(\mathsf M)\) is the matroid base polytope of \(\mathsf M\), and deduced the positivity from vanishing theorems.

In this paper, we combine these two approaches using the theory of \emph{Grothendieck weights} developed by the author in \cite{wang2026grothendieckweights}. We reprove and generalize the results of \cite{berget2025externalactivitycomplexpair}---the \(K\)-polynomial identity and the Cohen--Macaulayness of the external activity complex---from pairs to \(k\)-tuples of matroids with no common loop, for any \(k\geq2\), using the Cohen--Macaulayness of tropical degenerations established in \cite{eur2025vanishingtheoremscombinatorialgeometries}. Our method can translate a \(K\)-positivity problem on permutohedral toric varieties into a topological question. To demonstrate this method, we also prove a positivity result for Euler characteristics of tautological bundles twisted by a nef line bundle on \(X_E\).

\subsection{Main results}

Our first result establishes positivity for tautological bundles under nef twists. Fix a field \(\k\).  For a positive integer \(m\), write \([m]=\{1,\ldots,m\}\).  Let \(E=[n]\), and let \(X_E\) be the permutohedral toric variety on \(E\), a smooth projective toric variety of dimension \(n-1\).  Write \(K(X_E)\) for its Grothendieck ring.  For a matroid \(\mathsf M\) on \(E\), let \(\mathcal S_{\mathsf M},\mathcal Q_{\mathsf M}\in K(X_E)\) be its tautological subbundle and quotient classes, respectively, as defined in \cite{BergetEurSpinkTseng2023}.

\emph{Nef} torus-invariant line bundles on \(X_E\) correspond to lattice \emph{generalized permutohedra}. If \(P\) is a lattice generalized permutohedron, denote the corresponding nef line bundle on \(X_E\) by \(\mathcal L_P\).

\begin{restatable}{theorem}{thmENefTautologicalPositivity}
  \label{thm:nef-tautological-positivity}
  Let \(\mathsf M\) be a loopless matroid on \(E\), and let \(\mathcal L\) be a nef line bundle on \(X_E\). Then
  \[
    \frac{1}{1+u}\chi\left(
      X_E,
      \lambda_u(\mathcal S_{\mathsf M}^{\vee})
      \lambda_{-1}(\mathcal Q_{\mathsf M}^{\vee})\mathcal L
    \right)
    \in\Z_{\geq0}[u].
  \]
\end{restatable}

The geometric meaning of this result suggests a stronger vanishing conjecture.  When \(\mathsf{M}\) is realizable, suppose that \(L\subseteq\k^E\) realizes \(\mathsf M\), let \(W_L\subseteq X_E\) be its wonderful compactification, and let \(D_L\) be the boundary divisor. By \cite[Theorem~7.10]{BergetEurSpinkTseng2023}, \(W_L\) is the zero locus of a regular section of \(\mathcal Q_{L}\). Consequently, we can identify \([\lambda_{-1}\mathcal{Q}_L^\vee]\) with \([\mathcal{O}_{W_L}]\). Dualizing \cite[Theorem~8.8]{BergetEurSpinkTseng2023} gives the logarithmic Euler sequence
\[
  0\longrightarrow\Omega^1_{W_L}(\log D_L)
  \longrightarrow\mathcal S_L^\vee|_{W_L}
  \longrightarrow\mathcal O_{W_L}\longrightarrow0.
\]
For any lattice generalized permutohedron \(P\), this sequence gives
\begin{equation*}
  \frac{1}{1+u}\chi\left(
    X_E,
    \lambda_u(\mathcal S_{\mathsf M}^{\vee})
    \lambda_{-1}(\mathcal Q_{\mathsf M}^{\vee})\mathcal L_P
  \right)
  =\sum_{p\geq0}
  \chi\left(
    W_L,
    \Omega^p_{W_L}(\log D_L)\otimes\mathcal L_P|_{W_L}
  \right)u^p.
\end{equation*}

\cref{thm:nef-tautological-positivity} states that each \(\chi\left( W_L, \Omega^p_{W_L}(\log D_L)\otimes\mathcal L_P|_{W_L} \right)\) is nonnegative. This suggests the following conjecture.

\begin{conjecture}\label{conj:twisted-logarithmic-vanishing}
  Let \(L\subseteq\k^E\) realize a loopless matroid, and let \(P\) be a lattice generalized permutohedron.  Then, for every \(p\geq0\) and \(q>0\),
  \[
    H^q\left(
      W_L,
      \Omega^p_{W_L}(\log D_L)\otimes\mathcal L_P|_{W_L}
    \right)=0.
  \]
\end{conjecture}

The case \(p=0\) is \cite[Theorem~A(1)]{eur2025vanishingtheoremscombinatorialgeometries}, and the case \(P=\{0\}\) is \cite[Theorem~A(i)]{liu2025vanishingwonderful}. Thus \cref{conj:twisted-logarithmic-vanishing} is a common generalization of two known vanishing theorems.
When \(P\) is a \emph{realizable} polymatroid over a field of characteristic zero, the conjecture follows from \cite[Corollary~9.20 and Remark~9.31]{larson2026aspects}.

We next state our second result. We generalize the external activity complex of a pair of matroids in \cite{berget2025externalactivitycomplexpair} to arbitrary \(k\)-tuples with no common loop. We state its \(K\)-theoretic positivity result first.

For a tuple of matroids \(\bm{\mathsf{M}}=(\mathsf{M}_1,\ldots,\mathsf{M}_k)\) on \(E\) with no common loop, meaning that no element of \(E\) is a loop of every \(\mathsf M_j\), set
\[
  \mathcal K_{\bm{\mathsf M}}(V_1,\ldots,V_k)
  =
  \chi\left(
    X_E,
    \prod_{j=1}^k
    \lambda_{-V_j}(\mathcal Q_{\mathsf M_j}^{\vee})
  \right),
\]
where \(V_1,\ldots, V_k\) are indeterminates.
For a polynomial \(f\) and a monomial \(m\), we write \([m]f\) for the coefficient of \(m\) in \(f\).

\begin{restatable}{theorem}{thmAPositivity}
  \label{thm:A-positivity-K-poly}
  Let \(k\geq2\), and let \(\bm{\mathsf M}\) be a tuple of \(k\) matroids on
  \(E\) with no common loop.  For all \(d_1,\ldots,d_k\geq0\),
  \[
    (-1)^{d_1+\cdots+d_k-c}
    [U_1^{d_1}\cdots U_k^{d_k}]\,
    \mathcal K_{\bm{\mathsf M}}
    (1-U_1,\ldots,1-U_k)
    \geq0.
  \]
\end{restatable}

Here \(c\) is an invariant of the tuple that will be defined below. In the case \(k=2\), it is the corank of the \emph{diagonal Dilworth truncation} matroid defined in \cite{berget2025externalactivitycomplexpair}. For \(k>2\), the \emph{diagonal Dilworth polymatroid} is the correct generalization.
For \(A\subseteq E\), define
\[
  \rho(A)
  =
  \min_{A=A_1\sqcup\cdots\sqcup A_s}
  \sum_{q=1}^s
  \left(\sum_{j=1}^k\rk_{\mathsf M_j}(A_q)-1\right),
  \qquad
  \rho(\varnothing)=0,
\]
where the parts \(A_q\) are nonempty.  We prove in \cref{prop:rho-is-polymatroid} that \(\rho\) is the rank function of an integral polymatroid, which we call the \emph{diagonal Dilworth polymatroid} of \(\bm{\mathsf M}\). The integer \(c\) in \cref{thm:A-positivity-K-poly} is
\[
  c=(k-1)n-\rho(E).
\]

The positivity in \cref{thm:A-positivity-K-poly} is a consequence of the following stronger result.  For a sufficiently generic tuple of vectors \(\bm w\), we construct the external activity complex \(\Delta_{\bm w}(\bm{\mathsf M})\) on the ambient set \(\widetilde E=[k]\times E\); see \cref{def:k-fold-external-activity-complex}.  Equip its Stanley--Reisner ring with the \(\Z^k\)-grading in which the variable indexed by \((j,i)\in\widetilde E\) has degree \(e_j\), and denote the corresponding \(K\)-polynomial by \(\mathcal K(\Delta_{\bm w}(\bm{\mathsf M});V_1,\ldots,V_k)\).
\begin{restatable}{theorem}{thmBExternalActivity}
  \label{thm:B-CM-Delta}
  Let \(k\geq2\), let \(\bm{\mathsf M}\) be a tuple of \(k\) matroids on
  \(E\) with no common loop, and let \(\bm w\) be sufficiently generic.
  \begin{enumerate}
    \item One has
      \[
        \mathcal K(\Delta_{\bm w}(\bm{\mathsf M});V_1,\ldots,V_k)
        =\mathcal K_{\bm{\mathsf M}}(V_1,\ldots,V_k).
      \]
      In particular, the left-hand side is independent of \(\bm w\) and
      is valuative in each of \(\mathsf M_1,\ldots,\mathsf M_k\).
    \item The complex \(\Delta_{\bm w}(\bm{\mathsf M})\) is
      Cohen--Macaulay over \(\k\) of dimension
      \(
        n+\rho(E)-1.
      \)
  \end{enumerate}
\end{restatable}

As an application, we interpret certain Chern numbers as facet counts.

\begin{restatable}{theorem}{thmDChernMonomials}
  \label{thm:D-Chern-monomials}
  Let \(n\geq2\) and \(k\geq\max\{2,n-1\}\), let \(\bm{\mathsf M}=(\mathsf M_1,\ldots,\mathsf M_k)\) be a \(k\)-tuple of connected matroids on \(E=[n]\), and let \(\bm w\) be sufficiently generic. Let \(\bm l=(l_1,\ldots,l_k)\) be a tuple of nonnegative integers with
  \(l_1+\cdots+l_k=n-1\).  Then
  \[
    \int_{X_E}\prod_{j=1}^k c_{l_j}(\mathcal Q_{\mathsf M_j})
    =
    \#\left\{
      F\text{ a facet of }\Delta_{\bm w}(\bm{\mathsf M}):
      |F\cap(\{j\}\times E)|=\rk(\mathsf M_j)+l_j\text{ for every }j\in[k]
    \right\}.
  \]
\end{restatable}

Although \(\Delta_{\bm w}(\bm{\mathsf{M}})\) depends on \(\bm w\), the multidegree distribution counted in \cref{thm:D-Chern-monomials} does not. The facets appearing in \cref{thm:D-Chern-monomials} can equivalently be indexed by the integer bases of the diagonal Dilworth polymatroid of \(\bm{\mathsf M}\).
The nonnegativity of these Chern numbers follows directly from the fan displacement rule, since the Chern classes \(c_i(\mathcal Q_{\mathsf M_j})\) are represented by nonnegative Minkowski weights.  When \(\mathsf M_1=\cdots=\mathsf M_k=\mathsf M\), the theorem partially answers \cite[Question~1.4]{BergetEurSpinkTseng2023} for products of column Schur classes of \(\mathcal Q_\mathsf M\).

\subsection{Grothendieck weights and a prototype}

The proofs of the preceding results, as well as the construction of \(\Delta_{\bm w}(\bm{\mathsf M})\), rely on the theory of Grothendieck weights on the permutohedral toric variety developed in \cite{wang2026grothendieckweights}.  We extend this theory from nonequivariant to equivariant \(K\)-theory; see \cref{thm:equivariant-product-rule}.

We first briefly introduce the notion of Grothendieck weight.
Let \(\Sigma_E\) be the braid fan of \(X_E\), and set
\[
  x_\sigma=[\mathcal O_{V(\sigma)}]\in K(X_E)
  \qquad(\sigma\in\Sigma_E),
\]
where \(V(\sigma)\) is the closure of the torus orbit corresponding to \(\sigma\).  Grothendieck weights play in \(K\)-theory the role that Minkowski weights play in Chow theory.  A \emph{Grothendieck weight} on \(\Sigma_E\) is a function \(g:\Sigma_E\to\Z\) that annihilates every linear relation among the classes \(x_\sigma\): if
\[
  \sum_{\sigma\in\Sigma_E}c_\sigma x_\sigma=0
  \quad\text{in \(K(X_E)\),}
\]
then
\[
  \sum_{\sigma\in\Sigma_E}c_\sigma g(\sigma)=0.
\]
The Euler pairing identifies \(K(X_E)\) with the ring of Grothendieck
weights by sending \(\alpha\in K(X_E)\) to
\[
  g_\alpha(\sigma)=\chi(X_E,\alpha x_\sigma).
\]
In particular, \(g_\alpha(\{0\})=\chi(X_E,\alpha)\).

As in the theory of Minkowski weights, Grothendieck weights satisfy a \emph{\(K\)-balancing condition} that describes when a function \(g\) is a Grothendieck weight, and a \emph{product rule} that determines how two Grothendieck weights multiply in \(K(X_E)\).

The main ingredient in this paper is the product rule, which translates a problem in \(K\)-theory to a topological question. To state the product rule, choose \(w\in \R^n/\R(1,1,\ldots,1)\) to be a sufficiently generic vector. At the zero cone, the product rule says that, for any two Grothendieck weights \(g_1,g_2\),
\[
  (g_1g_2)(\{0\})
  =
  \sum_{\substack{\sigma,\tau\in\Sigma_E\\
      (\sigma+w)\cap\tau\neq\varnothing\\
  \sigma\cap\tau=\{0\}}}
  (-1)^{\dim\sigma+\dim\tau-n+1}
  g_1(\sigma)g_2(\tau).
\]

The summation above can be rewritten as follows. Pick two copies of \(\Sigma_E\), and move one of them by a sufficiently generic vector \(w\). Let \((\Sigma_E+w)\wedge \Sigma_E\) denote the common refinement of these two complexes. A cell of \((\Sigma_E+w)\wedge \Sigma_E\) takes the form \((\sigma^\circ+w)\cap\tau^\circ\), and the cell is \emph{bounded} if and only if \(\sigma\cap\tau=\{0\}\). Therefore, the product rule sums over all \emph{bounded} cells of \((\Sigma_E+w)\wedge \Sigma_E\). Note that the dimension of the cell is exactly \(\dim \sigma+\dim \tau-n+1\). The product rule for Minkowski weights sums over \emph{vertices} of \((\Sigma_E+w)\wedge \Sigma_E\).

Using the product rule, we can often translate a \(K\)-positivity problem into a question about a topological invariant, usually the Euler characteristic, of a subcomplex of \((\Sigma_E+w)\wedge\Sigma_E\).

Let us discuss a prototype example. Suppose we aim to prove the following statement.
\begin{restatable}[{Corollary of \cite[Theorem~B]{eur2025vanishingtheoremscombinatorialgeometries}}]{corollary}{corCAntiample}
  \label{thm:C-example-application}
  Let \(\mathsf M\) be a loopless matroid of rank \(r\) on \(E=[n]\), and
  let \(\mathcal L\) be an ample line bundle on \(X_E\).  Then
  \[
    (-1)^{r-1}\chi_\mathsf M(\mathcal L^{-1})\geq0.
  \]
\end{restatable}

Let \(\Sigma_\mathsf M\) be the Bergman fan of \(\mathsf{M}\), regarded as a subfan of \(\Sigma_E\), and let \(\Delta_\mathsf M\in K(X_E)\) be the class whose Grothendieck weight is the indicator function
\[
  g_{\Delta_\mathsf M}(\sigma)
  =
  \begin{cases}
    1, & \sigma\in\Sigma_\mathsf M,\\
    0, & \sigma\notin\Sigma_\mathsf M.
  \end{cases}
\]
The canonical Euler characteristic of the matroid \(K\)-ring satisfies
\[
  \chi_\mathsf M(\alpha)
  =\chi(X_E,\Delta_\mathsf M\alpha);
\]
see \cite[Section~1.5]{larson2024krings}.  Applying the product rule above gives
\[
  \chi_\mathsf M(\alpha)
  =\sum_{\tau\in\Sigma_E}c_\tau g_\alpha(\tau),
  \qquad
  c_\tau
  =
  \sum_{\substack{\sigma\in\Sigma_\mathsf M\\
      (\sigma+w)\cap\tau\neq\varnothing\\
  \sigma\cap\tau=\{0\}}}
  (-1)^{\dim\sigma+\dim\tau-n+1}.
\]

To prove \cref{thm:C-example-application}, we need to determine the sign of \(c_\tau\). Using ampleness and Ehrhart reciprocity, we only need to show that \((-1)^{n+r-\dim \tau}c_\tau\geq 0\); for more details, see \cref{sec:ample-k-positivity}.

The formula for \(c_\tau\) suggests considering the complex which consists of cells \((\sigma^\circ+w)\cap\tau^\circ\) with \(\sigma\in\Sigma_\mathsf{M}\) and \(\tau\) fixed. To make it closed, we need to consider
\[
  P_\tau^b=\bigcup_{\sigma\in \Sigma_\mathsf{M},\gamma\leq \tau}(\sigma^\circ+w)\cap \gamma^\circ
\]
where each intersection in the union is required to be bounded. Taking the Euler characteristic, we obtain the following formula:
\[
  \chi(P_\tau^b)=\sum_{\gamma\leq \tau}c_\gamma.
\]
The M\"obius inversion formula gives
\[
  c_\tau=\sum_{\gamma\leq \tau}(-1)^{\dim\tau-\dim\gamma}\chi(P_\gamma^b).
\]

In this way, we translate an Euler characteristic in \(K(X_E)\) to an Euler characteristic of some complex. In \cref{sec:ample-k-positivity}, we show that each \(P_\gamma^b\) is either empty or contractible, and the above summation can be interpreted as the reduced Euler characteristic of a Cohen--Macaulay simplicial complex. The Cohen--Macaulayness and Reisner's criterion then give the sign.

The proof of \cref{thm:B-CM-Delta} follows the same strategy.  An equivariant refinement of the rule retains the fine grading and expresses the Euler characteristic as an alternating sum over bounded intersections of translated cones.  Grouping terms and applying M\"obius inversion identifies each fine coefficient with the negative reduced Euler characteristic of a link of \(\Delta_{\bm w}(\bm{\mathsf M})\), so the desired signs follow from its Cohen--Macaulayness.  To prove this property, we realize \(\operatorname{MultiProj}\k[\Delta_{\bm w}(\bm{\mathsf M})]\) as the scheme-theoretic image in \((\PP^{k-1})^n\) of a tropical initial degeneration associated with \(\bigoplus_j\mathsf M_j\).  The results of \cite{eur2025vanishingtheoremscombinatorialgeometries} show that this image is Cohen--Macaulay, and an analysis of affine charts lifts this property to \(\Delta_{\bm w}(\bm{\mathsf M})\).

\subsection{Comparison with the method of Berget--Fink}

For pairs of matroids, Berget--Fink work first in Chow theory.  They view the Chern classes \(c_i(\mathcal Q_{\mathsf M_1})\) and \(c_j(\mathcal Q_{\mathsf M_2})\) as Minkowski weights and apply the fan displacement rule to generic translates of their supports.  As \(i\) and \(j\) vary, these intersections determine tropical cells indexed by certain faces of \(\Delta_{\bm w}(\mathsf M_1,\mathsf M_2)\).  Berget--Fink recover the finely graded \(K\)-polynomial from the compactly supported Euler characteristics of these cells; see \cite[Section~6]{berget2025externalactivitycomplexpair}.

The relation with Grothendieck weights is visible in the classes being multiplied. The splitting principle gives, for \(0\leq i\leq \rk(\mathcal Q_{\mathsf M_j})\),
\begin{equation}\label{eq:lambda-chern-leading-term}
  \left[U_j^{\rk(\mathcal Q_{\mathsf M_j})-i}\right]
  \ch\left(
    \lambda_{U_j-1}(\mathcal Q_{\mathsf M_j}^{\vee})
  \right)
  =
  c_i(\mathcal Q_{\mathsf M_j})
  +\text{terms of Chow degree greater than \(i\)}.
\end{equation}
Thus the Chern-cycle Minkowski weight used by Berget--Fink is the lowest-Chow-degree part of the corresponding coefficient of the \(\lambda\)-class whose Grothendieck weight enters our product rule.  In this sense, their fan-displacement calculation is the Chow-theoretic shadow of the Grothendieck-weight product rule.  Berget--Fink reconstruct the full \(K\)-polynomial from all of the resulting Chern-cycle cells, whereas the equivariant Grothendieck-weight product rule performs the multiplication directly in \(K\)-theory and retains the fine grading.  Its higher version also applies to tuples of arbitrary size.

When the selected Chern degrees sum to \(\dim X_E=n-1\), this comparison becomes an equality of numerical invariants.  Hirzebruch--Riemann--Roch shows that neither the higher-Chow-degree terms in the expansion above nor the positive-degree part of the Todd class can contribute.  The corresponding coefficient of the \(K\)-polynomial is therefore the integral of the product of the leading Chern classes.  This is the mechanism behind \cref{thm:D-Chern-monomials}.

\subsection{Outline of the paper}

In \cref{sec:preliminaries}, we fix the permutohedral notation and recall the required facts about matroids, Grothendieck weights, matroid \(K\)-rings, and Cohen--Macaulay complexes. In \cref{sec:equivariant-gw}, we develop equivariant Grothendieck weights, prove the equivariant product rule, state its higher analogue on the braid fan, and compute the weights of tautological classes. In \cref{sec:nef-tautological-positivity}, we prove positivity for tautological bundles under nef twists predicted by \cref{conj:twisted-logarithmic-vanishing} and deduce \cref{thm:C-example-application}. In \cref{sec:k-fold-positivity}, we define the diagonal Dilworth polymatroid and the external activity complex for a tuple, identify the relevant \(K\)-polynomial coefficients with Euler characteristics of links, and reduce \cref{thm:A-positivity-K-poly} to Cohen--Macaulayness. In \cref{sec:k-fold-point-complex-cm}, we prove this Cohen--Macaulayness and the facet--basis correspondence, completing the proofs of \cref{thm:A-positivity-K-poly,thm:B-CM-Delta}. In \cref{sec:Chern-Dilworth}, we prove \cref{thm:D-Chern-monomials}. Finally, \cref{sec:ample-k-positivity} gives an alternative proof of \cref{thm:C-example-application} and constructs \(\mathcal C_w(\mathsf M)\), while \cref{sec:proof-higher-product-rule} proves the higher product rule stated in \cref{sec:equivariant-gw}.

\subsection*{Acknowledgements}
The author would like to thank Eric Katz and Matt Larson for their helpful discussions and comments on an earlier draft.
\subsection*{AI disclosure}
The proofs of \cref{lemma:equivariant-grassmannian-comparison,lem:k-fold-loop-reduction,lem:k-fold-first-wall-retraction} were completed with the assistance of generative AI.  The overall mathematical framework and the central ideas of this paper were developed independently by the author, building on the author's earlier work \cite{wang2026grothendieckweights}.  The author has verified all AI-assisted arguments and takes full responsibility for the contents of the paper.

\section{Preliminaries}\label{sec:preliminaries}

All schemes and Stanley--Reisner rings are taken over \(\k\) unless otherwise specified.

\subsection{Fans and permutohedral notation}
\label{subsec:permutohedral-fan}

Let \(N\) be a lattice, let \(M=\Hom(N,\Z)\), and let \(\Sigma\) be a fan in \(N_\R\).  For a cone \(\sigma\in\Sigma\), set
\[
  N_\sigma=N\cap\operatorname{span}_\R(\sigma),
  \qquad
  N(\sigma)=N/N_\sigma,
  \qquad
  M(\sigma)=\sigma^\perp\cap M.
\]
Thus \(M(\sigma)=\Hom(N(\sigma),\Z)\) is the dual lattice of \(N(\sigma)\), where \(\sigma^\perp\subseteq M_\R\) denotes the real annihilator of \(\sigma\).
For \(\gamma\in\Sigma\), let
\[
  \pi_\gamma:N_\R\longrightarrow N(\gamma)_\R
  =N_\R/(N_\gamma)_\R
\]
be the quotient map.  Thus the star fan of \(\gamma\) is
\[
  \Star_\Sigma(\gamma)
  =
  \{\pi_\gamma(\sigma):\sigma\in\Sigma,\ \gamma\subseteq\sigma\}
  \subseteq N(\gamma)_\R.
\]
We say that a polyhedron \(P\subseteq N_\R\) is \emph{bounded modulo \((N_\gamma)_\R\)} if \(\pi_\gamma(P)\) is bounded.

We now specialize this notation to the permutohedral setting.  For a finite set \(A\), let \((e_i)_{i\in A}\) be the standard basis of \(\Z^A\), and set \(\mathbf 1_A=\sum_{i\in A}e_i\).  We use the lattices
\[
  N_A=\Z^A/\Z\mathbf 1_A,
  \qquad
  N_{A,\R}=N_A\otimes_\Z\R
  =\R^A/\R\mathbf 1_A,
\]
and
\[
  M_A=\Hom(N_A,\Z)
  =
  \left\{q\in\Z^A:\sum_{i\in A}q_i=0\right\},
  \qquad
  M_{A,\R}=M_A\otimes_\Z\R.
\]
Let
\[
  \pi_A:\R^A\longrightarrow N_{A,\R}
\]
be the quotient map.  Let \(\bar e_i=\pi_A(e_i)\), and, for \(S\subseteq A\), let
\[
  \bar e_S
  =
  \pi_A\left(\sum_{i\in S}e_i\right)
  =
  \sum_{i\in S}\bar e_i.
\]

The \emph{permutohedral fan}, also called the \emph{braid fan}, on \(A\) is the fan \(\Sigma_A\) in \(N_{A,\R}\) whose cones are indexed by flags
\[
  \mathcal G:
  \varnothing\subsetneq G_1\subsetneq\cdots\subsetneq G_\ell\subsetneq A.
\]
Our convention is
\[
  \sigma_{\mathcal G}
  =
  \cone\bigl(
    \bar e_{G_1},\ldots,
  \bar e_{G_\ell}\bigr),
  \qquad
  \ell(\mathcal G)=\ell=\dim\sigma_{\mathcal G}.
\]
In particular, the ray indexed by a nonempty proper subset \(S\subsetneq A\) is
\[
  \rho_S=\R_{\geq0}\bar e_S,
  \qquad
  u_{\rho_S}=\bar e_S.
\]
We write \(X_A=X_{\Sigma_A}\) for the corresponding permutohedral toric variety.

Throughout the paper, \(E=[n]\), and we abbreviate
\[
  N=N_E,
  \qquad
  M=M_E,
  \qquad
  N_\R=N_{E,\R},
  \qquad
  M_\R=M_{E,\R}.
\]
For a function \(g\) on \(\Sigma_E\), we abbreviate \(g(\sigma_{\mathcal G})\) to \(g(\mathcal G)\).

\subsection{Matroids and polymatroids}\label{subsec:matroids-polymatroids}

For a finite set \(A\), a function
\(\rho:2^A\to\Z_{\geq0}\) is the rank function of an \emph{integral polymatroid} if it satisfies
\begin{enumerate}
  \item (normalized) \(\rho(\varnothing)=0\),
  \item (nondecreasing) \(\rho(S)\leq\rho(T)\) for \(S\subseteq T\),
  \item (submodular) \(
    \rho(S)+\rho(T)\geq\rho(S\cap T)+\rho(S\cup T)\)
    for \(S,T\subseteq A\).
\end{enumerate}

For \(u\in\R^A\) and \(S\subseteq A\), write \(u(S)=\sum_{i\in S}u_i\).  The polymatroid polytope and base polytope of
\(\rho\) are
\[
  I(\rho)
  =\{u\in\R_{\geq0}^A:u(S)\leq\rho(S)
  \text{ for all }S\subseteq A\},
  \qquad
  B(\rho)=\{u\in I(\rho):u(A)=\rho(A)\}.
\]
We write
\[
  \mathcal B(\rho)=B(\rho)\cap\Z^A
\]
for the set of \emph{integer bases} of \(\rho\).

A \emph{matroid} \(\mathsf M\) on \(A\) is an integral polymatroid whose rank function additionally satisfies \(\rk_{\mathsf M}(\{i\})\leq1\) for every \(i\in A\). A subset \(I\subseteq A\) is independent if \(\rk_{\mathsf M}(I)=|I|\), and a maximal independent set is a basis; we write \(\mathcal B(\mathsf M)\) for the set of bases.

An element of \(A\) is a \emph{loop} of \(\mathsf{M}\) if it belongs to no basis and a \emph{coloop} if it belongs to every basis, and \(\Loop(\mathsf M)\subseteq A\) denotes the set of loops.

The uniform matroid of rank \(r\) on \(A\) is denoted \(\mathsf{U}_{r,A}\); when \(A=[m]\), we also write \(\mathsf{U}_{r,m}\). In particular, if \(r=|A|\), \(\mathsf{U}_{r,A}\) is the Boolean matroid: every subset of \(A\) is independent.

We use the standard notation for restriction (\(\mathsf{M}|S\)), contraction (\(\mathsf{M}/S\)), and direct sum of matroids (\(\mathsf{M}_1\oplus \mathsf{M}_2\)).

\subsection{Initial matroids and relaxed Bergman fans}
\label{subsec:initial-matroids}

\begin{definition}\label{def:flag-initial-matroid}
  Let \(\mathsf M\) be a matroid on \(E\), and let
  \[
    \mathcal F:\varnothing=F_0\subsetneq F_1\subsetneq\cdots\subsetneq
    F_\ell\subsetneq F_{\ell+1}=E
  \]
  be a flag of subsets.  The \emph{initial matroid of \(\mathsf M\) along \(\mathcal F\)} is
  \[
    \operatorname{in}_{\mathcal F}\mathsf M
    =
    \bigoplus_{i=0}^{\ell}(\mathsf M|F_{i+1})/F_i,
  \]
  where the \(i\)-th summand has ground set \(F_{i+1}\setminus F_i\).

  More generally, for \(u\in N_\R\), let \(\operatorname{in}_u\mathsf M\) be the matroid whose bases are the bases of \(\mathsf M\) having maximum \(u\)-weight.  This is independent of the chosen lift of \(u\) to \(\R^E\), since adding a constant to every coordinate changes the weight of every basis by the same amount.  If \(u\in\sigma_{\mathcal F}^{\circ}\), then the greedy algorithm gives
  \[
    \operatorname{in}_u\mathsf M
    =\operatorname{in}_{\mathcal F}\mathsf M.
  \]
  For the empty flag of proper subsets, where \(\ell=0\), this gives \(\operatorname{in}_{\varnothing}\mathsf M=\mathsf M\).
\end{definition}

We use the min-plus convention for tropical geometry notions.  If \(\mathsf M\) is loopless, its \emph{Bergman fan} is
\[
  \Sigma_{\mathsf M}
  =\{u\in N_\R:\Loop(\operatorname{in}_u\mathsf M)=\varnothing\}.
\]
It is the subfan of \(\Sigma_E\) whose cones \(\sigma_\mathcal F\) are indexed by flags \(\mathcal F\) of nonempty proper flats of \(\mathsf M\). If \(\mathsf M\) has rank \(r\), then \(\Sigma_\mathsf M\) has dimension \(r-1\).

More generally, for a matroid \(\mathsf M\) on \(E\) and \(S\subseteq E\), define the \emph{\(S\)-relaxed Bergman fan} by
\[
  \Sigma_{\mathsf M,S}
  =
  \left\{
    u\in N_\R:
    \Loop(\operatorname{in}_u\mathsf M)\subseteq S
  \right\}.
\]
This recovers \(\Sigma_{\mathsf M}\) when \(S=\varnothing\) and \(\mathsf M\) is loopless.  If \(\mathsf M\) has a loop outside \(S\), then \(\Sigma_{\mathsf M,S}=\varnothing\).

A subset \(C\subseteq N_\R\) is \emph{tropically convex} if, whenever the classes of \(u,v\in\R^E\) lie in \(C\) and \(a,b\in\R\), the class of the coordinatewise minimum \(\min(a+u,b+v)\) also lies in \(C\).  We refer to \cite{develin2004tropical} for further details on tropical convexity.

\begin{proposition}\label{prop:relaxed-bergman-convex}
  The relaxed Bergman fan \(\Sigma_{\mathsf M,S}\) is closed and tropically convex.
\end{proposition}

\begin{proof}
  We may assume \(\Loop(\mathsf M)\subseteq S\).  Let \(q_S:\widetilde E_S\to E\) have a two-element fiber \(\{s^0,s^1\}\) over each \(s\in S\) and a one-element fiber over each \(i\notin S\).  Let \(\operatorname{Ser}_S(\mathsf M)\) be the matroid obtained by successively applying the standard series-extension operation \cite[Section~5.4]{oxley2011matroid} to the elements of \(S\).  Its circuits are \(q_S^{-1}(C)\), where \(C\) runs over the circuits of \(\mathsf M\). It is loopless because each singleton circuit \(\{s\}\) is replaced by the two-element circuit \(\{s^0,s^1\}\).

  Define the coordinate-duplication map
  \[
    \delta_S:N_\R\longrightarrow N_{\widetilde E_S,\R},
    \qquad
    (\delta_Su)_a=u_{q_S(a)}.
  \]
  For the maximum-weight initial matroid, the circuit criterion says that an element is a loop precisely when it is the unique \(u\)-minimum of a circuit.  The minimum of \(\delta_Su\) on \(q_S^{-1}(C)\) is unique precisely when the minimum of \(u\) on \(C\) is unique and attained outside \(S\).  Hence
  \[
    \Sigma_{\mathsf M,S}
    =
    \delta_S^{-1}
    \bigl(\Sigma_{\operatorname{Ser}_S(\mathsf M)}\bigr).
  \]
  Bergman fans are closed and tropically convex \cite[Corollary~7]{develin2004tropical}, and \(\delta_S\) is tropical linear because it only duplicates coordinates.  The result follows.
\end{proof}

For a finite set \(A\) and \(S\subseteq A\), let
\[
  C_S=\cone\{\bar e_i:i\in S\}
  \subseteq N_{A,\R}.
\]
If \(S\subsetneq A\), this is the cone over the simplex on \(S\); if \(S=A\), then \(C_A=N_{A,\R}\).

\begin{lemma}
  \label{lem:relaxed-bergman-completion}
  Let \(\mathsf M\) be a loopless matroid on a finite set \(E\), let \(v\in N_\R\), and let \(S\subseteq E\).  Then
  \[
    \Loop(\operatorname{in}_v\mathsf M)\subseteq S
    \quad\Longleftrightarrow\quad
    (v+C_S)\cap\Sigma_{\mathsf M}\neq\varnothing.
  \]
  Equivalently,
  \[
    \Sigma_{\mathsf M,S}=\Sigma_{\mathsf M}-C_S.
  \]
\end{lemma}

\begin{proof}
  If \(S=E\), then the condition on the loop set is automatic.  Moreover, \(C_E=N_\R\), because \(-\bar e_i=\sum_{j\in E\setminus\{i\}}\bar e_j\) for every \(i\), and \(0\in\Sigma_{\mathsf M}\) because \(\mathsf M\) is loopless.  Thus both assertions are immediate in this case, so we may assume \(S\subsetneq E\).

  Let \(\pi_E:\R^E\to N_\R\) be the quotient map, and choose a lift of \(v\), still denoted by \(v\).  Throughout the proof, replace \(\Sigma_{\mathsf M}\) by \(\pi_E^{-1}(\Sigma_{\mathsf M})\) and, for \(T\subseteq E\), replace \(C_T\) by \(\cone\{e_i:i\in T\}\subseteq\R^E\), retaining the same notation.  Since \(\pi_E(C_T)\) is the original \(C_T\), this does not
  change the intersection condition.

  Suppose first that \(L=\Loop(\operatorname{in}_v\mathsf M)\subseteq S\).  It suffices to prove the assertion when \(L=S\).  Indeed, applying that case with \(S\) replaced by \(L\) gives \((v+C_L)\cap\Sigma_{\mathsf M}\neq\varnothing\), and \(C_L\subseteq C_S\).

  Assume, then, that \(S=\Loop(\operatorname{in}_v\mathsf M)\).  A basis of \(\mathsf M\) is a basis of \(\operatorname{in}_v\mathsf M\) precisely when it has maximum \(v\)-weight; call such a basis \(v\)-maximal.  Every \(i\in E\setminus S\) belongs to a \(v\)-maximal basis, whereas every \(v\)-maximal basis avoids \(S\).

  Choose a \(v\)-maximal basis \(B\) and set \(m=v(B)\).  Choose \(\epsilon\in C_S\) maximizing \(\epsilon(S)\) subject to
  \[
    v(D)+\epsilon(D\cap S)\leq m
    \qquad\text{for every }D\in\mathcal B(\mathsf M).
  \]
  Such a maximum exists: the feasible set contains the origin, and each \(\epsilon_i\) is bounded by the inequality for a basis containing \(i\), which exists because \(\mathsf M\) is loopless.  Since every \(v\)-maximal basis avoids \(S\), it remains \((v+\epsilon)\)-maximal. Moreover, for every \(s\in S\), some \((v+\epsilon)\)-maximal basis contains \(s\); otherwise, by finiteness of \(\mathcal B(\mathsf M)\), a sufficiently small increase of \(\epsilon_s\) would preserve all the inequalities and contradict the maximality of \(\epsilon(S)\).  Thus \(\operatorname{in}_{v+\epsilon}\mathsf M\) is loopless, so \(v+\epsilon\in\Sigma_{\mathsf M}\).

  Conversely, suppose that \(v+\epsilon\in\Sigma_{\mathsf M}\) for some \(\epsilon\in C_S\).  If \(i\in E\setminus S\) were a loop of \(\operatorname{in}_v\mathsf M\), the circuit criterion would give a circuit \(C\) of \(\mathsf M\) such that
  \[
    v_i<v_j\qquad\text{for every }j\in C\setminus\{i\}.
  \]
  Since \(\epsilon_i=0\) and every coordinate of \(\epsilon\) is nonnegative, \(i\) would remain the unique minimum on \(C\) for \(v+\epsilon\).  It would therefore be a loop of \(\operatorname{in}_{v+\epsilon}\mathsf M\), a contradiction.  Hence \(\Loop(\operatorname{in}_v\mathsf M)\subseteq S\).

  Finally, the first equivalence says that \(v\in\Sigma_{\mathsf M,S}\) precisely when \(v\in\Sigma_{\mathsf M}-C_S\), which proves the second formulation.
\end{proof}

\subsection{Grothendieck weights}\label{subsec:grothendieck-weights}

We recall the Grothendieck-weight formalism from \cite[Sections~2 and~5]{wang2026grothendieckweights}.  Let \(\Sigma\) be a complete unimodular fan of dimension \(d\) in \(N_\R\), and set \(X=X_\Sigma\).  For a cone \(\sigma\in\Sigma\), write
\[
  x_\sigma=[\mathcal O_{V(\sigma)}]\in K(X).
\]
In particular, \(x_{\{0\}}=1\).  Since \(X\) is smooth, we identify the Grothendieck ring of vector bundles with the Grothendieck group of coherent sheaves.  The Euler characteristic is denoted by
\[
  \chi:K(X)\longrightarrow\Z.
\]

A \emph{Grothendieck weight} on \(\Sigma\) is a function \(g:\Sigma\to\Z\) that annihilates every linear relation among the classes \(x_\sigma\): if \(\sum_\sigma a_\sigma x_\sigma=0\) in \(K(X)\), then \(\sum_\sigma a_\sigma g(\sigma)=0\).  We write \(\GW(\Sigma)\) for the abelian group of such functions.  The Euler pairing on \(K(X)\) is perfect, and hence
\begin{equation}\label{eq:gw-identification}
  K(X)\xrightarrow{\ \sim\ }\GW(\Sigma),
  \qquad
  \xi\longmapsto g_\xi,
  \qquad
  g_\xi(\sigma)=\chi(\xi x_\sigma).
\end{equation}
We transport the ring structure of \(K(X)\) across this isomorphism.  Thus \(g_{\xi\eta}=g_\xi g_\eta\), and evaluation at the zero cone gives
\[
  g_\xi(\{0\})=\chi(\xi).
\]

We also fix our \(\lambda\)-class notation here.  For a vector bundle \(\mathcal E\),
\[
  \lambda_u(\mathcal E)
  =\sum_{i=0}^{\operatorname{rk}\mathcal E}
  [\wedge^i\mathcal E]u^i.
\]
We use the same notation for virtual classes, with \(\lambda_u(\alpha+\beta)=\lambda_u(\alpha)\lambda_u(\beta)\); in particular,
\(\lambda_u(\alpha-\beta)=\lambda_u(\alpha)/\lambda_u(\beta)\) in \(K(X)[[u]]\).  The same convention applies in equivariant \(K\)-theory.

For the product rule, recall that \(\Sigma\) is \emph{strongly unimodular} if \([N:N_\sigma+N_\tau]\) is either \(1\) or infinite for every \(\sigma,\tau\in\Sigma\).  A vector \(w\in N_\R\) is sufficiently generic if every nonempty intersection \((\sigma+w)\cap\tau\) satisfies \((\sigma^\circ+w)\cap\tau^\circ\neq\varnothing\) and has dimension \(\dim\sigma+\dim\tau-d\).  Such vectors form a dense rational polyhedral open subset of \(N_\R\).

If \(\sigma,\tau\supseteq\gamma\) and \((\sigma+w)\cap\tau\) is nonempty, then its recession cone is \(\sigma\cap\tau\).  Thus the intersection is bounded modulo \((N_\gamma)_\R\) if and only if \(\sigma\cap\tau\subseteq(N_\gamma)_\R\).  Since \(\gamma\) is a face of \(\sigma\cap\tau\), this is equivalent to \(\sigma\cap\tau=\gamma\).

\begin{proposition}[Product rule]\label{prop:gw-product-rule}
  Suppose that \(\Sigma\) is complete and strongly unimodular, and let \(w\in N_\R\) be sufficiently generic.  For \(\xi,\eta\in K(X)\) and \(\gamma\in\Sigma\),
  \[
    g_{\xi\eta}(\gamma)
    =
    \sum_{\substack{\sigma,\tau\supseteq\gamma\\
        (\sigma+w)\cap\tau\neq\varnothing\\
    \sigma\cap\tau=\gamma}}
    (-1)^{\dim\sigma+\dim\tau-d+\dim\gamma}
    g_\xi(\sigma)g_\eta(\tau).
  \]
\end{proposition}

For later use, let \(\mathcal L\) be a nef torus-invariant line bundle on \(X\), and choose a torus-invariant Cartier divisor \(D=\sum_{\rho\in\Sigma(1)}a_\rho D_\rho\) representing \(\mathcal L\).  We use the maximizing convention for its lattice polytope:
\[
  P = \left\{
    m\in M_\R:
    \langle m,u_\rho\rangle\leq a_\rho
    \text{ for every }\rho\in\Sigma(1)
  \right\}.
\]
For \(v\in\tau^\circ\), set
\[
  h_P(v)=\max_{m\in P}\langle m,v\rangle,
  \qquad
  \face_\tau(P)=\{m\in P:\langle m,v\rangle=h_P(v)\}.
\]
This face is independent of the choice of \(v\in\tau^\circ\), and we set \(\face_{\{0\}}(P)=P\).  The toric lattice-point formula and Demazure vanishing give
\begin{equation}\label{eq:gw-nef-line-bundle}
  g_\mathcal L(\tau)=|\face_\tau(P)\cap M|.
\end{equation}
Indeed, the restriction of \(\mathcal L\) to \(V(\tau)\) is nef and has lattice polytope \(\face_\tau(P)\); see \cite[Sections~4.3 and~6.1 and Theorem~9.2.3]{cox2011toric}.

If \(\mathcal L\) is ample, then \(\face_\tau(P)\) has dimension \(d-\dim\tau\).  With \((-)^\circ\) denoting relative interior, Ehrhart reciprocity on \(V(\tau)\) gives
\begin{equation}\label{eq:gw-line-bundle}
  g_{\mathcal L^{-1}}(\tau)
  =(-1)^{d-\dim\tau}|\face_\tau(P)^\circ\cap M|.
\end{equation}

\subsection{Matroid \texorpdfstring{\(K\)}{K}-rings and Euler characteristics}
\label{subsec:matroid-k-ring}

Let \(\mathsf M\) be a loopless matroid on \(E\). Set
\[
  K(\mathsf M)=K(X_{\Sigma_\mathsf M}),
\]
the Grothendieck ring of vector bundles on the smooth toric variety of the Bergman fan. This is the unaugmented \(K\)-ring of \(\mathsf M\) in the terminology of \cite[Section~1.5]{larson2024krings}. Since \(X_{\Sigma_\mathsf M}\) is smooth, this is also the Grothendieck group of coherent sheaves.

For each nonempty proper flat \(F\) of \(\mathsf M\), set
\[
  x_F
  =[\mathcal O_{V_{\Sigma_\mathsf M}(\rho_F)}]
  \in K(\mathsf M).
\]
The presentation of \cite[Theorem~5.2]{larson2024krings}, equivalently the standard toric presentation, says that these classes generate \(K(\mathsf M)\), with relations
\[
  x_Fx_G=0
  \quad\text{if \(F\) and \(G\) are incomparable},
\]
and, for every \(i,j\in E\),
\[
  \prod_{\substack{F\in \mathcal{L}(\mathsf{M})\setminus\{\varnothing,E\}\\i\notin F}}
  (1-x_F)
  =
  \prod_{\substack{F\in \mathcal{L}(\mathsf{M})\setminus\{\varnothing,E\}\\j\notin F}}
  (1-x_F),
\]
where \(\mathcal{L}(\mathsf{M})\) denotes the set of flats of \(\mathsf{M}\).
For a flag \(\mathcal F:F_1\subsetneq\cdots\subsetneq F_\ell\) of nonempty proper flats, write
\[
  x_\mathcal F
  =x_{F_1}\cdots x_{F_\ell},
  \qquad
  x_\varnothing=1.
\]

The variety \(X_{\Sigma_\mathsf M}\) is generally not proper, so its Euler characteristic is not defined by pushforward to a point.  Instead, \cite[Section~1.5]{larson2024krings} constructs a canonical homomorphism
\[
  \chi(\mathsf M,-):K(\mathsf M)\longrightarrow\Z
\]
such that
\[
  \chi(\mathsf M,x_\mathcal F)=1
  \qquad\text{for every flag of flats \(\mathcal F\)}.
\]
The pairing
\[
  K(\mathsf M)\times K(\mathsf M)\longrightarrow\Z,
  \qquad
  (a,b)\longmapsto\chi(\mathsf M,ab),
\]
is perfect. In this paper, we use \(\chi_\mathsf{M}\) to denote \(\chi(\mathsf{M},-)\).

We now pass to the Grothendieck-weight formulation of \cite[Section~7]{wang2026grothendieckweights}.  A Grothendieck weight on \(\mathsf M\) is a function on flags of flats that annihilates every linear relation among the classes \(x_\mathcal F\).  If \(\GW(\mathsf M)\) denotes the group of these weights, the perfect pairing gives an isomorphism
\[
  K(\mathsf M)\xrightarrow{\ \sim\ }\GW(\mathsf M),
  \qquad
  a\longmapsto g_a,
  \qquad
  g_a(\mathcal F)
  =\chi(\mathsf M,a x_\mathcal F).
\]
In particular, \(1\in K(\mathsf M)\) corresponds to the constant weight \(1\) on \(\Sigma_\mathsf M\).

The inclusion of fans \(\Sigma_\mathsf M\subseteq\Sigma_E\) induces an open toric immersion \(X_{\Sigma_\mathsf M}\hookrightarrow X_E\).  Its pullback is the surjective homomorphism
\[
  \pi_\mathsf M:K(X_E)\longrightarrow K(\mathsf M),
  \qquad
  x_S\longmapsto
  \begin{cases}
    x_S, & S\text{ is a flat of \(\mathsf M\)},     \\
    0,   & \text{otherwise},
  \end{cases}
\]
where \(x_S=x_{\rho_S}\) for every nonempty proper subset \(S\subsetneq E\).

For \(\xi\in K(X_E)\), we abuse the notation \(\chi_\mathsf{M}\) and write:
\begin{equation}\label{eq:matroid-euler-characteristic}
  \chi_\mathsf M(\xi)
  =\chi(\mathsf M,\pi_\mathsf M(\xi)).
\end{equation}

The pushforward homomorphism induced by the inclusion \(\Sigma_\mathsf{M}\subseteq \Sigma_E\) is simply the zero extension in the language of Grothendieck weights. For any Grothendieck weight \(g\) on \(\mathsf{M}\), its zero extension is a function on \(K(X_E)\), defined by
\begin{equation}\label{eq:bergman-gw}
  \mathcal G \mapsto
  \begin{cases}
    g(\mathcal{G}), & \mathcal G\text{ is a flag of flats of \(\mathsf M\)}, \\
    0, & \text{otherwise}
  \end{cases}
  \qquad(\sigma_\mathcal G\in\Sigma_E).
\end{equation}
The zero extension of the constant weight on \(\Sigma_\mathsf M\) is the indicator function that records whether a flag of subsets is a flag of flats.

We use \(\Delta_\mathsf M\) also for the corresponding class in \(K(X_E)\) under \eqref{eq:gw-identification}.  The projection formula of \cite[Proposition~5.6]{larson2024krings}, expressed in Grothendieck-weight notation, gives
\begin{equation}\label{eq:matroid-euler-as-product}
  \chi_\mathsf M(\xi)
  =\chi(\Delta_\mathsf M\xi)
  =(\Delta_\mathsf M g_\xi)(\{0\}).
\end{equation}

\subsection{Tropical initial degenerations} \label{subsec:tropical-initial-degeneration}

The geometric input to \cref{sec:ample-k-positivity,sec:k-fold-point-complex-cm} is the tropical initial degeneration of \cite[Section~3]{eur2025vanishingtheoremscombinatorialgeometries}, which we recall here.

Following the paragraph preceding \cite[Definition~3.1]{eur2025vanishingtheoremscombinatorialgeometries}, a vector \(w\) is sufficiently generic if every nonempty intersection \(\sigma\cap(w+\tau)\), with \(\sigma,\tau\in\Sigma_E\), is transverse; in particular \(\sigma\) and \(w+\tau\) are disjoint whenever \(\dim\sigma+\dim\tau<n-1\). This is exactly our definition of sufficiently generic in \cref{subsec:grothendieck-weights}.

For the rest of this subsection, fix a rational vector \(w\in N_\R\) that is sufficiently generic.

\begin{definition}[{\cite[Definition~3.1]
  {eur2025vanishingtheoremscombinatorialgeometries}}]
  \label{def:tropical-initial-degeneration}
  Let \(\mathsf M\) be a matroid of rank \(r\) on \(E\).  If \(\mathsf M\) is loopless, the \emph{tropical initial degeneration} \(\operatorname{ind}_w\mathsf M\subseteq X_E\) is the union of the strata \(V(\sigma)\), over those cones \(\sigma\in\Sigma_E\) of codimension \(r-1\) that meet a maximal cone of \(w+\Sigma_\mathsf M\).  If \(\mathsf M\) has a loop, then \(\operatorname{ind}_w\mathsf M=\varnothing\).
\end{definition}

If \(\mathsf M\) is loopless, then \(\operatorname{ind}_w\mathsf M\) is by construction a reduced union of torus-orbit closures, pure of dimension \(r-1\).  When such a matroid is realized by \(L\subseteq\k^E\), \(\operatorname{ind}_w\mathsf M\) is a Gr\"obner degeneration of the wonderful compactification \(W_L\) associated with \(L\) \cite[Proposition~3.2]{eur2025vanishingtheoremscombinatorialgeometries}.

The following three statements are the only properties of \(\operatorname{ind}_w\mathsf M\) that we use.

\begin{lemma}[{\cite[Lemma~3.5] {eur2025vanishingtheoremscombinatorialgeometries}}] \label{lem:efl-orbit-criterion}
  For a loopless \(\mathsf M\) and any cone \(\sigma\in\Sigma_E\),
  \[
    V(\sigma)\subseteq\operatorname{ind}_w\mathsf M
    \quad\Longleftrightarrow\quad
    \sigma\cap(w+\Sigma_\mathsf M)\neq\varnothing.
  \]
\end{lemma}

\begin{proposition}[{\cite[Proposition~3.6] {eur2025vanishingtheoremscombinatorialgeometries}}]
  \label{prop:efl-euler-characteristic}
  For every \(\xi\in K(X_E)\) and every loopless matroid \(\mathsf M\),
  \[
    \chi_\mathsf M(\xi)
    =\chi\bigl(\operatorname{ind}_w\mathsf M,\xi\bigr),
  \]
  where the left-hand side is the Euler characteristic \eqref{eq:matroid-euler-characteristic} of the matroid \(K\)-ring.
\end{proposition}

\begin{theorem}[{\cite[Corollaries~3.8 and~3.11] {eur2025vanishingtheoremscombinatorialgeometries}}]
  \label{thm:efl-cm}
  Let \(\mathsf M\) be a loopless matroid on \(E\).
  \begin{enumerate}
    \item The scheme \(\operatorname{ind}_w\mathsf M\) is Cohen--Macaulay,
      geometrically connected, and geometrically reduced.
    \item Let \(P\subseteq M_\R\) be a generalized permutohedron, so its normal fan coarsens \(\Sigma_E\).  Let \(X_P\) be the toric variety of the normal fan of \(P\), and let \(f_P:X_E\to X_P\) be the induced toric morphism.  Denote the scheme-theoretic image of \(\operatorname{ind}_w\mathsf M\) by \(f_P(\operatorname{ind}_w\mathsf M)\).  Then the natural map
      \[
        \mathcal O_{f_P(\operatorname{ind}_w\mathsf M)}
        \longrightarrow
        Rf_{P*}\mathcal O_{\operatorname{ind}_w\mathsf M}
      \]
      is an isomorphism, and \(f_P(\operatorname{ind}_w\mathsf M)\) is Cohen--Macaulay.
  \end{enumerate}
\end{theorem}

\subsection{Cohen--Macaulay simplicial complexes}
\label{subsec:cm-simplicial-complexes}

Let \(\Delta\) be a finite simplicial complex on an ambient vertex set \(V\), possibly with ghost vertices (i.e. vertices that are not in any face).  Its Stanley--Reisner ring is
\[
  \k[\Delta]
  =\k[x_v:v\in V]/I_\Delta,
  \qquad
  I_\Delta=(x_F:F\subseteq V,\ F\notin\Delta),
\]
where \(x_F=\prod_{v\in F}x_v\). In particular, \(x_v\in I_\Delta\) if \(v\) is a ghost vertex.

\begin{definition}[{\cite[Definition~1.12 and Theorem~1.13]{miller2005combinatorial}}]\label{def:simplicial-k-polynomial}
  Give \(\k[\Delta]\) the fine \(\Z^V\)-grading \(\deg x_v=e_v\).  Its \emph{finely graded \(K\)-polynomial} \(\mathcal K_V(\Delta;\bm z)\) is defined by
  \[
    \operatorname{Hilb}(\k[\Delta];\bm z)
    =
    \frac{\mathcal K_V(\Delta;\bm z)}
    {\prod_{v\in V}(1-z_v)}.
  \]
  Equivalently,
  \[
    \mathcal K_V(\Delta;\bm z)
    =
    \sum_{F\in\Delta}
    \prod_{v\in F}z_v
    \prod_{v\in V\setminus F}(1-z_v).
  \]
  More generally, for a map \(\pi:V\to[k]\), the \(\Z^k\)-graded \(K\)-polynomial with \(\deg x_v=e_{\pi(v)}\) is
  \[
    \mathcal K(\Delta;Z_1,\ldots,Z_k)
    =
    \left.
    \mathcal K_V(\Delta;\bm z)
    \right|_{z_v=Z_{\pi(v)}}.
  \]
\end{definition}

\begin{definition}\label{def:cm-simplicial-complex}
  The simplicial complex \(\Delta\) is \emph{Cohen--Macaulay over \(\k\)} if \(\k[\Delta]\) is Cohen--Macaulay.  For a face \(F\in\Delta\), its \emph{link} is
  \[
    \link_\Delta(F)
    =\{G\in\Delta:G\cap F=\varnothing,\ F\cup G\in\Delta\}.
  \]
  In particular, \(\link_\Delta(\varnothing)=\Delta\).
\end{definition}

\begin{theorem}[Reisner's criterion]\label{thm:reisner-criterion}
  The complex \(\Delta\) is Cohen--Macaulay over \(\k\) if and only if
  \[
    \widetilde H_i\bigl(\link_\Delta(F);\k\bigr)=0
    \qquad
    \text{for every \(F\in\Delta\) and every
    \(i<\dim\link_\Delta(F)\)}.
  \]
\end{theorem}

This is \cite[Theorem~1]{reisner1976cohen}; see also \cite[(11.5)]{bjorner1995topological}.  In particular, a Cohen--Macaulay simplicial complex is pure, every one of its links is Cohen--Macaulay, and its reduced homology vanishes below its dimension.  For a finite polytopal complex, we use the analogous terminology: it is Cohen--Macaulay if it is pure and the same face-link homology vanishing holds \cite[Section~5.1]{rowlands2024topology}.

\section{Equivariant Grothendieck weights}\label{sec:equivariant-gw}
In this section, we record the equivariant version of the Grothendieck weight formalism in \cite{wang2026grothendieckweights}.  The main new feature is that the balancing relations remember characters of the acting torus.

\subsection{Equivariant conventions}\label{subsec:equivariant-conventions}

Write \(\mathbb G_m=\operatorname{Spec}\k[t^{\pm1}]\). For the equivariant theory on the permutohedral variety, let \(H=\mathbb G_m^E\) be the homogeneous coordinate torus.  Its action on \(X_E\) factors through the dense torus \(H/\mathbb G_m\), where \(\mathbb G_m\subseteq H\) is the diagonal subtorus.  We retain the full torus \(H\) and write
\[
  R(H)=\Z[t_i^{\pm1}:i\in E].
\]

We follow the conventions of \cite{BergetEurSpinkTseng2023,berget2025externalactivitycomplexpair}.  For \(t=(t_i)_{i\in E}\in H\) and \(v=(v_i)_{i\in E}\in\k^E\), the action on \(\k^E\) is
\[
  t\cdot v=(t_i^{-1}v_i)_{i\in E}.
\]
We use the induced action on \(\operatorname{Gr}(r;E)\), and the trivial bundle \(\underline{\k^E}=X_E\times\k^E\) is linearized by
\[
  t\cdot(x,v)=\bigl(t\cdot x,(t_i^{-1}v_i)_{i\in E}\bigr).
\]
Thus the dual tautological classes have local weights \(t_i\).  For \(B\subseteq E\), write \(\bm t_B=\prod_{i\in B}t_i\).
\subsection{Equivariant balancing}

Let \(N\) be a lattice, let \(\Sigma\) be a complete strongly unimodular fan in \(N_\R\), let \(d=\dim \Sigma\), and let \(X=X_\Sigma\).  Here strongly unimodular means that \(\Sigma\) is unimodular and that, for every pair of cones \(\sigma,\tau\), the index \([N:N_\sigma+N_\tau]\) is either \(1\) or infinite \cite[Definition~3.1]{wang2026grothendieckweights}.  We write \(T\) for the dense torus, \(M=\Hom(N,\Z)\), and
\[
  R(T)=\Z[M].
\]
For \(m\in M\), the corresponding character is denoted by \(\chi^m\).  For a cone \(\sigma\in\Sigma\), set
\[
  x_\sigma=[\mathcal{O}_{V(\sigma)}]\in K_T(X).
\]
Since \(X\) is smooth, we identify equivariant \(K\)-theory and equivariant \(G\)-theory.  We denote the equivariant Euler characteristic, or equivalently the pushforward to a point, by
\[
  \chi_T:K_T(X)\to R(T).
\]

\begin{definition}
  An \emph{equivariant Grothendieck weight} on \(\Sigma\) is a function \(g:\Sigma\to R(T)\) such that for every relation
  \[
    \sum_{\sigma\in\Sigma}a_\sigma x_\sigma=0
    \qquad \text{in }K_T(X),
  \]
  with \(a_\sigma\in R(T)\), one has
  \[
    \sum_{\sigma\in\Sigma}a_\sigma g(\sigma)=0.
  \]
  We denote the \(R(T)\)-module of equivariant Grothendieck weights by \(\GW_T(\Sigma)\).
\end{definition}

As in the ordinary case, we show below that \(\GW_T(\Sigma)\) is isomorphic to \(K_T(X)\) as an \(R(T)\)-module. The key input is the following equivariant Kronecker duality, proved in \cite[Theorem~6.1]{anderson2015operational}.

\begin{lemma}\label{lem:perfectness-euler-pairing}
  For a complete unimodular fan \(\Sigma\), the Euler pairing defined by
  \[
    (a,b)\mapsto \chi_T(ab)
  \]
  is a perfect pairing.
\end{lemma}
\begin{proof}
  Every toric variety is \(T\)-linear, so \cite[Theorem~6.1]{anderson2015operational} applies to the complete toric variety \(X\) and gives an isomorphism
  \[
    \operatorname{op}K_T^\circ(X)
    \longrightarrow \Hom_{R(T)}(K_{\circ}^T(X),R(T)),
    \qquad
    \beta\longmapsto\bigl(\alpha\mapsto\chi_T(\alpha\beta)\bigr).
  \]
  Since \(X\) is smooth, we may identify both \(K_\circ^T(X)\) and \(\operatorname{op}K_T^\circ(X)\) with \(K_T(X)\). Under these identifications, the displayed isomorphism is exactly the perfectness of the Euler pairing.
\end{proof}

\begin{proposition}
  \(\GW_T(\Sigma)\) is isomorphic to \(K_T(X)\) as an \(R(T)\)-module.
\end{proposition}
\begin{proof}
  For an arbitrary class \(\alpha\in K_T(X)\), we assign a Grothendieck weight \(g_\alpha^T\) as follows:
  \[
    g_\alpha^T(\sigma)=\chi_T(\alpha x_\sigma).
  \]
  The function \(g_\alpha^T\) is clearly an equivariant Grothendieck weight, and this assignment is \(R(T)\)-linear. We claim this assignment defines an \(R(T)\)-isomorphism.

  The classes \(x_\sigma\) generate \(K_T(X)\) over \(R(T)\). This is because \(X\) has a stratification defined by torus orbits
  \[
    X_{\leq j} \coloneqq \bigcup_{\dim O(\sigma)\leq j} O(\sigma),
  \]
  and the localization sequence for the equivariant \(K\)-ring applied successively to \((X_{\leq j},X_{\leq j-1})\) proves the result. See, for example, \cite[Theorem~8]{Merkurjev2005}.

  Therefore, the perfectness of the Euler pairing proves the result.
\end{proof}

Vezzosi--Vistoli give a multiplicative Stanley--Reisner presentation in \cite[Theorem~6.4]{vezzosi2003higher}.

\begin{theorem}[{\cite[Theorem~6.4]{vezzosi2003higher}}]\label{thm:VV-SR-presentation}
  For a complete smooth toric variety,
  \[
    K_T(X)\simeq
    \frac{\Z[(1-x_\rho)^{\pm1}:\rho\in\Sigma(1)]}{\mathcal{I}_{\mathrm{SR}}},
  \]
  where \(\mathcal{I}_{\mathrm{SR}}\) is generated by the relations
  \[
    \prod_{\rho\in S}x_\rho=0
    \quad\text{if \(S\) is not contained in a cone}.
  \]
  The \(R(T)\)-algebra structure of \(K_T(X)\) is determined by
  \[
    \chi^m=\prod_{\rho\in\Sigma(1)}
    (1-x_\rho)^{-\langle m,u_\rho\rangle}
    \qquad(m\in M).
  \]
\end{theorem}
\begin{remark}
  For a cone \(\sigma\), transversality of the invariant divisors gives \(x_\sigma=\prod_{\rho\in\sigma(1)}x_\rho\).
\end{remark}

\iffalse
\begin{lemma}\label{lemma:localized-fixed-point-basis}
  Let \(F=\operatorname{Frac}R(T)\).  The classes
  \(x_\sigma\) for maximal cones \(\sigma\) form an \(F\)-basis of
  \(K_T(X)\otimes_{R(T)}F\).
\end{lemma}

\begin{proof}
  For a class \(\alpha\in K_T(X)\), the Euler pairing with the point class \(x_\sigma\) is the same as the restriction to a point, by the projection formula:
  \[
    \chi(\alpha x_\sigma)=\alpha|_{V(\sigma)}.
  \]

  Since the Euler pairing is perfect by \cref{lem:perfectness-euler-pairing}, the \(x_\sigma\) form an \(F\)-basis if and only if the maps \(\alpha\mapsto \alpha|_{V(\sigma)}\) form a basis of \(\mathrm{Hom}_{R(T)}(K_T(X),F)=\mathrm{Hom}_{F}(K_T(X)\otimes_{R(T)}F,F)\), and the latter is a restatement of the equivariant localization theorem.
\end{proof}
\fi

We now give the equivariant balancing conditions.  Define
\[
  Q_\Sigma=\{q\in M_\R:-1\leq \langle q,u_\rho\rangle\leq 1
  \text{ for all }\rho\in\Sigma(1)\}.
\]
For \(q\in Q_\Sigma\cap M\), set
\[
  \bm P_q=\{\rho\in\Sigma(1):\langle q,u_\rho\rangle=1\},
  \qquad
  \bm N_q=\{\rho\in\Sigma(1):\langle q,u_\rho\rangle=-1\}.
\]
Recall that \(M(\tau)=\tau^\perp\cap M\) is the annihilator lattice of \(\tau\).
If \(q\in M(\tau)\), we call \(\tau\) a \(q\)-neutral cone.

\begin{theorem}[Equivariant \(K\)-balancing]\label{thm:equivariant-k-balancing}
  A function \(g:\Sigma\to R(T)\) is an equivariant Grothendieck weight if and only if, for every \(q\in Q_\Sigma\cap M\) and every \(q\)-neutral cone \(\tau\), one has
  \[
    \sum_{\substack{\sigma\supseteq\tau\\
    \sigma(1)\setminus\tau(1)\subseteq\bm P_q}}
    (-1)^{\dim\sigma}g(\sigma)
    =
    \chi^{-q}
    \sum_{\substack{\sigma\supseteq\tau\\
    \sigma(1)\setminus\tau(1)\subseteq\bm N_q}}
    (-1)^{\dim\sigma}g(\sigma).
  \]
  Equivalently,
  \[
    (1-\chi^{-q})(-1)^{\dim\tau}g(\tau)
    +\sum_{\substack{\sigma\supsetneq\tau\\
    \sigma(1)\setminus\tau(1)\subseteq\bm P_q}}
    (-1)^{\dim\sigma}g(\sigma)
    -\chi^{-q}
    \sum_{\substack{\sigma\supsetneq\tau\\
    \sigma(1)\setminus\tau(1)\subseteq\bm N_q}}
    (-1)^{\dim\sigma}g(\sigma)=0.
  \]
\end{theorem}

\begin{proof}
  Let
  \[
    \Psi_T:R(T)^\Sigma\longrightarrow K_T(X),
    \qquad e_\sigma\longmapsto x_\sigma.
  \]
  Fix \(q\in Q_\Sigma\cap M\) and a \(q\)-neutral cone \(\tau\).  Since \(q\) is integral and belongs to \(Q_\Sigma\), every pairing \(\langle q,u_\rho\rangle\) lies in \(\{0,\pm1\}\).  Moreover, \(q\in M(\tau)\), so no ray of \(\tau\) lies in \(\bm P_q\cup\bm N_q\).  The character relation in the Stanley--Reisner presentation \cref{thm:VV-SR-presentation}, multiplied by \(x_\tau\), is
  \[
    x_\tau\left(
      \prod_{\rho\in\bm P_q}(1-x_\rho)
      -\chi^{-q}\prod_{\rho\in\bm N_q}(1-x_\rho)
    \right)=0.
  \]
  Expanding the products and discarding the monomials whose rays do not span a cone gives
  \[
    \sum_{\substack{\sigma\supseteq\tau\\
    \sigma(1)\setminus\tau(1)\subseteq\bm P_q}}
    (-1)^{\dim\sigma-\dim\tau}x_\sigma
    -
    \chi^{-q}
    \sum_{\substack{\sigma\supseteq\tau\\
    \sigma(1)\setminus\tau(1)\subseteq\bm N_q}}
    (-1)^{\dim\sigma-\dim\tau}x_\sigma=0.
  \]
  After multiplying by \((-1)^{\dim\tau}\), we obtain the relation described in the statement. This proves the ``only if'' direction.

  To prove the ``if'' direction, let \(b^T_{q,\tau}\in R(T)^\Sigma\) be the coefficient vector of the relation in the statement.  Thus \(b^T_{q,\tau}\in\ker\Psi_T\).

  We now compare these relations with the ordinary ones.  In the notation of \cite[Theorem~4.2]{wang2026grothendieckweights}, set
  \[
    \Psi:\Q^\Sigma\longrightarrow K(X)\otimes_\Z\Q,
    \qquad e_\sigma\longmapsto x_\sigma,
  \]
  and, for every admissible pair \((q,\tau)\), let
  \[
    b_{q,\tau}
    =
    \sum_{\substack{\sigma\supsetneq\tau\\
    \sigma(1)\setminus\tau(1)\subseteq\bm P_q}}
    (-1)^{\dim\sigma}e_\sigma
    -
    \sum_{\substack{\sigma\supsetneq\tau\\
    \sigma(1)\setminus\tau(1)\subseteq\bm N_q}}
    (-1)^{\dim\sigma}e_\sigma.
  \]
  The proof of the cited theorem also shows that the vectors \(b_{q,\tau}\) span \(\ker\Psi\) over \(\Q\).  Under the specialization map \(R(T)\to\Z\), the vector \(b^T_{q,\tau}\) specializes exactly to \(b_{q,\tau}\).

  Following the notation in \cite[Theorem~4.2]{wang2026grothendieckweights}, we set \(M_{\GW}^T\) and \(M_{\GW}\) to be the matrices whose row vectors are \(b_{q,\tau}^T\) and \(b_{q,\tau}\), respectively. It is proved in the same reference that
  \[
    \rk M_{\GW} = |\Sigma|-\dim_{\Q} K(X)\otimes_\Z \Q.
  \]

  Since \(M_{\GW}^T\) specializes to \(M_{\GW}\), we claim that \(M_{\GW}^T\) has rank at least the same number over \(F=\operatorname{Frac}R(T)\). Take any minor with nonzero determinant that witnesses the rank of \(M_{\GW}\); then the same minor also witnesses the rank of \(M_{\GW}^T\) over \(F\).

  On the other hand, by \cite[Theorem~6.9 and Corollary~6.10]{vezzosi2003higher}, \(K_T(X)\) is a projective \(R(T)\)-module of rank equal to the number of maximal cones of \(\Sigma\), and \(K(X)\) is a free abelian group of the same rank. Therefore
  \[
    \begin{aligned}
      \dim_F \ker (\Psi_T\otimes_{R(T)} F) & = |\Sigma| - \dim_F (K_T(X)\otimes_{R(T)}F) \\
      & =|\Sigma| - \dim_\Q (K(X)\otimes_{\Z}\Q).    \\
    \end{aligned}
  \]
  Thus \(M_{\GW}^T\) has rank at most the same number. Together, these bounds show that \(M_{\GW}^T\) has rank exactly \(|\Sigma| - \dim_\Q (K(X)\otimes_{\Z}\Q)\) over \(F\).
  Thus the vectors \(b_{q,\tau}^T\) span \(\ker \Psi_T\) over \(F\).

  If \(g\) satisfies the equations in the theorem, its \(F\)-linear extension annihilates \(\ker(\Psi_T\otimes_{R(T)} F)\).  If \(a=(a_\sigma)\in\ker\Psi_T\), then \(\sum_\sigma a_\sigma g(\sigma)=0\) in \(F\).  This sum belongs to \(R(T)\), and \(R(T)\) injects into \(F\), so it is already zero in \(R(T)\).  Thus \(g\) is an equivariant Grothendieck weight. This proves the ``if'' direction.
\end{proof}

\begin{remark}
  Under the augmentation \(R(T)\to\Z\), \(\chi^{-q}\mapsto 1\), the term \((1-\chi^{-q})g(\tau)\) disappears and \cref{thm:equivariant-k-balancing} becomes the ordinary \(K\)-balancing condition of \cite[Theorem~4.2]{wang2026grothendieckweights}.
\end{remark}

We specialize this to the permutohedral fan using the notation fixed in \cref{subsec:permutohedral-fan,subsec:equivariant-conventions}. We will pass to the \(H\)-equivariant version when needed below. Note that the surjection \(H\to T\) induces a free extension map
\[
  R(T)\to R(H),\qquad \chi^{e_i-e_j}\mapsto \frac{t_i}{t_j},
\]
where we write \(R(H)=\mathbb{Z}[t_i^{\pm 1}, i\in E]\). Using this map,
\[
  K_H(X_E)=K_T(X_E)\otimes_{R(T)}R(H).
\]
The Euler pairing, Grothendieck weight formalism, \(K\)-balancing and product rules below are all compatible with this base change.

For \(i\neq j\), a flag \(\mathcal G\) is called \(\{i,j\}\)-neutral if every \(G_a\) either contains both \(i\) and \(j\) or contains neither of them.  Let \(\bm S_{ij}(\mathcal G)\) be the set of strict refinements \(\mathcal F\succneqq\mathcal G\) such that every new set \(B\in\mathcal F\setminus\mathcal G\) contains \(i\) and does not contain \(j\).

\begin{corollary}\label{cor:equivariant-permutohedral-balancing}
  A function \(g:\Sigma_E\to R(H)\) is an \(H\)-equivariant Grothendieck weight if and only if, for every \(i\neq j\) and every \(\{i,j\}\)-neutral flag \(\mathcal G\), one has
  \[
    \left(1-\frac{t_j}{t_i}\right)(-1)^{\ell(\mathcal G)}g(\mathcal G)
    +\sum_{\mathcal F\in\bm S_{ij}(\mathcal G)}
    (-1)^{\ell(\mathcal F)}g(\mathcal F)
    -\frac{t_j}{t_i}
    \sum_{\mathcal F\in\bm S_{ji}(\mathcal G)}
    (-1)^{\ell(\mathcal F)}g(\mathcal F)=0.
  \]
\end{corollary}

\begin{proof}
  Apply \cref{thm:equivariant-k-balancing} to \(q=e_i-e_j\).  For a ray \(\rho_S\) of \(\Sigma_E\),
  \[
    \langle e_i-e_j,u_{\rho_S}\rangle
    =
    \begin{cases}
      1,  & i\in S,\ j\notin S, \\
      -1, & i\notin S,\ j\in S, \\
      0,  & \text{otherwise}.
    \end{cases}
  \]
  This translates \(q\)-neutrality and the two sets \(\bm P_q,\bm N_q\) into the stated flag language, while \(\chi^{-q}\) maps to \(t_j/t_i\).  By \cite[Remark~4.3 and Section~6]{wang2026grothendieckweights}, the elements \(e_i-e_j\) give enough \(q\)'s for the permutohedral fan.
\end{proof}

\subsection{The equivariant product rule}

The product rule also has an equivariant form.  The diagonal \(\delta:X\to X\times X\) is equivariant for the diagonal \(T\)-action on \(X\times X\).  The same diagonal degeneration used in \cite[Section~5]{wang2026grothendieckweights} is diagonal-\(T\)-invariant, so the equality of structure sheaf classes holds in \(K_T(X\times X)\) with the same coefficients as in the ordinary case.

\begin{theorem}[Equivariant product rule]\label{thm:equivariant-product-rule}
  Let \(\Sigma\) be a complete strongly unimodular fan of dimension \(d\), and let \(v\in N_\R\) be sufficiently generic, as in \cref{subsec:grothendieck-weights}. For any two \(T\)-equivariant Grothendieck weights \(g_1,g_2\),
  \[
    (g_1g_2)(\gamma)=
    \sum_{\substack{\sigma,\tau\supseteq\gamma\\
        (\sigma+v)\cap\tau\neq\varnothing\\
    \sigma\cap\tau=\gamma}}
    (-1)^{\dim\sigma+\dim\tau-d+\dim\gamma}
    g_1(\sigma)g_2(\tau).
  \]
\end{theorem}

\begin{proof}
  The diagonal degeneration initially requires an integral translation vector. Since the condition of being sufficiently generic in \cref{subsec:grothendieck-weights} is equivalent to being in a maximal-dimensional chamber of a polyhedral complex, we can choose \(w\in N\) so that replacing \(v\) by \(w\) does not affect the summation.
  %The finitely many rational polyhedral cones \(\tau-\sigma\), for \(\sigma,\tau\in\Sigma\), cut the generic locus into rational polyhedral chambers.
  %Choose \(w\in N\) in the same chamber as
  %\(v\).  For every \(\sigma,\tau\), \[ (\sigma+v)\cap\tau\neq\varnothing \quad\Longleftrightarrow\quad v\in\tau-\sigma \quad\Longleftrightarrow\quad w\in\tau-\sigma \quad\Longleftrightarrow\quad (\sigma+w)\cap\tau\neq\varnothing. \] The condition \(\sigma\cap\tau=\gamma\) is independent of the translation vector.  Thus the summation determined by \(v\) is the same as the one determined by \(w\).

  The one-parameter subgroup associated with \(w\) gives
  \[
    \delta_*x_\gamma=
    \sum_{\substack{\sigma,\tau\supseteq\gamma\\
        (\sigma+w)\cap\tau\neq\varnothing\\
    \sigma\cap\tau=\gamma}}
    (-1)^{\dim\sigma+\dim\tau-d+\dim\gamma}
    x_\sigma\boxtimes x_\tau
  \]
  in \(K_T(X\times X)\).  Multiply both sides by \(p_1^*\alpha_1\cdot p_2^*\alpha_2\) and apply \(\chi_T\).  The projection formula identifies the left-hand side with
  \[
    \chi_T(\alpha_1\alpha_2 x_\gamma)=(g_1g_2)(\gamma).
  \]
  For each term on the right, the equivariant K\"unneth formula for Euler characteristics gives
  \[
    \chi_T(p_1^*(\alpha_1x_\sigma)\cdot p_2^*(\alpha_2x_\tau))
    =
    \chi_T(\alpha_1x_\sigma)\chi_T(\alpha_2x_\tau)
    =
    g_1(\sigma)g_2(\tau).
  \]
  This proves the formula.
\end{proof}

\subsection{Higher product rules on the braid fan}

We extend the diagonal degeneration to an arbitrary number of factors on the braid fan.  The proof, including the multi-cone lattice-index calculation needed beyond \cite[Section~5]{wang2026grothendieckweights}, is given in \cref{sec:proof-higher-product-rule}.

Fix an integer \(k\geq2\), and set \(d=|E|-1\).  Let
\[
  \bm\sigma=(\sigma_1,\ldots,\sigma_k),
  \qquad
  \bm v=(v_1,\ldots,v_k)\in (N_\R)^k.
\]
For \(\sigma_1,\ldots,\sigma_k\in\Sigma_E\), set
\[
  C(\bm\sigma;\bm v)
  =
  \bigcap_{i=1}^k(\sigma_i+v_i)
  \subseteq N_\R.
\]
We call \(\bm v\) \emph{jointly generic} if every such nonempty intersection meets the relative interiors of all \(k\) translated cones and has the expected dimension
\begin{equation}\label{eq:multi-expected-dimension}
  \dim C(\bm\sigma;\bm v) = \sum_{i=1}^k\dim\sigma_i-(k-1)d.
\end{equation}

\begin{theorem}[Equivariant \(k\)-fold product rule]
  \label{thm:equivariant-multi-product-rule}
  Let \(g_1,\ldots,g_k\) be equivariant Grothendieck weights on the braid fan \(\Sigma_E\).  For a jointly generic tuple \(\bm v\in(N_\R)^k\), their product satisfies, for every \(\gamma\in\Sigma_E\),
  \[
    \left(\prod_{i=1}^k g_i\right)(\gamma)
    =
    \sum_{\substack{
        \sigma_i\supseteq\gamma\text{ for all }i\\
        C(\bm\sigma;\bm v)\neq\varnothing\\
    \cap_i \sigma_i = \gamma\\ }}
    (-1)^{\dim C(\bm\sigma;\bm v)+\dim\gamma}
    \prod_{i=1}^kg_i(\sigma_i).
  \]
\end{theorem}

\subsection{Equivariant tautological weights}

Finally, we compute the equivariant tautological weights used below, in the notation of \cref{subsec:initial-matroids,subsec:equivariant-conventions}.

In \cite[Section~3]{fink2012kclasses}, Fink--Speyer define \(y(M)\) equivariantly for the diagonal torus action on \(G(d,n)\).  Their fixed-point formula is independent of the ground field.  Under the identifications \([n]=E\) and \(d=r\), this is the action of our homogeneous coordinate torus \(H\). We denote the resulting class in \(K_H(\operatorname{Gr}(r;E))\) by \(y_H(\mathsf M)\), to emphasize the role of \(H\).

Let \(\mathsf M\) be a rank-\(r\) matroid on \(E\), and let \(\mathcal E\in K_H(\operatorname{Gr}(r;E))\).  For a maximal flag \(\mathcal F\), let \(p_{\mathcal F}\in X_E\) be the corresponding fixed point and let \(B_{\mathcal F}(\mathsf M)\) be the unique basis of \(\operatorname{in}_{\mathcal F}\mathsf M\).  We denote by \(\mathcal E_{\mathsf M}\in K_H(X_E)\) the class characterized by
\[
  \mathcal E_{\mathsf M}|_{p_{\mathcal F}}
  =
  \mathcal E|_{p_{B_{\mathcal F}(\mathsf M)}},
\]
where \(p_B\in\operatorname{Gr}(r;E)\) is the coordinate fixed point indexed by \(B\).  This is \cite[Proposition~3.13]{BergetEurSpinkTseng2023} in our inverse-action convention. The same proposition states that the map \(\mathcal{E}\mapsto \mathcal{E}_\mathsf{M}\) is a ring homomorphism.

\begin{lemma}\label{lemma:equivariant-grassmannian-comparison}
  Let \(\mathsf M\) be a matroid of rank \(r\) on
  \(E\), and let
  \(\mathcal E\in K_H(\operatorname{Gr}(r;E))\).  Then
  \[
    \chi_H(X_E,\mathcal E_{\mathsf M})
    =
    \chi_H(\operatorname{Gr}(r;E),y_H(\mathsf M)\mathcal E).
  \]
\end{lemma}

\begin{proof}
  In this proof, we work over \(\mathbb{C}\).  This is because both \(K_H(X_E)\) and \(K_H(\operatorname{Gr}(r;E))\) have field-independent descriptions.

  Suppose that \(\mathsf M\) is realized by \(L\subseteq\mathbb C^E\), and set \(Y_L=\overline{H\cdot L}\subseteq\operatorname{Gr}(r;E)\), with the inverse action.  Define the \(H\)-equivariant toric morphism
  \[
    f_L:X_E\longrightarrow Y_L,
    \qquad
    \bar t\longmapsto t^{-1}L.
  \]
  In the notation of \cite[Proposition~3.13]{BergetEurSpinkTseng2023}, this is \(f_L=\varphi_L\circ\operatorname{crem}\).  The Grassmannian convention of \cite{fink2012kclasses} and the fiber convention for the tautological classes on \(X_E\) in \cite{BergetEurSpinkTseng2023} both use inverse coordinate scaling.  Hence the cited proposition gives \(\mathcal E_{\mathsf M}=f_L^*j_L^*\mathcal E\), where \(j_L:Y_L\hookrightarrow\operatorname{Gr}(r;E)\) is the inclusion. Since \(f_L\) is obtained from the corresponding coarsening of fans, \((f_L)_*[\mathcal O_{X_E}]=[\mathcal O_{Y_L}]\) in \(K_H(Y_L)\); see \cite[Theorem~9.2.5]{cox2011toric}.  Moreover, the defining property of the Fink--Speyer class gives \(y_H(\mathsf M)=(j_L)_*[\mathcal O_{Y_L}]\).  The equivariant projection formula now gives
  \[
    \chi_H(X_E,\mathcal E_{\mathsf M})
    =
    \chi_H\bigl(\operatorname{Gr}(r;E),
    (j_L)_*[\mathcal O_{Y_L}]\,\mathcal E\bigr)
    =
    \chi_H\bigl(\operatorname{Gr}(r;E),
    y_H(\mathsf M)\mathcal E\bigr).
  \]

  For arbitrary \(\mathsf M\), the right-hand side is valuative by the definition of \(y_H(\mathsf M)\).  \cite[Proposition~5.8]{BergetEurSpinkTseng2023} states that, for fixed \([\mathcal E]\), the assignment \(\mathsf M\mapsto[\mathcal E_{\mathsf M}]\) is valuative.  Applying \(\chi_H\) shows that the left-hand side is valuative as well.  Finally, \cite[Lemma~5.9]{BergetEurSpinkTseng2023} states that the indicator function of every matroid base polytope is an integral linear combination of those of \(\mathbb C\)-realizable matroids.  The equality for realizable matroids therefore implies the equality for all matroids.
\end{proof}

For a matroid \(\mathsf N\) on \(A\subseteq E\), recall the homogeneous multivariate Tutte polynomial from \cite{branden2020lorentzian}.  For
\(\bm w=(w_i)_{i\in A}\), it is
\[
  Z_{\mathsf N}(q,w_0,\bm w)
  =
  \sum_{B\subseteq A}
  q^{-\rk_{\mathsf N}(B)}
  w_0^{|A|-|B|}
  \prod_{i\in B}w_i.
\]
Define
\[
  T_{\mathsf N}^H(u,v;\bm t)
  =
  Z_{\mathsf N}\bigl(v/u,1,(v t_i)_{i\in A}\bigr)
  =
  \sum_{B\subseteq A}
  u^{\rk_{\mathsf N}(B)}
  v^{|B|-\rk_{\mathsf N}(B)}
  \bm t_B.
\]

The following proposition is essentially a restatement of \cite[Theorem~5.2]{fink2012kclasses}.
\begin{proposition}\label{prop:equivariant-tautological-gw}
  Let \(\mathsf M\) be a matroid on \(E\).  The \(H\)-equivariant Grothendieck weight of
  \[
    \lambda_u(\mathcal S_{\mathsf M}^{\vee})
    \lambda_v(\mathcal Q_{\mathsf M}^{\vee})\in K_H(X_E)[u,v]
  \]
  is
  \[
    \mathcal F\longmapsto
    T_{\operatorname{in}_{\mathcal F}\mathsf M}^H(u,v;\bm t).
  \]
\end{proposition}

\begin{proof}
  First assume that \(\mathcal F\) is empty, and set \(r=\rk\mathsf M\).  Let \(\mathcal S,\mathcal Q\) be the tautological bundles on \(\operatorname{Gr}(r;E)\), and set \(\mathcal E_{u,v}=\lambda_u(\mathcal S^\vee)\lambda_v(\mathcal Q^\vee)\). By the inverse-action translation of \cite[Proposition~3.13]{BergetEurSpinkTseng2023} used above, the matroid analogue construction sends \(\mathcal E_{u,v}\) to \(\lambda_u(\mathcal S_{\mathsf M}^{\vee}) \lambda_v(\mathcal Q_{\mathsf M}^{\vee})\).  By \cref{lemma:equivariant-grassmannian-comparison},
  \[
    \chi_H\bigl(X_E,
      \lambda_u(\mathcal S_{\mathsf M}^{\vee})
    \lambda_v(\mathcal Q_{\mathsf M}^{\vee})\bigr)
    =
    \chi_H\bigl(\operatorname{Gr}(r;E),y_H(\mathsf M)\mathcal E_{u,v}\bigr).
  \]

  We now apply Fink--Speyer's equivariant rank-generating identity.  They use the same inverse coordinate action as we do.  In their notation, a rank \(d\) matroid \(M\) has ground set \([n]\) and rank function \(\rho_M\), and \cite[Theorem~5.2]{fink2012kclasses} states
  \[
    \int^T\sum_{p=0}^{d}\sum_{q=0}^{n-d}
    y(M)[\mathcal O(1)]
    [\textstyle\bigwedge^p\mathcal S]
    [\textstyle\bigwedge^q\mathcal Q^\vee]u^pv^q
    =
    \sum_{S\subseteq[n]}
    t^{e_S}u^{d-\rho_M(S)}v^{|S|-\rho_M(S)}.
  \]
  Their equivariant pushforward \(\int^T\) is our \(\chi_H\), their class \(y(M)\) is our \(y_H(\mathsf M)\), and under the identifications \([n]=E\), \(d=r\), and \(\rho_M=\rk_{\mathsf M}\), one has \(t^{e_S}=\bm t_S\).  Replace \(u\) by \(u^{-1}\) in this identity and multiply both sides by \(u^r\).  On the left-hand side,
  \[
    u^r[\mathcal O(1)]\lambda_{u^{-1}}(\mathcal S)
    =\lambda_u(\mathcal S^\vee),
  \]
  so the result is
  \[
    \chi_H\bigl(\operatorname{Gr}(r;E),
    y_H(\mathsf M)\mathcal E_{u,v}\bigr) =
    \sum_{S\subseteq E}
    \bm t_S u^{\rk_{\mathsf M}(S)}
    v^{|S|-\rk_{\mathsf M}(S)}
    =T_{\mathsf M}^H(u,v;\bm t).
  \]

  For a general flag \(\mathcal F\), the orbit closure \(V(\sigma_{\mathcal F})\) is a product of smaller permutohedral varieties, one for each interval \(F_{i+1}\setminus F_i\).  The restrictions of \(\mathcal S_{\mathsf M}\) and \(\mathcal Q_{\mathsf M}\) decompose over the direct-sum components of \(\operatorname{in}_{\mathcal F}\mathsf M\), namely the successive minors \((\mathsf M|F_{i+1})/F_i\), by the minor decomposition property of tautological classes \cite[Propositions~5.2 and~5.3] {BergetEurSpinkTseng2023}.  Multiplying the empty-flag formula over these intervals gives the stated product.
\end{proof}

\begin{corollary}\label{cor:equivariant-single-tautological-gw}
  Let \(\mathsf M\) be a matroid on \(E\), and let \(\mathcal F\) denote an arbitrary flag of subsets. The \(H\)-equivariant Grothendieck weight of \(\lambda_v(\mathcal Q_{\mathsf M}^{\vee})\) is
  \[
    \mathcal F\longmapsto
    L_{\operatorname{in}_{\mathcal F}\mathsf M}^H(v,\bm t)
    =
    \prod_{i\in\Loop(\operatorname{in}_{\mathcal F}\mathsf M)}(1+vt_i),
  \]
  where, for any matroid \(\mathsf N\) on \(A\),
  \[
    L_\mathsf{N}^H(v,\bm t)
    =
    \prod_{i\in\Loop(\mathsf{N})}(1+vt_i).
  \]

  The \(H\)-equivariant Grothendieck weight of \(\lambda_u(\mathcal S_{\mathsf M}^{\vee})\) is
  \[
    \mathcal F\longmapsto
    I_{\operatorname{in}_{\mathcal F}\mathsf M}^H(u,\bm t)
    =
    \sum_{B\in\mathcal I(\operatorname{in}_{\mathcal F}\mathsf M)}
    u^{|B|}\bm t_B,
  \]
  where, for any matroid \(\mathsf N\) on \(A\),
  \[
    I_\mathsf{N}^H(u,\bm t)
    =
    \sum_{B\in\mathcal I(\mathsf{N})}u^{|B|}\bm t_B.
  \]
\end{corollary}

\begin{proof}
  Set \(u=0\) and \(v=0\), respectively, in \cref{prop:equivariant-tautological-gw}.  On each direct-sum component \((\mathsf M|F_{i+1})/F_i\) of \(\operatorname{in}_{\mathcal F}\mathsf M\), the specialization \(u=0\) keeps exactly the rank-zero subsets, which are the subsets of its loop set. This gives the product formula for \(\lambda_v(\mathcal Q_{\mathsf M}^{\vee})\).  The specialization \(v=0\) keeps exactly the subsets \(B\) satisfying \(|B|=\rk(B)\), which are the independent subsets of the same component.  This gives the formula for \(\lambda_u(\mathcal S_{\mathsf M}^{\vee})\).
\end{proof}

\section{Positivity for tautological bundles under nef twists}
\label{sec:nef-tautological-positivity}

\thmENefTautologicalPositivity*

We begin by applying the product rule and identifying the resulting coefficients.  Let
\[
  \mathcal E
  =\lambda_u(\mathcal S_{\mathsf M}^{\vee})
  \lambda_{-1}(\mathcal Q_{\mathsf M}^{\vee}).
\]
In the notation of \cite[Proposition~1.8]{wang2026grothendieckweights}, the Grothendieck weight of \(\lambda_u(\mathcal S_{\mathsf M}^{\vee}) \lambda_v(\mathcal Q_{\mathsf M}^{\vee})\) on a braid cone \(\sigma_\mathcal F\) is
\[
  u^rT_{\operatorname{in}_\mathcal F\mathsf M}
  (1+u^{-1},1+v).
\]
Specializing \(v=-1\), we obtain
\[
  g_\mathcal E(\mathcal F)
  =u^rT_{\operatorname{in}_\mathcal F\mathsf M}
  (1+u^{-1},0),
\]
where \(r=\rk\mathsf M\) and \(T_\mathsf N(x,y)\) denotes the ordinary Tutte polynomial. To keep the notation simple, we fix a sufficiently generic \(w\in N_\R\), and define
\[
  \Gamma_w
  =\left\{
    (\sigma,\tau)\in\Sigma_E\times\Sigma_E:
    (\sigma+w)\cap\tau\text{ is nonempty and bounded}
  \right\}.
\]
The product rule \cref{prop:gw-product-rule} gives
\[
  \chi(X_E,\mathcal E\mathcal L)
  =\sum_{(\sigma_\mathcal F,\sigma_\mathcal G)\in\Gamma_w}
  (-1)^{\ell(\mathcal F)+\ell(\mathcal G)-n+1}
  g_\mathcal E(\mathcal F)g_\mathcal L(\mathcal G).
\]
If \(\operatorname{in}_\mathcal F\mathsf M\) has a loop, then its Tutte polynomial is divisible by its second variable, so \(g_\mathcal E(\mathcal F)=0\).  We may therefore restrict the summation to \(\sigma_\mathcal F\in\Sigma_\mathsf M\).

Fix the natural order \(1<\cdots<n\) on \(E\).  For a loopless matroid \(\mathsf N\) on \(E\), let \(\operatorname{BC}(\mathsf N)\) be its broken-circuit complex, with the convention that the broken circuit of a circuit \(C\) is \(C\setminus\{\max C\}\).  Whitney's broken-circuit identity \cite[Section~7.4]{oxley2011matroid} states that
\[
  T_\mathsf N(x,0)
  =\sum_{I\in\operatorname{BC}(\mathsf N)}
  (x-1)^{\rk\mathsf N-|I|}.
\]
Substituting \(x=1+u^{-1}\) and multiplying by \(u^{\rk\mathsf N}\) gives
\[
  u^{\rk\mathsf N}T_\mathsf N(1+u^{-1},0)
  =\sum_{I\in\operatorname{BC}(\mathsf N)}u^{|I|}.
\]
Substituting this identity into the product-rule formula and regrouping by \(I\subseteq E\), we obtain
\begin{equation}\label{eq:nef-tautological-coefficient-expansion}
  \chi(X_E,\mathcal E\mathcal L)
  =\sum_{I\subseteq E}a_Iu^{|I|},
\end{equation}
where
\begin{equation}\label{eq:nef-tautological-a-I-initial}
  a_I
  =\sum_{\substack{
      (\sigma_\mathcal F,\sigma_\mathcal G)\in\Gamma_w\\
      \sigma_\mathcal F\in\Sigma_\mathsf M\\
  I\in\operatorname{BC}(\operatorname{in}_\mathcal F\mathsf M)}}
  (-1)^{\ell(\mathcal F)+\ell(\mathcal G)-n+1}
  g_\mathcal L(\mathcal G).
\end{equation}
The condition involving \(I\) in this summation naturally leads to the following subsets of the Bergman fan.

\begin{lemma}\label{lem:broken-circuit-support-convex}
  For every \(I\subseteq E\), the set
  \[
    B_I
    =\left\{
      x\in|\Sigma_\mathsf M|:
      I\in\operatorname{BC}(\operatorname{in}_x\mathsf M)
    \right\}
  \]
  is closed and tropically convex.
\end{lemma}

\begin{proof}
  For \(e\in E\), set \(I_{<e}=I\cap[e-1]\).  If \(\mathsf N\) is loopless, then
  \begin{equation}\label{eq:broken-circuit-closure-characterization}
    I\in\operatorname{BC}(\mathsf N)
    \quad\Longleftrightarrow\quad
    e\notin\operatorname{cl}_\mathsf N(I_{<e})
    \quad\text{for every \(e\in E\)}.
  \end{equation}
  Indeed, if \(e\in\operatorname{cl}_\mathsf N(I_{<e})\), looplessness gives a circuit \(C\) such that \(e\in C\subseteq I_{<e}\cup\{e\}\).  Then \(e=\max C\), so \(C\setminus\{e\}\) is a broken circuit contained in \(I\).  Conversely, if \(C\setminus\{\max C\}\subseteq I\), then, with \(e=\max C\), one has \(C\setminus\{e\}\subseteq I_{<e}\), and hence \(e\in\operatorname{cl}_\mathsf N(I_{<e})\).

  Let
  \[
    p_e:N_\R\longrightarrow N_{E\setminus I_{<e},\R}
  \]
  be the coordinate projection.  We claim that
  \begin{equation}\label{eq:broken-circuit-support-relaxed-bergman}
    B_I
    =|\Sigma_\mathsf M|
    \cap\bigcap_{e\in E}
    p_e^{-1}\left(
      \left|
      \Sigma_{\mathsf M/I_{<e},\,
      E\setminus(I_{<e}\cup\{e\})}
      \right|
    \right).
  \end{equation}
  We first record the compatibility needed to prove this equality.  If a subset \(J\) is independent in \(\operatorname{in}_x\mathsf M\), then
  \begin{equation}\label{eq:initial-contraction-compatibility}
    (\operatorname{in}_x\mathsf M)/J
    =\operatorname{in}_{p_J(x)}(\mathsf M/J),
  \end{equation}
  where \(p_J\) deletes the coordinates in \(J\).  To see this, extend \(J\) to an \(x\)-maximum basis of \(\mathsf M\).  Thus the maximum \(x\)-weight among bases containing \(J\) is the unrestricted maximum. Comparing the weights of the bases \(A\) of \(\mathsf M/J\) with those of the corresponding bases \(A\cup J\) of \(\mathsf M\) proves the claimed compatibility.

  Suppose first that \(x\in B_I\), and let \(\mathsf N=\operatorname{in}_x\mathsf M\).  The matroid \(\mathsf N\) is loopless, and \(I\in\operatorname{BC}(\mathsf N)\), so \(I\), and hence every \(I_{<e}\), is independent in \(\mathsf N\).  By \cref{eq:initial-contraction-compatibility,eq:broken-circuit-closure-characterization},
  \[
    e\notin
    \Loop\left(
      \operatorname{in}_{p_e(x)}(\mathsf M/I_{<e})
    \right).
  \]
  This is precisely the condition that \(p_e(x)\) belong to the relaxed Bergman fan appearing on the right-hand side of \eqref{eq:broken-circuit-support-relaxed-bergman}.

  Conversely, suppose that \(x\) belongs to the right-hand side of \eqref{eq:broken-circuit-support-relaxed-bergman}, and again let \(\mathsf N=\operatorname{in}_x\mathsf M\).  We prove successively for \(e=1,\ldots,n\) that \(I_{<e}\) is independent in \(\mathsf N\).  This is clear for \(e=1\).  If it holds for some \(e<n\), then \eqref{eq:initial-contraction-compatibility} and the relaxed Bergman condition show that \(e\) is not a loop of \(\mathsf N/I_{<e}\).  If \(e\in I\), this implies that \(I_{<e+1}=I_{<e}\cup\{e\}\) is independent; if \(e\notin I\), then \(I_{<e+1}=I_{<e}\).  The induction follows.  We may therefore apply \eqref{eq:initial-contraction-compatibility} for every \(e\).  The relaxed Bergman conditions then give \(e\notin\operatorname{cl}_\mathsf N(I_{<e})\) for every \(e\), so \eqref{eq:broken-circuit-closure-characterization} gives \(I\in\operatorname{BC}(\mathsf N)\).  This proves \eqref{eq:broken-circuit-support-relaxed-bergman}.

  The factor \(|\Sigma_\mathsf M|\) and each relaxed Bergman fan on the right-hand side of \eqref{eq:broken-circuit-support-relaxed-bergman} are closed and tropically convex by \cref{prop:relaxed-bergman-convex}.  Moreover, coordinate projection commutes with tropical linear combinations:
  \[
    p_e\bigl(\min(a+x,b+y)\bigr)
    =\min\bigl(a+p_e(x),b+p_e(y)\bigr).
  \]
  Hence each inverse image in this intersection is closed and tropically convex, and so is their finite intersection with \(|\Sigma_\mathsf M|\).  Thus \(B_I\) is closed and tropically convex. Since every set in the intersection is a union of braid cones, so is \(B_I\).
\end{proof}

\begin{lemma}\label{lem:gp-normal-cone-tropically-convex}
  Every normal cone of a generalized permutohedron is tropically convex.
\end{lemma}

\begin{proof}
  The normal fan of a generalized permutohedron coarsens the braid fan \cite[Proposition~3.2]{postnikov2008faces}.  Hence each normal cone is a convex union of braid cones and is cut out by weak coordinate-order relations \(x_i\leq x_j\).  Each halfspace \(x_i\leq x_j\) is tropically convex, since the same inequalities for \(x\) and \(y\) imply
  \[
    \min(a+x_i,b+y_i)\leq\min(a+x_j,b+y_j).
  \]
  The result follows by taking intersections.
\end{proof}

\begin{proof}[Proof of \cref{thm:nef-tautological-positivity}]
  By \cref{lem:broken-circuit-support-convex}, the broken-circuit condition in \eqref{eq:nef-tautological-a-I-initial} is equivalent to \(\sigma_\mathcal F\subseteq B_I\).  Thus
  \[
    a_I
    =\sum_{\substack{
        (\sigma_\mathcal F,\sigma_\mathcal G)\in\Gamma_w\\
    \sigma_\mathcal F\subseteq B_I}}
    (-1)^{\ell(\mathcal F)+\ell(\mathcal G)-n+1}
    g_\mathcal L(\mathcal G).
  \]

  Choose a torus-invariant Cartier divisor representing \(\mathcal L\), and let \(P\subseteq M_\R\) be its lattice polytope.  Since \(\mathcal L\) is nef, its support function is convex, and the normal fan of \(P\) coarsens the braid fan \cite[Sections~6.1 and~6.3]{cox2011toric}; equivalently, \(P\) is a generalized permutohedron \cite[Proposition~3.2]{postnikov2008faces}.  The polytope is determined by \(\mathcal L\) only up to translation by a lattice vector, which does not affect the argument.  By \eqref{eq:gw-nef-line-bundle},
  \[
    g_\mathcal L(\mathcal G)
    =|\face_\mathcal G(P)\cap M|.
  \]

  For each \(m\in P\cap M\), let \(F(m)\) be the unique face of \(P\) whose relative interior contains \(m\), and let \(N_m\subseteq N_\R\) be its \emph{outer} normal cone.
  Explicitly,
  \[
    N_m =
    \left\{
      v\in N_\R:
      \langle p,v\rangle\leq\langle m,v\rangle
      \text{ for every }p\in P
    \right\}.
  \]
  Since the normal fan of \(P\) coarsens the braid fan, \(N_{\face_\mathcal G(P)}\) is the smallest normal cone of \(P\) containing \(\sigma_\mathcal G\).  Therefore
  \[
    m\in\face_\mathcal G(P)
    \quad\Longleftrightarrow\quad
    \sigma_\mathcal G\subseteq N_m.
  \]
  Indeed, the right-hand side is equivalent to \(N_{\face_\mathcal G(P)}\subseteq N_m\), and the order-reversing correspondence between faces and normal cones identifies this with \(F(m)\subseteq\face_\mathcal G(P)\), or equivalently with the left-hand side. Consequently,
  \[
    |\face_\mathcal G(P)\cap M|
    =\sum_{m\in P\cap M}
    \mathbf1_{\sigma_\mathcal G\subseteq N_m}.
  \]

  Set
  \[
    C_{I,m}=(B_I+w)\cap N_m.
  \]
  \(N_m\) is tropically convex by \cref{lem:gp-normal-cone-tropically-convex}.  Thus \(C_{I,m}\) is closed and tropically convex.

  The braid cones give \(C_{I,m}\) a finite polyhedral cell structure.  By genericity of \(w\), its nonempty cells are
  \[
    (\sigma_\mathcal F^\circ+w)
    \cap\sigma_\mathcal G^\circ,
    \qquad
    \sigma_\mathcal F\subseteq B_I,
    \quad
    \sigma_\mathcal G\subseteq N_m,
  \]
  and such a cell has dimension
  \[
    \ell(\mathcal F)+\ell(\mathcal G)-n+1.
  \]
  Denote the union of such \emph{bounded} cells by \(C_{I,m}^b\).  Expanding the lattice-point count in the formula for \(a_I\) and then regrouping by \(m\) gives
  \begin{align*}
    a_I
    &=\sum_{m\in P\cap M}
    \sum_{\substack{
        (\sigma_\mathcal F,\sigma_\mathcal G)\in\Gamma_w\\
        \sigma_\mathcal F\subseteq B_I,
    \sigma_\mathcal G\subseteq N_m}}
    (-1)^{\ell(\mathcal F)+\ell(\mathcal G)-n+1}\\
    &=\sum_{m\in P\cap M}\chi(C_{I,m}^b).
  \end{align*}

  Every cell of \(C_{I,m}\) has a pointed recession cone because braid cones are pointed in \(N_\R\).  The cellwise argument of \cite[Proof of Proposition~3.4] {eur2025vanishingtheoremscombinatorialgeometries} therefore gives a deformation retraction \(C_{I,m}\to C_{I,m}^b\).  On the other hand, a nonempty tropically convex set is contractible by \cite[Theorem~2]{develin2004tropical}.  Hence
  \[
    \chi(C_{I,m}^b)
    =\chi(C_{I,m})
    =\mathbf1_{C_{I,m}\neq\varnothing}.
  \]
  It follows that \(a_I\geq0\) for every \(I\subseteq E\).  Under our convention that a broken circuit is obtained by deleting the largest element of a circuit, the element \(n\) belongs to no broken circuit.  Hence, for every \(I\subseteq E\setminus\{n\}\) and every \(x\in|\Sigma_\mathsf M|\),
  \[
    I\in\operatorname{BC}(\operatorname{in}_x\mathsf M)
    \quad\Longleftrightarrow\quad
    I\cup\{n\}\in\operatorname{BC}(\operatorname{in}_x\mathsf M).
  \]
  Thus \(B_I=B_{I\cup\{n\}}\), and consequently \(a_I=a_{I\cup\{n\}}\).  Pairing these terms in \eqref{eq:nef-tautological-coefficient-expansion} gives
  \[
    \chi(X_E,\mathcal E\mathcal L)
    =(1+u)\sum_{I\subseteq E\setminus\{n\}}a_Iu^{|I|}.
  \]
  The theorem follows because every \(a_I\) is nonnegative.
\end{proof}
\begin{remark}
  In fact, the proof shows that
  \[
    a_I = \#\{m\in P\cap M:(B_I+w)\cap N_m\neq\varnothing\}.
  \]
\end{remark}

\corCAntiample*

\begin{proof}
  The cases \(n\leq2\) are immediate.  Indeed, if \(n=1\), then \(X_E\) is a point.  If \(n=2\), write \(\mathcal L=\mathcal O_{\PP^1}(d)\) with \(d\geq1\).  The only loopless matroids are \(\mathsf{U}_{1,2}\) and \(\mathsf{U}_{2,2}\).  Their matroid fans are the zero fan and \(\Sigma_E\), respectively.  Thus their matroid Euler characteristics of \(\mathcal L^{-1}\) are \(1\) and \(\chi(\PP^1,\mathcal O_{\PP^1}(-d))=1-d\), so the asserted quantities are \(1\) and \(d-1\). We may therefore assume \(n\geq3\).

  By \cite[Proposition~1.8]{wang2026grothendieckweights}, specialized at \(u=0\) and \(v=-1\), the Grothendieck weight of \(\lambda_{-1}(\mathcal Q_\mathsf M^\vee)\) is the indicator function of \(\Sigma_\mathsf M\).  Thus \(\lambda_{-1}(\mathcal Q_\mathsf M^\vee)=\Delta_\mathsf M\), and \eqref{eq:matroid-euler-as-product} gives
  \[
    \chi_\mathsf M(\mathcal L^{-1})
    =\chi\left(
      X_E,\lambda_{-1}(\mathcal Q_\mathsf M^\vee)\mathcal L^{-1}
    \right).
  \]

  Since \(\rk\mathcal Q_\mathsf M=n-r\), exterior-power duality gives
  \[
    \lambda_{-1}(\mathcal Q_\mathsf M)
    =(-1)^{n-r}\det(\mathcal Q_\mathsf M)
    \lambda_{-1}(\mathcal Q_\mathsf M^\vee).
  \]
  The tautological relation \(\mathcal S_\mathsf M+\mathcal Q_\mathsf M=\mathcal O_{X_E}^{E}\) gives \(\det(\mathcal Q_\mathsf M)=\det(\mathcal S_\mathsf M^\vee)\). Serre duality \cite[Chapter~III, Corollary~7.7]{hartshorne1977algebraic} on the \((n-1)\)-dimensional variety \(X_E\) therefore gives, with \(\omega_{X_E}\) denoting the canonical line bundle,
  \begin{align*}
    (-1)^{r-1}\chi_\mathsf M(\mathcal L^{-1})
    &=(-1)^{r-1+(n-1)}
    \chi\left(
      X_E,
      \lambda_{-1}(\mathcal Q_\mathsf M)
      \mathcal L\omega_{X_E}
    \right)\\
    &=(-1)^{r-1+(n-1)+(n-r)}
    \chi\left(
      X_E,
      \det(\mathcal S_\mathsf M^\vee)
      \lambda_{-1}(\mathcal Q_\mathsf M^\vee)
      \mathcal L\omega_{X_E}
    \right)\\
    &=\chi\left(
      X_E,
      \det(\mathcal S_\mathsf M^\vee)
      \lambda_{-1}(\mathcal Q_\mathsf M^\vee)
      \mathcal L\omega_{X_E}
    \right).
  \end{align*}

  We next verify that \(\mathcal L\omega_{X_E}\) is nef directly from its polytope.  Write the lattice polytope of \(\mathcal L\) as
  \[
    P=\left\{m\in M_\R:
      \langle m,\bar e_S\rangle\leq a_S
      \text{ for every }\varnothing\subsetneq S\subsetneq E
    \right\}.
  \]
  The divisor--polytope correspondence and the toric canonical-bundle formula \cite[Sections~4.3 and~6.1 and Theorem~8.2.3]{cox2011toric} identify the polytope of \(\mathcal L\omega_{X_E}\) with
  \[
    P'
    =\left\{m\in M_\R:
      \langle m,\bar e_S\rangle\leq a_S-1
      \text{ for every }\varnothing\subsetneq S\subsetneq E
    \right\}.
  \]
  Set \(a_\varnothing=a_E=0\).  Since \(\mathcal L\) is ample, the function \(S\mapsto a_S\) is integral and strictly submodular.  Let \(b_\varnothing=b_E=0\) and \(b_S=a_S-1\) for \(\varnothing\subsetneq S\subsetneq E\).  We verify directly that \(b\) is submodular.  It is enough to show, for every \(i,j\) and \(A\subseteq E\setminus\{i,j\}\), that
  \[
    b_{A\cup\{i\}}+b_{A\cup\{j\}}
    \geq b_A+b_{A\cup\{i,j\}}.
  \]
  If \(A\neq\varnothing\) and \(A\cup\{i,j\}\neq E\), then submodularity of \(a\) gives
  \[
    b_{A\cup\{i\}}+b_{A\cup\{j\}}
    =a_{A\cup\{i\}}+a_{A\cup\{j\}}-2
    \geq a_A+a_{A\cup\{i,j\}}-2
    =b_A+b_{A\cup\{i,j\}}.
  \]
  If \(A=\varnothing\), then \(\{i,j\}\neq E\) because \(n\geq3\), and strict submodularity and integrality of \(a\) give
  \[
    b_{\{i\}}+b_{\{j\}}
    =a_{\{i\}}+a_{\{j\}}-2
    \geq a_{\{i,j\}}-1
    =b_\varnothing+b_{\{i,j\}}.
  \]
  Finally, if \(A\cup\{i,j\}=E\), then \(A\neq\varnothing\) and
  \[
    b_{A\cup\{i\}}+b_{A\cup\{j\}}
    =a_{A\cup\{i\}}+a_{A\cup\{j\}}-2
    \geq a_A-1
    =b_A+b_E.
  \]
  Thus \(b\) is submodular, \(P'\) is a generalized permutohedron, and \(\mathcal L\omega_{X_E}\) is nef.

  Applying \cref{thm:nef-tautological-positivity} to this line bundle gives
  \[
    \chi\left(
      X_E,
      \lambda_u(\mathcal S_\mathsf M^\vee)
      \lambda_{-1}(\mathcal Q_\mathsf M^\vee)
      \mathcal L\omega_{X_E}
    \right)
    \in\Z_{\geq0}[u].
  \]
  By the definition of \(\lambda_u\), its coefficient of \(u^r\) is the Euler characteristic obtained above by Serre duality, and is therefore nonnegative.
\end{proof}

\section{Positivity for tuples of matroids}\label{sec:k-fold-positivity}

For a pair \(\bm{\mathsf M}=(\mathsf M_1,\mathsf M_2)\) without common loops, Berget and Fink construct the diagonal Dilworth truncation \(D=D(\bm{\mathsf M})\) and a Cohen--Macaulay external activity complex \(\Delta_w=\Delta_w(\bm{\mathsf M})\).  By \cite[Theorem~C(2)--(3)]{berget2025externalactivitycomplexpair}, its \(\Z^2\)-graded \(K\)-polynomial is
\[
  \mathcal K(\Delta_w;V_1,V_2)
  =
  \mathcal K_{\bm{\mathsf M}}(V_1,V_2)
  :=
  \chi\left(
    X_{E},
    \lambda_{-V_1}(\mathcal Q_{\mathsf M_1}^{\vee})
    \lambda_{-V_2}(\mathcal Q_{\mathsf M_2}^{\vee})
  \right).
\]
Thus, if \(c=n-\rk D\), then
\[
  (-1)^{d_1+d_2-c}
  [U_1^{d_1}U_2^{d_2}]\,
  \mathcal K_{\bm{\mathsf M}}(1-U_1,1-U_2)
  \geq0
  \qquad(d_1,d_2\geq0).
\]
We extend their Dilworth truncation and positivity result to arbitrary \(k\).

\subsection{The diagonal Dilworth polymatroid and external activity complex}
\label{subsec:external-activity-tuple}

Fix an integer \(k\geq2\) and a tuple
\(
  \bm{\mathsf M}=(\mathsf M_1,\ldots,\mathsf M_k)
\)
of matroids on \(E\) with no common loop.  Set
\[
  \mathcal K_{\bm{\mathsf M}}(V_1,\ldots,V_k)
  =
  \chi\left(
    X_E,
    \prod_{j=1}^k
    \lambda_{-V_j}(\mathcal Q_{\mathsf M_j}^{\vee})
  \right).
\]

\begin{definition}\label{def:k-fold-dilworth}
  For a tuple \(\bm{\mathsf{M}}=(\mathsf{M}_1,\ldots,\mathsf{M}_k)\) of matroids on \(E\) with no common loop, its \emph{diagonal Dilworth polymatroid} \(D(\bm{\mathsf{M}})\) is defined by the following set function:
  \[
    \rho(A)
    =
    \min_{A=A_1\sqcup\cdots\sqcup A_s}
    \sum_{q=1}^s
    \left(\sum_{j=1}^k\rk_{\mathsf M_j}(A_q)-1\right),
    \qquad
    \rho(\varnothing)=0.
  \]
  Here \(A\subseteq E\), no \(A_q\) is empty, and \(s\) may be \(0\).
\end{definition}
\begin{remark}
  When \(k=2\), \cite[Definition~4.1] {berget2025externalactivitycomplexpair} gives the rank function of \(D(\mathsf M_1,\mathsf M_2)\) as
  \[
    \rk_D(A)
    =
    \min_{T_1\sqcup\cdots\sqcup T_s\subseteq A}
    \left\{
      \sum_{q=1}^{s}
      \bigl(\rk_{\mathsf M_1}(T_q)+\rk_{\mathsf M_2}(T_q)-1\bigr)
      +\left|A\setminus\bigcup_{q=1}^sT_q\right|
    \right\},
  \]
  where the \(T_q\) are nonempty and \(s\) may be zero.  Write \(T=\bigsqcup_qT_q\).  Minimizing first over partitions of \(T\), we obtain
  \[
    \rk_D(A)=\min_{T\subseteq A}\bigl(\rho(T)+|A\setminus T|\bigr).
  \]
  Taking \(T=A\) gives \(\rk_D(A)\leq\rho(A)\).  Conversely, append the singleton parts \(\{i\}\), for \(i\in A\setminus T\), to a partition attaining \(\rho(T)\).  Since the pair has no common loops, each appended part has cost
  \[
    \rk_{\mathsf M_1}(\{i\})+\rk_{\mathsf M_2}(\{i\})-1\in\{0,1\}.
  \]
  Hence \(\rho(A)\leq\rho(T)+|A\setminus T|\) for every \(T\subseteq A\).
  Minimizing over \(T\) gives \(\rho(A)\leq\rk_D(A)\), and therefore \(\rho=\rk_D\).
\end{remark}

\begin{proposition}\label{prop:rho-is-polymatroid}
  The function \(\rho\) is the rank function of an integral polymatroid.
\end{proposition}
\begin{proof}
  The Dilworth truncation construction \cite[Section~48.2, especially Theorem~48.2] {schrijver2003combinatorial} associates to every submodular function \(f:2^E\to\mathbb R\) a largest submodular function \(\hat f\) such that \(\hat f(\varnothing)=0\) and \(\hat f(A)\leq f(A)\) for every nonempty \(A\subseteq E\).  It is given by
  \[
    \hat f(A) = \min_{A=A_1\sqcup\cdots\sqcup A_s} \sum_{q=1}^s f(A_q),
    \qquad
    \hat f(\varnothing)=0.
  \]
  Apply this theorem to the integral submodular function
  \[
    f(A)=\sum_{j=1}^k\rk_{\mathsf M_j}(A)-1,
  \]
  for which \(f(\varnothing)=-1\).  We obtain the submodularity and integrality of \(\rho=\hat f\).  Moreover, \(f\) is nondecreasing and \(f(A)\geq0\) for every nonempty \(A\), since the \(\mathsf M_j\) have no common loop.  The partition formula then shows that \(\rho\) is nondecreasing and nonnegative.  Therefore \(\rho\) is the rank function of an integral polymatroid.
\end{proof}

Set \(\widetilde E=[k]\times E\), and choose \(\bm w=(w_j)_{j\in[k]}\in\R^{\widetilde E}\).  We also write \(w_j\) for the image of its \(j\)-th block in \(N_\R\), and write \(\widetilde w\) for the image of \(\bm w\) in \(N_{\widetilde E,\R}\).  We require the tuple \((w_j)_{j\in[k]}\in(N_\R)^k\) to be jointly generic in the sense of \cref{thm:equivariant-multi-product-rule}, and \(\widetilde w\) to be sufficiently generic on \(\Sigma_{\widetilde E}\) in the sense of \cref{subsec:tropical-initial-degeneration}. The pullbacks of both requirements to \(\R^{\widetilde E}\) are dense rational open conditions, so they can be met simultaneously by an integral choice of \(\bm w\).

Introduce formal symbols \(x_i^{(j)}\), for \(j\in[k]\) and \(i\in E\).  For \(A\subseteq E\), set
\[
  \bm x_A^{(j)}=\prod_{i\in A}x_i^{(j)},
  \qquad
  \bm x^{(j)}=\bm x_E^{(j)},
  \qquad
  \bm X=\prod_{j=1}^k\bm x^{(j)}.
\]
For \(m\mid\bm X\), set \(m^\vee=\bm X/m\). For \(u\in N_{\R}\), define its loop label by
\[
  \lambda_{\bm w}(u)
  =
  \prod_{j=1}^k
  \bm x_{\Loop(\operatorname{in}_{u-w_j}\mathsf M_j)}^{(j)}.
\]

\begin{definition}\label{def:k-fold-external-activity-complex}
  For a tuple of matroids \(\bm{\mathsf M}=(\mathsf M_1,\ldots,\mathsf M_k)\), its \emph{external activity complex} is
  \[
    \Delta_{\bm w}(\bm{\mathsf M})
    =
    \left\{
      p\mid\bm X:
      p\mid\lambda_{\bm w}(u)^\vee
      \text{ for some }u\in N_\R
    \right\}.
  \]
  We identify a square-free divisor \(p\mid\bm X\) with its support
  \[
    \{(j,i)\in\widetilde E:x_i^{(j)}\mid p\}.
  \]
  Under this identification, divisibility becomes inclusion, \(p^\vee\) becomes the complement in \(\widetilde E\), and intersections and unions of square-free monomials are taken on their supports.  Thus \(\Delta_{\bm w}(\bm{\mathsf M})\) is a simplicial complex on the ambient set \(\widetilde E\), possibly with ghost vertices, and we use the same symbol for a face, its support, and its square-free monomial.
\end{definition}

\begin{remark}
  When \(k=2\), \cref{def:k-fold-external-activity-complex} coincides with \cite[Theorem~6.15]{berget2025externalactivitycomplexpair}.
\end{remark}

Write \(\Delta=\Delta_{\bm w}(\bm{\mathsf M})\) if \(\bm w\) and \(\bm{\mathsf M}\) are both clear from the context.

For \(v=(j,i)\in\widetilde E\), write \(z_i^{(j)}=z_v\) in the \(K\)-polynomial notation of \cref{def:simplicial-k-polynomial}.  Set
\[
  c=(k-1)n-\rho(E).
\]
We now restate the two main results.

\thmAPositivity*

\thmBExternalActivity*

\begin{lemma}\label{lem:theorem-b-implies-a}
  \Cref{thm:B-CM-Delta} implies \cref{thm:A-positivity-K-poly}.
\end{lemma}

\begin{proof}
  Assume \cref{thm:B-CM-Delta}, and set \(\Delta=\Delta_{\bm w}(\bm{\mathsf M})\).  In the face formula of \cref{def:simplicial-k-polynomial}, put \(z_i^{(j)}=1-x_i^{(j)}\).  For \(S\subseteq\widetilde E\), write \(\bm x^S=\prod_{(j,i)\in S}x_i^{(j)}\).  The face formula gives
  \[
    [\bm x^S]\mathcal K_{\widetilde E}(\Delta;\bm 1-\bm x)
    =
    \begin{cases}
      -\widetilde\chi\bigl(\link_\Delta(\widetilde E\setminus S)\bigr),
      &\widetilde E\setminus S\in\Delta,\\
      0,&\widetilde E\setminus S\notin\Delta.
    \end{cases}
  \]
  Suppose that \(\widetilde E\setminus S\in\Delta\), and set \(L=\link_\Delta(\widetilde E\setminus S)\).  By \cref{thm:B-CM-Delta}, the complex \(\Delta\) is Cohen--Macaulay of dimension \(n+\rho(E)-1\).  Hence \(L\) is Cohen--Macaulay of dimension
  \[
    d
    =n+\rho(E)-1-(kn-|S|)
    =|S|-c-1.
  \]
  By \cref{thm:reisner-criterion}, its reduced homology vanishes below degree \(d\).  Therefore
  \[
    (-1)^{|S|-c}
    [\bm x^S]\mathcal K_{\widetilde E}(\Delta;\bm 1-\bm x)
    =
    \dim_\k\widetilde H_d(L;\k)
    \geq0.
  \]
  This also includes \(d=-1\), with the usual convention for \(\{\varnothing\}\).

  Under the specialization \(x_i^{(j)}=U_j\), the \(K\)-polynomial identity in \cref{thm:B-CM-Delta} gives
  \[
    \mathcal K_{\bm{\mathsf M}}(1-U_1,\ldots,1-U_k)
    =
    \left.
    \mathcal K_{\widetilde E}(\Delta;\bm 1-\bm x)
    \right|_{x_i^{(j)}=U_j}.
  \]
  Thus the coefficient of \(U_1^{d_1}\cdots U_k^{d_k}\) is the sum of the fine coefficients indexed by the subsets \(S\subseteq\widetilde E\) satisfying
  \[
    |S\cap(\{j\}\times E)|=d_j
    \qquad(j\in[k]).
  \]
  Every such \(S\) has \(|S|=d_1+\cdots+d_k\), so all summands have the sign asserted in \cref{thm:A-positivity-K-poly}.
\end{proof}

\begin{lemma}\label{lem:k-fold-loop-reduction}
  It suffices to prove \cref{thm:B-CM-Delta} when every \(\mathsf M_j\) is loopless.
\end{lemma}

\begin{proof}
  For each \(j\in[k]\), let \(L_j=\Loop(\mathsf M_j)\), set \(\ell_j=|L_j|\), and let \(\widehat{\mathsf M}_j\) be obtained from \(\mathsf M_j\) by replacing every element of \(L_j\) by a coloop. Equivalently, if \(\mathsf N_j=\mathsf M_j\setminus L_j\), then
  \[
    \mathsf M_j=\mathsf N_j\oplus \mathsf{U}_{0,L_j},
    \qquad
    \widehat{\mathsf M}_j
    =\mathsf N_j\oplus \mathsf{U}_{\ell_j,L_j}.
  \]
  Set
  \[
    \widehat{\bm{\mathsf M}}
    =(\widehat{\mathsf M}_1,\ldots,\widehat{\mathsf M}_k).
  \]
  This tuple is loopless.

  For every \(A\subseteq E\), replacing the loops in \(L_j\) by coloops gives
  \[
    \rk_{\widehat{\mathsf M}_j}(A)
    =\rk_{\mathsf M_j}(A)+|A\cap L_j|.
  \]
  Write \(\widehat\rho\) for the diagonal Dilworth rank function of \(\widehat{\bm{\mathsf M}}\).  Since the sets in a partition of \(A\) are disjoint, the partition formula gives
  \[
    \widehat\rho(A)
    =\rho(A)+\sum_{j=1}^k|A\cap L_j|.
  \]
  In particular, if \(\ell=\sum_{j=1}^k\ell_j\), then
  \begin{equation}\label{eq:k-fold-rho-loop-reduction}
    \widehat\rho(E)=\rho(E)+\ell.
  \end{equation}

  Tautological \(K\)-classes are compatible with direct sum \cite[Proposition~5.13]{BergetEurSpinkTseng2023}.  In the notation of that proposition, apply the direct-sum formula to \(\mathsf N_j\oplus \mathsf{U}_{0,L_j}\) and \(\mathsf N_j\oplus \mathsf{U}_{\ell_j,L_j}\).  Since
  \[
    [\mathcal Q_{\mathsf{U}_{0,1}}^\vee]=[\mathcal O_{X_E}],
    \qquad
    [\mathcal Q_{\mathsf{U}_{1,1}}^\vee]=0,
  \]
  it specializes to
  \[
    [\mathcal Q_{\mathsf M_j}^\vee]
    =
    [\mathcal Q_{\widehat{\mathsf M}_j}^\vee]
    +\ell_j[\mathcal O_{X_E}].
  \]
  The multiplicativity of total exterior powers therefore gives
  \begin{equation}\label{eq:k-fold-loop-reduction}
    \mathcal K_{\bm{\mathsf M}}(\bm V)
    =
    \left(\prod_{j=1}^k(1-V_j)^{\ell_j}\right)
    \mathcal K_{\widehat{\bm{\mathsf M}}}(\bm V).
  \end{equation}

  Shrinking the genericity condition if necessary, take \(\bm w\) sufficiently generic for both tuples.  Set
  \[
    \widetilde L
    =\{(j,i)\in\widetilde E:i\in L_j\}.
  \]
  For every \(u\in N_\R\) and \(j\in[k]\), replacing the loops in \(L_j\) by coloops gives
  \[
    \left\{
      (j,i):i\notin
      \Loop(\operatorname{in}_{u-w_j}\widehat{\mathsf M}_j)
    \right\}
    =
    \left\{
      (j,i):i\notin
      \Loop(\operatorname{in}_{u-w_j}\mathsf M_j)
    \right\}
    \sqcup(\{j\}\times L_j).
  \]
  Let \(\Delta^\circ\) be the restriction of \(\Delta_{\bm w}(\bm{\mathsf M})\) to \(\widetilde E\setminus\widetilde L\).  The definition of the external activity complex shows that \(\Delta_{\bm w}(\bm{\mathsf M})\) is \(\Delta^\circ\) with \(\widetilde L\) adjoined as ghost vertices, while
  \begin{equation}\label{eq:k-fold-complex-loop-reduction}
    \Delta_{\bm w}(\widehat{\bm{\mathsf M}})
    =
    \Delta^\circ*2^{\widetilde L}.
  \end{equation}
  Consequently,
  \[
    \k[\Delta_{\bm w}(\widehat{\bm{\mathsf M}})]
    \simeq
    \k[\Delta^\circ]
    [x_i^{(j)}:(j,i)\in\widetilde L].
  \]

  Assume \cref{thm:B-CM-Delta} for the loopless tuple \(\widehat{\bm{\mathsf M}}\).  Polynomial extension reflects Cohen--Macaulayness, so \(\Delta^\circ\), and equivalently \(\Delta_{\bm w}(\bm{\mathsf M})\), is Cohen--Macaulay.  Moreover, \cref{eq:k-fold-rho-loop-reduction,eq:k-fold-complex-loop-reduction}, together with the loopless dimension formula, give
  \[
    \dim\Delta_{\bm w}(\bm{\mathsf M})
    =n+\widehat\rho(E)-1-\ell
    =n+\rho(E)-1.
  \]

  Adjoining a ghost vertex of degree \(e_j\) multiplies the \(\Z^k\)-graded \(K\)-polynomial by \(1-V_j\), while adjoining a cone vertex does not change it.  Hence
  \[
    \mathcal K(\Delta_{\bm w}(\bm{\mathsf M});\bm V)
    =
    \left(\prod_{j=1}^k(1-V_j)^{\ell_j}\right)
    \mathcal K(\Delta_{\bm w}(\widehat{\bm{\mathsf M}});\bm V).
  \]
  Comparing this identity with \eqref{eq:k-fold-loop-reduction} proves the \(K\)-polynomial identity for \(\bm{\mathsf M}\).  Its independence of \(\bm w\) follows immediately, and its multivaluativity follows exactly as in the proof of \cref{prop:k-fold-k-polynomial-loopless}.
\end{proof}

\subsection{The \texorpdfstring{\(K\)}{K}-polynomial identity in the loopless case}
For the rest of the section, we assume that every \(\mathsf M_j\) is loopless. We prove part~(1) of \cref{thm:B-CM-Delta} under this assumption.

To apply the product rule, we introduce the following notations. For a tuple of flags
\(
  \bm{\mathcal F}
  =(\mathcal F_1,\ldots,\mathcal F_k)
\), set
\[
  D_{\bm{\mathcal F}}
  =
  \bigcap_{j=1}^k(w_j+\sigma_{\mathcal F_j}).
\]
These are the cells in the common refinement of the translated braid fans. Joint genericity gives
\[
  D_{\bm{\mathcal{F}}}^\circ
  =
  \bigcap_{j=1}^k(w_j+\sigma_{\mathcal{F}_j}^\circ)
\]
and
\[
  \dim D_{\bm{\mathcal F}}
  =
  \sum_{j=1}^k\dim\sigma_{\mathcal F_j}-(k-1)(n-1).
\]
Moreover, \(D_{\bm{\mathcal F}}\) is bounded exactly when
\[
  \bigcap_{j=1}^k\sigma_{\mathcal F_j}=\{0\}.
\]
Define the loop label of \(\bm{\mathcal{F}}\) by
\[
  \lambda(\bm{\mathcal F})
  =
  \prod_{j=1}^k
  \bm x_{\Loop(\operatorname{in}_{\mathcal F_j}\mathsf M_j)}^{(j)}.
\]
If \(u\in D_{\bm{\mathcal F}}^\circ\), then
\[
  \lambda_{\bm w}(u)=\lambda(\bm{\mathcal F}),
\]
because \(u-w_j\in\sigma_{\mathcal F_j}^\circ\) and hence \(\operatorname{in}_{u-w_j}\mathsf M_j =\operatorname{in}_{\mathcal F_j}\mathsf M_j\) for every \(j\), by \cref{def:flag-initial-matroid}.

Now set
\[
  \Psi
  =\Psi_{\bm{\mathsf M}}(V_1,\ldots,V_k;\bm t)
  =
  \chi_H\left(
    X_E,
    \prod_{j=1}^k
    \lambda_{-V_j}(\mathcal Q_{\mathsf M_j}^{\vee})
  \right).
\]
By \cref{thm:equivariant-multi-product-rule,cor:equivariant-single-tautological-gw},
after temporarily setting \(x_i^{(j)}=1-V_jt_i\), we obtain
\begin{equation}\label{eq:k-fold-gw-expansion}
  \Psi
  =
  \sum_{\substack{
      D_{\bm{\mathcal F}}\neq\varnothing\\
  D_{\bm{\mathcal F}}\text{ bounded}}}
  (-1)^{\dim D_{\bm{\mathcal F}}}
  \lambda(\bm{\mathcal F}).
\end{equation}
We emphasize again that the \(x_i^{(j)}\) are formal symbols.
For a square-free \(m\mid\bm X\), define its fine coefficient by
\begin{equation}\label{eq:k-fold-fine-coefficient}
  [m]\Psi
  =
  \sum_{\substack{
      D_{\bm{\mathcal F}}\neq\varnothing\\
      D_{\bm{\mathcal F}}\text{ bounded}\\
  \lambda(\bm{\mathcal F})=m}}
  (-1)^{\dim D_{\bm{\mathcal F}}},\qquad \Psi=\sum_{m|\bm X} ([m]\Psi) m.
\end{equation}

To interpret these coefficients topologically, set
\[
  Q_m
  =
  \{\bm{\mathcal F}:D_{\bm{\mathcal F}}\neq\varnothing,\ \lambda(\bm{\mathcal F})\mid m\},
  \qquad
  Y_m
  =
  \bigcup_{\bm{\mathcal F}\in Q_m}D_{\bm{\mathcal F}}^\circ.
\]

\begin{lemma}\label{lem:k-fold-Ym}
  \(Y_m\) is closed and tropically convex. Moreover,
  \[
    Y_m\neq\varnothing \Longleftrightarrow m^\vee\in\Delta.
  \]
\end{lemma}
\begin{proof}
  We claim that
  \[
    Y_m = \{u\in N_\R: \lambda_{\bm w}(u)\mid m\}.
  \]
  In fact, if \(\bm{\mathcal{F}}\in Q_m\), take any \(z\in D_{\bm{\mathcal{F}}}^\circ\). Then \(\lambda_{\bm w}(z)=\lambda(\bm{\mathcal{F}})\mid m\), so \(z\) belongs to the right-hand side. Conversely, given \(z\) such that \(\lambda_{\bm w}(z)\mid m\), let \(\bm{\mathcal{F}}\) be the unique tuple of flags such that \(z\in D_{\bm{\mathcal{F}}}^\circ\). Then \(\lambda(\bm{\mathcal{F}})=\lambda_{\bm w}(z)\mid m\), so \(\bm{\mathcal{F}}\in Q_m\) and \(z\in Y_m\).

  For \(m\mid\bm X\), write \(|m|\) for its degree and set
  \[
    \operatorname{supp}_j(m)
    =\{i\in E:x_i^{(j)}\mid m\}
    \qquad(j\in[k]).
  \]

  The divisibility \(\lambda_{\bm w}(u)\mid m\) is equivalent to
  \[
    \Loop(\operatorname{in}_{u-w_j}\mathsf{M}_j)\subseteq \operatorname{supp}_j(m).
  \]
  Then, in terms of relaxed Bergman fans,
  \[
    Y_m
    =
    \bigcap_{j=1}^k
    \left(w_j+\Sigma_{\mathsf M_j,\operatorname{supp}_j(m)}\right).
  \]
  By \cref{prop:relaxed-bergman-convex}, it is therefore closed and tropically convex.

  The set \(Y_m\) is nonempty if and only if there is some \(u\in N_\R\) such that \(\lambda_{\bm w}(u)\mid m\), or equivalently
  \[
    m^\vee\mid\lambda_{\bm w}(u)^\vee.
  \]
  By \cref{def:k-fold-external-activity-complex}, this is exactly the condition \(m^\vee\in\Delta\).
\end{proof}

\begin{proposition}\label{prop:k-fold-link-coefficient}
  For every square-free \(m\mid\bm X\),
  \[
    [m]\Psi
    =
    \begin{cases}
      -\widetilde\chi\bigl(
        \link_\Delta(m^\vee)
      \bigr),
      & m^\vee\in\Delta,    \\
      0, & m^\vee\notin\Delta.
    \end{cases}
  \]
\end{proposition}

\begin{proof}
  Let \(Y_m^b\) be the union of the bounded common-refinement cells contained in \(Y_m\).  If \(Y_m=\varnothing\), then also \(Y_m^b=\varnothing\).  Suppose now that \(Y_m\neq\varnothing\).  By \cref{lem:k-fold-Ym}, \(Y_m\) is closed and is a union of cells in the finite common refinement of \(w_j+\Sigma_E\), for \(j\in[k]\).  Hence it is a finite subcomplex of that common refinement.  Every nonempty closed cell
  \[
    D_{\bm{\mathcal F}}
    =
    \bigcap_{j=1}^k(w_j+\sigma_{\mathcal F_j})
  \]
  has recession cone
  \[
    \operatorname{rec}(D_{\bm{\mathcal F}})
    =
    \bigcap_{j=1}^k\sigma_{\mathcal F_j},
  \]
  which is pointed.  Hence every cell has a vertex, and that vertex is a bounded cell of the same subcomplex.  Thus \(Y_m^b\neq\varnothing\). Retracting the unbounded cells onto their boundaries in decreasing order of dimension, as in \cite[Proof of Lemma~3.4]{elias2025categoricalvaluative}, gives a deformation retraction
  \[
    Y_m\longrightarrow Y_m^b.
  \]
  Since \(Y_m\) is tropically convex, it is contractible by \cite[Theorem~2]{develin2004tropical}.  Hence
  \[
    \chi(Y_m^b)
    =
    \mathbf 1_{Y_m\neq\varnothing}.
  \]
  The cells of \(Y_m^b\) are exactly the bounded \(D_{\bm{\mathcal F}}\) for which \(\lambda(\bm{\mathcal F})\mid m\).  Their cellular Euler characteristic and \eqref{eq:k-fold-fine-coefficient} therefore give
  \[
    \sum_{r\mid m}[r]\Psi
    =
    \mathbf 1_{Y_m\neq\varnothing}.
  \]
  Boolean M\"obius inversion yields
  \[
    \begin{aligned}
      [m]\Psi
      & =\sum_{r\mid m}(-1)^{|m|-|r|}
      \mathbf 1_{r^\vee\in\Delta}                                    \\
      & =\sum_{\substack{r'\mid m\\m^\vee r'\in\Delta}}(-1)^{|r'|},
    \end{aligned}
  \]
  where in the second line we set \(r'=m/r\).
  As in \cref{def:k-fold-external-activity-complex}, we identify square-free monomials with their supports.  Since \(\Delta\) is a simplicial complex, it is divisor-closed.  Thus, if \(m^\vee\notin\Delta\), the last indexing set is empty.  If \(m^\vee\in\Delta\), then \cref{def:cm-simplicial-complex} gives
  \[
    \link_\Delta(m^\vee)
    =
    \{r'\in\Delta:r'\cap m^\vee=\varnothing,
    \ m^\vee\cup r'\in\Delta\}.
  \]
  In square-free monomial notation, the disjointness condition is equivalent to \(r'\mid m\), since \(m=\bm X/m^\vee\) contains exactly the variables outside \(m^\vee\).  For such \(r'\), the union \(m^\vee\cup r'\) is represented by the product \(m^\vee r'\).  Moreover, \(m^\vee r'\in\Delta\) already implies \(r'\in\Delta\), because \(\Delta\) is divisor-closed.  Hence
  \[
    \link_\Delta(m^\vee)
    =
    \{r'\mid m:m^\vee r'\in\Delta\},
  \]
  which is precisely the last indexing set.  Moreover, \((-1)^{|r'|}=-(-1)^{\dim r'}\), and the empty face is included in the sum.  Thus the last sum is the negative reduced Euler characteristic of that link.
\end{proof}

The following proposition proves part~(1) of \cref{thm:B-CM-Delta} for loopless tuples.

\begin{proposition}\label{prop:k-fold-k-polynomial-loopless}
  If every \(\mathsf M_j\) is loopless, then
  \[
    \Psi_{\bm{\mathsf M}}(V_1,\ldots,V_k;\bm t)
    =
    \left.
    \mathcal K_{\widetilde E}(\Delta;\bm z)
    \right|_{z_i^{(j)}=V_jt_i}.
  \]
  Consequently,
  \[
    \mathcal K(\Delta;V_1,\ldots,V_k)
    =\mathcal K_{\bm{\mathsf M}}(V_1,\ldots,V_k),
  \]
  where the left-hand side is the \(\Z^k\)-graded \(K\)-polynomial for
  \(\deg x_i^{(j)}=e_j\).  This polynomial is multivaluative in
  \(\mathsf M_1,\ldots,\mathsf M_k\).
\end{proposition}

\begin{proof}
  The coefficient calculation in the proof of \cref{lem:theorem-b-implies-a} shows that, after setting \(z_i^{(j)}=1-x_i^{(j)}\), the coefficient of a square-free monomial \(m\mid\bm X\) in \(\mathcal K_{\widetilde E}(\Delta;\bm z)\) is
  \[
    \begin{cases}
      -\widetilde\chi\bigl(\link_\Delta(m^\vee)\bigr),
      &m^\vee\in\Delta,\\
      0,&m^\vee\notin\Delta.
    \end{cases}
  \]
  By \cref{prop:k-fold-link-coefficient}, these are exactly the fine coefficients \([m]\Psi\).  This proves the first identity after the specialization \(z_i^{(j)}=V_jt_i\).

  Setting \(t_i=1\) gives the \(\Z^k\)-graded identity.  Finally, the assignment \(\mathsf{M}\to\lambda_{-V}\mathcal{Q}_\mathsf{M}^\vee\) is valuative \cite[Proposition~5.8]{BergetEurSpinkTseng2023}. This proves valuativity in each \(\mathsf M_j\).
\end{proof}

\begin{remark}
  When \(k=2\), such coefficients are also considered in
  \cite[Proposition~2.4]{berget2025externalactivitycomplexpair}, where they
  are interpreted as finely graded Betti numbers.  We refer the interested
  reader to their paper for details.
\end{remark}

Part~(2) of \cref{thm:B-CM-Delta} for loopless tuples is the following proposition, proved in \cref{sec:k-fold-point-complex-cm}.

\begin{proposition}\label{prop:k-fold-point-complex-cm}
  The simplicial complex \(\Delta\) is Cohen--Macaulay over \(\k\) of
  dimension
  \[
    \dim\Delta=n+\rho(E)-1.
  \]
\end{proposition}

\begin{remark}
  When \(k=2\), \(\rho(E)=\rk D\), so the dimension and Cohen--Macaulay assertions agree with \cite[Proposition~4.19 and Theorem~7.2] {berget2025externalactivitycomplexpair}.
\end{remark}

\section{Cohen--Macaulayness of \texorpdfstring{\(\Delta\)}{Delta}}
\label{sec:k-fold-point-complex-cm}

In this section, we prove \cref{prop:k-fold-point-complex-cm}. We retain the notation and the looplessness assumption from the preceding section.

Recall that \(\widetilde E=[k]\times E\), and set
\[
  \mathsf N=\bigoplus_{j=1}^k\mathsf M_j,
  \qquad
  S_i=[k]\times\{i\}.
\]
For every nonempty \(A\subseteq\widetilde E\), write \(\Delta_A=\operatorname{conv}\{e_v:v\in A\}\), and set \(P=\sum_{i\in E}\Delta_{S_i}\). We identify \(v=(j,i)\in\widetilde E\) with the vertex \(x_i^{(j)}\), and write \(x_v=x_i^{(j)}\).
For \(F\subseteq\widetilde E\), write \(F^{(j)}=\{i:(j,i)\in F\}\).

As a Minkowski sum of coordinate simplices, \(P\) is a generalized permutohedron, so its normal fan coarsens \(\Sigma_{\widetilde E}\). The toric variety of this normal fan is
\[
  X_P\simeq(\PP^{k-1})^n,
\]
and the resulting toric morphism is
\[
  f_P:X_{\widetilde E}\longrightarrow X_P.
\]

Let
\[
  W=\operatorname{ind}_{\widetilde w}\mathsf N\subseteq X_{\widetilde E},
  \qquad
  Y=f_P(W)\subseteq X_P,
\]
where \(Y\) is the scheme-theoretic image.  By \cref{thm:efl-cm}(1), \(W\) is Cohen--Macaulay, geometrically connected, and geometrically reduced.  Since \(P\) is a generalized permutohedron, \cref{thm:efl-cm}(2) applies and shows that the natural map
\[
  \mathcal O_Y\longrightarrow Rf_{P*}\mathcal O_W
\]
is an isomorphism, and \(Y\) is Cohen--Macaulay.  In particular, \(\mathcal O_Y\simeq f_{P*}\mathcal O_W\) is a sheaf of reduced rings, so \(Y\) is also reduced.
Let
\[
  \delta:N\hookrightarrow N_{\widetilde E},
  \qquad
  \delta(u)_{(j,i)}=u_i,
\]
be the diagonal embedding, where the displayed formula is taken on representatives.  It is well-defined because replacing \(u\) by \(u+c\mathbf 1_E\) replaces \(\delta(u)\) by \(\delta(u)+c\mathbf 1_{\widetilde E}\).  If \(N_P\) denotes the cocharacter lattice of the dense torus of \(X_P\), then \(f_P\) is induced by the quotient map
\[
  q:N_{\widetilde E}\twoheadrightarrow N_P
  \simeq N_{\widetilde E}/\delta(N).
\]
We use the same symbols for the induced maps on their real extensions. Under the natural identification \(N_P\simeq(\Z^k/\Z(1,1,\ldots,1))^E\), the map \(q\) is induced by the blockwise quotient
\[
  (\Z^k)^E\longrightarrow
  (\Z^k/\Z(1,1,\ldots,1))^E.
\]
Explicitly, taking coordinates in any representative,
\[
  \ker(q)
  =\operatorname{im}(\delta)
  =
  \left\{
    y\in N_{\widetilde E,\R}:
    y_{(j,i)}=y_{(j',i)}
    \text{ for all }j,j'\in[k]\text{ and }i\in E
  \right\}.
\]

\begin{lemma}
  \label{lem:k-fold-projected-orbit-test}
  For every cone \(\tau\) of the normal fan of \(P\),
  \[
    V_P(\tau)\subseteq Y
    \quad\Longleftrightarrow\quad
    \tau\cap q\bigl(\widetilde w+|\Sigma_{\mathsf N}|\bigr)
    \neq\varnothing.
  \]
  Here \(V_P(\tau)\) means the orbit closure in \(X_P\).
\end{lemma}

\begin{proof}
  By \cref{lem:efl-orbit-criterion}, \(V(\sigma)\subseteq W\) if and only if \(\sigma\cap(\widetilde w+|\Sigma_{\mathsf N}|)\neq\varnothing\).
  Let \(\tau_\sigma\) be the smallest cone in the normal fan of \(P\) that contains \(q(\sigma)\).
  By \cite[Lemma~3.3.21]{cox2011toric}, \(f_P(V(\sigma))\subseteq V_P(\tau_\sigma)\).
  Since \(q\) is surjective, the induced map from the dense orbit of \(V(\sigma)\) to the dense orbit of \(V_P(\tau_\sigma)\) is dominant.
  The morphism \(f_P|_{V(\sigma)}:V(\sigma)\to V_P(\tau_\sigma)\) is both dominant and proper, hence surjective.
  We conclude that
  \(
    f_P(V(\sigma))=V_P(\tau_\sigma).
  \)

  We claim that \(V_P(\tau)\subseteq Y\) if and only if
  \[
    \sigma\cap\bigl(\widetilde w+|\Sigma_{\mathsf N}|\bigr)\neq\varnothing
    \text{ and }q(\sigma)\subseteq\tau
    \text{ for some }\sigma.
  \]
  Indeed, if \(q(\sigma)\subseteq\tau\) and \(V(\sigma)\subseteq W\), then
  \(\tau_\sigma\subseteq\tau\), and hence
  \[
    V_P(\tau)\subseteq V_P(\tau_\sigma)
    =f_P(V(\sigma))\subseteq Y.
  \]
  Conversely, if \(V_P(\tau)\subseteq Y\), then \cref{def:tropical-initial-degeneration} realizes \(W\) as a torus-invariant union of orbit closures.  Hence \(W\cap f_P^{-1}(V_P(\tau))\) is again such a union, and
  \[
    V_P(\tau)=f_P\bigl(W\cap f_P^{-1}(V_P(\tau))\bigr)=\bigcup_{\sigma:V(\sigma)\subset W\cap f_P^{-1}(V_P(\tau))}V_P(\tau_\sigma).
  \]
  Since \(V_P(\tau)\) is irreducible, \(V_P(\tau)=V_P(\tau_\sigma)\), hence \(\tau=\tau_\sigma\) for some such \(\sigma\).

  Since \(q\) is compatible with the two fans,
  \[
    q^{-1}(\tau)=\bigcup_{q(\sigma)\subseteq\tau}\sigma.
  \]
  Indeed, for \(x\in q^{-1}(\tau)\), let \(\sigma\in\Sigma_{\widetilde E}\) be the unique cone whose relative interior contains \(x\), and again denote by \(\tau_\sigma\) the smallest target cone containing \(q(\sigma)\). By the minimality of \(\tau_\sigma\), the image of \(\sigma^\circ\) lies in \(\tau_\sigma^\circ\), so \(q(x)\in\tau\cap\tau_\sigma^\circ\).  Since two cones of a fan meet along a common face, this implies \(\tau_\sigma\subseteq\tau\), and hence \(q(\sigma)\subseteq\tau\).  The reverse inclusion in the displayed equality is immediate.
  The preceding equivalence is therefore the same as
  \[
    q^{-1}(\tau)
    \cap\bigl(\widetilde w+|\Sigma_{\mathsf N}|\bigr)
    \neq\varnothing,
  \]
  or, equivalently, the condition in the statement.
\end{proof}

For \(F\subseteq\widetilde E\) such that \(F\cap S_i\neq\varnothing\) for every \(i\in E\), let \(Z_F\subseteq(\PP^{k-1})^n\) be the product of the coordinate subspaces spanned by the sets \(F\cap S_i\).  If \(\tau_F\) is the inner normal cone of \(\sum_i\Delta_{F\cap S_i}\), then \(Z_F=V_P(\tau_F)\). Give \(\k[\Delta]\) the \(\Z^E\)-grading \(\deg x_i^{(j)}=e_i\).  We write \(\operatorname{MultiProj}\k[\Delta]\) for the resulting multigraded Stanley--Reisner subscheme of \(X_P\).  Since the defining ideal of \(\k[\Delta]\) is monomial, \(\operatorname{MultiProj}\k[\Delta]\) is torus-invariant.

\begin{proposition} \label{prop:k-fold-projected-tropical-model}
  For every \(F\subseteq\widetilde E\) meeting each \(S_i\), one has
  \[
    F\in\Delta
    \quad\Longleftrightarrow\quad
    \tau_F\cap q\bigl(\widetilde w+|\Sigma_{\mathsf N}|\bigr)
    \neq\varnothing
    \quad\Longleftrightarrow\quad
    Z_F\subseteq Y.
  \]
  Every facet \(G\) of \(\Delta\) satisfies \(G\cap S_i\neq\varnothing\) for every \(i\in E\), and
  \begin{equation}\label{eq:k-fold-efl-identification}
    \operatorname{MultiProj}\k[\Delta]=Y.
  \end{equation}
\end{proposition}

\begin{proof}[Proof of \cref{prop:k-fold-projected-tropical-model}]
  Since initial matroids commute with direct sums,
  \[
    \operatorname{in}_{\delta(u)-\widetilde w}\mathsf N
    =
    \bigoplus_{j=1}^k\operatorname{in}_{u-w_j}\mathsf M_j.
  \]
  Hence the definition of \(\Delta\), followed by \cref{lem:relaxed-bergman-completion} applied once to \(\mathsf N\), gives
  \[
    \begin{aligned}
      F\in\Delta
      & \quad\Longleftrightarrow\quad
      \Loop\bigl(
        \operatorname{in}_{\delta(u)-\widetilde w}\mathsf N
      \bigr)\subseteq \widetilde{E}\setminus F
      \text{ for some }u\in N_\R       \\
      & \quad\Longleftrightarrow\quad
      \bigl(
        \delta(u)-\widetilde w+C_{\widetilde E\setminus F}
      \bigr)\cap|\Sigma_{\mathsf N}|\neq\varnothing
      \text{ for some }u\in N_\R.
    \end{aligned}
  \]
  After translating the last intersection by \(\widetilde w\), we see that its nonemptiness is equivalent to
  \[
    \bigl(\widetilde w+|\Sigma_{\mathsf N}|\bigr)
    \cap
    \bigl(\delta(u)+C_{\widetilde E\setminus F}\bigr)
    \neq\varnothing
    \quad\text{for some }u\in N_\R.
  \]
  Since \(\operatorname{im}(\delta)=\ker(q)\), this is precisely
  \begin{equation}\label{eq:k-fold-projected-face-test}
    F\in\Delta
    \quad\Longleftrightarrow\quad
    \bigl(\widetilde w+|\Sigma_{\mathsf N}|\bigr)
    \cap
    \bigl(\ker(q)+C_{\widetilde E\setminus F}\bigr)
    \neq\varnothing.
  \end{equation}

  If \(F\cap S_i\neq\varnothing\) for every \(i\), then \(Z_F\) is the product of the coordinate linear subspaces spanned by the sets \(F\cap S_i\). Thus one expects the rays of \(\tau_F\) to be indexed by the complementary coordinates \(\widetilde E\setminus F\). Formally, on the \(i\)-th simplex, the inner normal cone of \(\Delta_{F\cap S_i}\) is generated by the coordinate rays indexed by \(S_i\setminus(F\cap S_i)\). Since the normal fan of \(P\) is the product of the simplex normal fans,
  \[
    \tau_F=q\bigl(C_{\widetilde E\setminus F}\bigr).
  \]
  The blockwise description of \(q\) therefore gives
  \begin{equation}\label{eq:k-fold-target-cone-preimage}
    q^{-1}(\tau_F)=\ker(q)+C_{\widetilde E\setminus F}.
  \end{equation}
  Thus \eqref{eq:k-fold-projected-face-test} gives the first equivalence, and \cref{lem:k-fold-projected-orbit-test} gives the second.

  For arbitrary \(F\in\Delta\), we next show that \(F\) is contained in a larger face \(G\) such that \(G\cap S_i\neq\varnothing\) for every \(i\). Choose a point \(y\) in the intersection in \eqref{eq:k-fold-projected-face-test}.  Write \[y=\delta(u)+d\] with \(d\in C_{\widetilde E\setminus F}\), choose a lift
  \[
    d=\sum_{v\in\widetilde E\setminus F}d_ve_v,
    \qquad d_v\geq0,
  \]
  and choose compatible representatives of \(u\) and \(y\).
  Define \(G\subseteq\widetilde E\) by
  \[
    G\cap S_i
    =
    \left\{v\in S_i:y_v=\min_{v'\in S_i}y_{v'}\right\}
    \qquad(i\in E).
  \]
  For \(v\in F\cap S_i\), one has \(d_v=0\), whereas \(d_{v'}\geq0\) for every \(v'\in S_i\).  Hence every element of \(F\cap S_i\) minimizes \(y\) on \(S_i\), so \(F\subseteq G\); each \(G\cap S_i\) is nonempty by definition.  By definition of the normal cone, the face of \(P\) minimized by \(q(y)\) is \(\sum_i\Delta_{G\cap S_i}\), so \(y\in q^{-1}(\tau_G)\), \(q(y)\in q(\widetilde{w}+|\Sigma_\mathsf{N}|)\cap \tau_G\).  Now the established first equivalence gives \(G\in\Delta\).  In particular, every facet meets every \(S_i\).

  Both \(Y\) and \(\operatorname{MultiProj}\k[\Delta]\) are torus-invariant closed subschemes of \(X_P\).  By the definition of the Stanley--Reisner ideal, for every \(F\) meeting each \(S_i\),
  \[
    V_P(\tau_F)\subseteq\operatorname{MultiProj}\k[\Delta]
    \quad\Longleftrightarrow\quad
    F\in\Delta.
  \]
  Together with the preceding equivalences, this shows that the two schemes contain exactly the same orbit closures and hence have the same underlying closed set.  Since the Stanley--Reisner ideal of \(\Delta\) is square-free, \(\operatorname{MultiProj}\k[\Delta]\) is reduced.  Since \(Y\) is also
  reduced, the statement follows.
  \iffalse
  Let \(I_\Delta\) be the Stanley--Reisner
  ideal, and let
  \[
    \mathfrak B
    =
    \prod_{i\in E}(x_i^{(1)},\ldots,x_i^{(k)})
  \]
  be the Cox irrelevant ideal.  We claim that
  \(I_\Delta:\mathfrak B^\infty=I_\Delta\).  Both ideals are monomial, so it
  is enough to show that every monomial outside \(I_\Delta\) also lies
  outside \(I_\Delta:\mathfrak B^\infty\).

  Let \(m\notin I_\Delta\) be a monomial with support \(F\).  By the
  non-face description of a Stanley--Reisner ideal, this means precisely
  that \(F\) is a face of \(\Delta\).  Extend \(F\) to a face \(G\) meeting
  every \(S_i\), and choose \(v_i\in G\cap S_i\).  Then
  \[
    m\left(\prod_{i\in E}x_{v_i}\right)^r\notin I_\Delta
    \qquad(r\geq1),
  \]
  because its support is contained in the face \(G\), and hence is again a
  face of \(\Delta\).  Since
  \(\prod_i x_{v_i}\in\mathfrak B\), this shows that
  \(\mathfrak B^r m\nsubseteq I_\Delta\) for every \(r\geq1\).  Hence
  \(m\notin I_\Delta:\mathfrak B^\infty\), proving the claim.  Finally,
  \(I_\Delta\) is square-free, so both sides of
  \eqref{eq:k-fold-efl-identification} are reduced and the scheme equality
  follows.
  \fi
\end{proof}

\begin{proposition}
  \label{prop:k-fold-point-complex-geometry}
  The scheme \(Y\) is Cohen--Macaulay, and
  \(\Delta\) is pure of dimension
  \[
    \dim\Delta=n+\rho(E)-1.
  \]
\end{proposition}

\begin{proof}
  By \cref{thm:efl-cm}(1), \(W\) is connected, so its image \(Y\) is connected. The fact that \(Y\) is Cohen--Macaulay was already established at the beginning of this section. We will deduce the dimension by calculating \(\chi(Y,\mathcal{O}_Y(d_1,\ldots,d_n))\) for our \(Y\subseteq (\mathbb{P}^{k-1})^n\)

  Recall that \(W=\operatorname{ind}_{\widetilde{w}}\mathsf{N}\subseteq X_{\widetilde{E}}\).
  Let \(g:W\to Y\) be the restriction of \(f_P\).  For every nonempty \(S\subseteq\widetilde E\), let \(\mathcal L_S\) be the toric line bundle on \(X_{\widetilde E}\) associated with the simplex \(\Delta_S\).  By construction of \(f_P\),
  \[
    g^*\mathcal O_Y(d_1,\ldots,d_n)
    =
    \left.\bigotimes_{i\in E}\mathcal L_{S_i}^{\otimes d_i}\right|_W.
  \]
  \Cref{thm:efl-cm}(2) gives \(Rg_*\mathcal O_W\simeq\mathcal O_Y\).  The projection formula, followed by \cref{prop:efl-euler-characteristic}, gives
  \[
    \chi\bigl(Y,\mathcal O_Y(d_1,\ldots,d_n)\bigr)
    =
    \chi\left(
      W,
      \left.\bigotimes_{i\in E}\mathcal L_{S_i}^{\otimes d_i}
      \right|_W
    \right)
    =
    \chi\left(
      \mathsf N,
      \bigotimes_{i\in E}\mathcal L_{S_i}^{\otimes d_i}
    \right).
  \]
  Proposition~2.15 of \cite{eur2026ktheoreticpositivity} states
  \[
    \chi\left(
      \mathsf N,
      \bigotimes_{\varnothing\neq S\subseteq\widetilde E}
      \mathcal L_S^{\otimes t_S}
    \right)
    =
    \sum_{\bm k\text{ satisfies the dragon Hall--Rado condition}}
    \bm t^{(\bm k)}.
  \]
  Here
  \[
    t^{(r)}=\binom{t+r-1}{r},
    \qquad
    \bm t^{(\bm k)}
    =
    \prod_{\varnothing\neq S\subseteq\widetilde E}t_S^{(k_S)},
  \]
  where \(r\in\Z_{\geq0}\) and \(\bm k=(k_S)_{\varnothing\neq S\subseteq\widetilde E}\).  The condition on \(\bm k\) means that the sequence in which each \(S\) occurs \(k_S\) times satisfies the dragon Hall--Rado condition.
  Equivalently,
  \[
    \rk_{\mathsf N}\left(\bigcup_{S\in\mathcal A}S\right)
    \geq
    1+\sum_{S\in\mathcal A}k_S
  \]
  for every nonempty subcollection \(\mathcal A\subseteq2^{\widetilde E}\setminus\{\varnothing\}\).

  We apply this formula by setting
  \[
    t_S=
    \begin{cases}
      d_i, & \text{ if }S=S_i, i\in E, \\
      0,   & \text{otherwise.}
    \end{cases}
  \]
  If \(S\notin\{S_i:i\in E\}\), then
  \[t_S^{(k_S)}=\binom{k_S-1}{k_S}=
    \begin{cases}
      0,& k_S> 0,\\
      1,& k_S=0.
  \end{cases}\]
  Thus, if \(\bm t^{(\bm k)}\neq 0\), then \(\bm k\) is supported on \(\{S_i:i\in E\}\).  To simplify the notation, write \(u=(u_i)_{i\in E}\in\Z_{\geq0}^E\), where \(u_i=k_{S_i}\).  The specialized formula is
  \begin{equation}\label{eq:k-fold-snapper}
    \chi\bigl(Y,\mathcal O_Y(d_1,\ldots,d_n)\bigr)
    =
    \sum_{\substack{u\in\Z_{\geq0}^E\\
    }}
    \prod_{i\in E}\binom{d_i+u_i-1}{u_i},
  \end{equation}
  where in the summation \(u\) satisfies the dragon Hall--Rado condition for \(\mathsf N\) and \((S_i)_{i\in E}\).
  We claim that this is equivalent to saying
  \[
    u(A)\leq\rho(A)
    \qquad\text{for every }A\subseteq E.
  \]
  Equivalently, these vectors are the integral points of the polymatroid polytope \(I(\rho)\), defined in \cref{subsec:matroids-polymatroids}.

  To prove the claim, fix \(u\in\Z_{\geq0}^E\) and set \(f(A)=\sum_{j=1}^k\rk_{\mathsf M_j}(A)-1\), as in \cref{prop:rho-is-polymatroid}. It is enough to test the dragon Hall--Rado inequalities on collections of the form \(\{S_i:i\in A\}\).  Indeed, for any nonempty collection \(\mathcal A\), set \(A=\{i\in E:S_i\in\mathcal A\}\).
  Then, since \(k_S=0\) for \(S\neq S_i\),
  \[
    \sum_{S\in\mathcal A}k_S=\sum_{i\in A}k_{S_i} = u(A),
    \qquad
    \rk_{\mathsf N}\left(\bigcup_{S\in\mathcal A}S\right)
    \geq
    \rk_{\mathsf N}\left(\bigcup_{i\in A}S_i\right).
  \]
  If \(A\neq\varnothing\), these relations show that the dragon Hall--Rado condition for \(\{S_i:i\in A\}\) implies that condition for \(\mathcal A\).  If \(A=\varnothing\), then \(u(A)=0\); the required lower bound is one, and the inequality is automatic because \(\mathsf N\) is loopless. Therefore the dragon Hall--Rado condition is equivalent to
  \[
    u(A)
    \leq
    \rk_{\mathsf N}\left(\bigcup_{i\in A}S_i\right)-1
    =\sum_{j=1}^k\rk_{\mathsf M_j}(A)-1
    =f(A)
    \qquad(\varnothing\neq A\subseteq E).
  \]

  Consider the function \(A\mapsto u(A)\). This function is modular and maps \(\varnothing\) to \(0\). The maximality property of the Dilworth truncation \cite[Section~48.2]{schrijver2003combinatorial} says that \(\rho=\widehat f\) is the largest normalized submodular function bounded above by \(f\) on nonempty subsets.  Thus \(u(A)\leq f(A)\) for every nonempty \(A\) implies \(u(A)\leq\rho(A)\) for every \(A\subseteq E\).  Conversely, if \(u(A)\leq\rho(A)\) for every \(A\subseteq E\), then \(u(A)\leq f(A)\) for every nonempty \(A\), since \(\rho(A)\leq f(A)\).  This proves the claim.

  Finally, to derive the dimension of \(Y\) and \(\Delta\), consider the top-degree part of \eqref{eq:k-fold-snapper}.  The summand indexed by \(u\) has degree \(u(E)\leq\rho(E)\), so the top-degree terms are indexed by \(u(E)=\rho(E)\), namely by the set \(\mathcal B(\rho)\) of integer bases defined in \cref{subsec:matroids-polymatroids}.

  Since \(\rho\) is an integral polymatroid, \(\mathcal B(\rho)\) is nonempty.  Therefore, the polynomial in \eqref{eq:k-fold-snapper} has degree \(\rho(E)\), with top-degree part
  \begin{equation}\label{eq:k-fold-snapper-top-degree}
    \sum_{b\in\mathcal B(\rho)}
    \prod_{i\in E}\frac{d_i^{b_i}}{b_i!}.
  \end{equation}
  Since \(\mathcal O_{X_P}(1,\ldots,1)\) is ample, this total degree is \(\dim Y\).  Hence \(Y\) has dimension \(\rho(E)\).  Because \(Y\) is connected and Cohen--Macaulay, it is pure-dimensional.  Moreover, \(Y\) is the union of the \(Z_F\), so its irreducible components are
  \[
    Z_F\simeq\prod_{i\in E}\PP^{|F\cap S_i|-1},
  \]
  indexed by the facets \(F\) of \(\Delta\).  Purity gives \(|F|-n=\rho(E)\), hence \(\dim\Delta=n+\rho(E)-1\).
\end{proof}

\begin{remark}\label{rmk:facet-basis-bijection}
  The same computation shows that
  \[
    F\longmapsto
    \bigl(|F\cap S_i|-1\bigr)_{i\in E}
  \]
  is a bijection from the facets of \(\Delta\) to \(\mathcal B(\rho)\).
  Indeed, because \(Y\) is reduced and pure, its top-degree part is the sum of the top-degree parts of its components.  The component \(Z_F\) contributes
  \[
    \prod_{i\in E}
    \frac{d_i^{|F\cap S_i|-1}}{(|F\cap S_i|-1)!}
  \]
  to the top-degree part. Comparing these component contributions with \eqref{eq:k-fold-snapper-top-degree} gives
  \[
    \#\left\{
      F\text{ facet}:|F\cap S_i|-1=b_i\text{ for all }i
    \right\}
    =
    \begin{cases}
      1, & b\in\mathcal B(\rho),    \\
      0, & b\notin\mathcal B(\rho).
    \end{cases}
  \]

  When \(k=2\), \(\rho\) is the rank function of \(D(\mathsf M_1,\mathsf M_2)\), so \(\mathcal B(\rho)\) consists of the indicator vectors of its bases.  The bijection above sends a facet \(F\) to
  \[
    B_F=\{i\in E:|F\cap S_i|=2\}.
  \]
  In the notation of Berget--Fink, \(B_F\) is the set of indices \(i\) for which both \(x_i\) and \(y_i\) occur in \(F\).  Thus the bijection above recovers their indexing of facets by bases.  Their definition \cite[Definition~4.18]{berget2025externalactivitycomplexpair} additionally describes the facet indexed by \(B\): for each \(i\notin B\), external activity determines whether \(x_i\) or \(y_i\) occurs.
\end{remark}

We are now in a position to prove \cref{prop:k-fold-point-complex-cm}.
\begin{proof}[Proof of \cref{prop:k-fold-point-complex-cm}]
  The dimension assertion follows from \cref{prop:k-fold-point-complex-geometry}.  It remains to prove that \(\Delta\) is Cohen--Macaulay.

  Fix \(F\in\Delta\), and set \(L=\link_\Delta(F)\).  First suppose that \(F\) meets every \(S_i\).  We claim that \(L\) is Cohen--Macaulay. The key input is the Cohen--Macaulayness of \(Y=\operatorname{MultiProj}\k[\Delta]\).

  Let \(U_{\tau_F}\subseteq X_P\) be the affine toric chart on which every variable indexed by \(F\) is nonzero.  Choose one \(b_i\in F\cap S_i\) for every \(i\), and set
  \[
    z_v=\frac{x_v}{x_{b_i}}
    \quad(v\in S_i\setminus F),
    \qquad
    t_f=\frac{x_f}{x_{b_i}}
    \quad(f\in F\cap S_i\setminus\{b_i\}).
  \]
  For \(J\subseteq\widetilde E\setminus F\), the monomial \(\prod_{v\in J}z_v\) lies in the localized Stanley--Reisner ideal exactly when \(F\cup J\notin\Delta\), or equivalently when \(J\notin L\).  Thus
  \[
    \k[L]
    \simeq
    \k[z_v:v\in\widetilde E\setminus F]
    /(\text{nonface monomials of }L).
  \]
  Hence
  \[
    \Gamma(Y\cap U_{\tau_F},\mathcal O_Y)
    \simeq
    \k[L][t_f^{\pm1}:f\in F\setminus\{b_i:i\in E\}].
  \]
  Consequently,
  \[
    Y\cap U_{\tau_F}
    \simeq
    \operatorname{Spec}\k[L]\times(\mathbb G_m)^{|F|-n}.
  \]
  The left-hand side is Cohen--Macaulay because it is open in \(Y\).  Since Laurent polynomial extension reflects Cohen--Macaulayness, \(\k[L]\) is Cohen--Macaulay, so \(L\) is Cohen--Macaulay by \cref{def:cm-simplicial-complex}.

  Now assume \(F\cap S_i=\varnothing\) for some \(i\).  Since every facet meets every \(S_i\) by \cref{prop:k-fold-projected-tropical-model}, the face \(F\) is not a facet, so \(L\) has at least one vertex.  We claim that \(L\) is contractible.

  For \(G\subseteq\widetilde E\), identify \(G\) with the corresponding square-free divisor of \(\bm X\), as in \cref{def:k-fold-external-activity-complex}, and set
  \[
    \mathcal U_G=Y_{G^\vee}.
  \]
  We use \(\mathcal U\) instead of \(Y\) to avoid conflict with the scheme \(Y=f_P(W)\) already fixed in this section.  The description of \(Y_m\) in the proof of \cref{lem:k-fold-Ym} gives
  \[
    \mathcal U_G =
    \left\{
      u\in N_\R:
      G^{(j)}\cap
      \Loop\bigl(\operatorname{in}_{u-w_j}\mathsf M_j\bigr)=\varnothing
      \text{ for every }j\in[k]
    \right\}.
  \]
  By its definition, \(\mathcal U_G\) is a union of cells of the common polyhedral refinement.  Moreover, \cref{lem:k-fold-Ym} shows that it is closed and tropically convex and that
  \[
    \mathcal U_G\neq\varnothing
    \quad\Longleftrightarrow\quad
    G\in\Delta.
  \]
  When nonempty, it is therefore contractible by \cite[Theorem~2]{develin2004tropical}.
  Let
  \[
    \mathcal U_{>F}
    =
    \bigcup_{v\in\widetilde E\setminus F}\mathcal U_{F\cup\{v\}}.
  \]
  This set is nonempty because \(L\) has a vertex and the displayed equivalence holds. By the cell description above, \(\mathcal U_{>F}\) is a triangulable polyhedral set.  The nonempty sets \(\mathcal U_{F\cup\{v\}}\), indexed by the vertices \(v\) of \(L\), form a finite closed cover of \(\mathcal U_{>F}\).  For every nonempty set \(J\) of vertices of \(L\),
  \[
    \bigcap_{v\in J}\mathcal U_{F\cup\{v\}}
    =\mathcal U_{F\cup J},
  \]
  which is nonempty and contractible exactly when \(J\in L\).  Thus the nerve of this cover is \(L\), and the nerve theorem \cite[Theorem~10.7]{bjorner1995topological} gives
  \[
    |\link_\Delta(F)|\simeq\mathcal U_{>F}.
  \]
  It remains to prove that \(\mathcal U_{>F}\) is contractible.
  For \(i\in E\), set
  \[
    \mathcal A_i
    =\left\{
      u\in N_\R:
      i\notin\Loop\bigl(\operatorname{in}_{u-w_j}\mathsf M_j\bigr)
      \text{ for some }j\in[k]
    \right\}.
  \]
  Choose \(i\in E\) such that \(F\cap S_i=\varnothing\).  Then
  \begin{equation}\label{eq:k-fold-retraction-target}
    \mathcal U_F\cap\mathcal A_i
    =
    \bigcup_{v\in S_i}\mathcal U_{F\cup\{v\}}
    \subseteq\mathcal U_{>F}.
  \end{equation}
  By \cref{lem:k-fold-first-wall-retraction} below, it is a strong deformation retract of both \(\mathcal U_F\) and \(\mathcal U_{>F}\).  Since \(\mathcal U_F\) is contractible, this proves the claim and hence the contractibility of \(\link_\Delta(F)\).

  Thus, for every face \(F\) of \(\Delta\), the link \(L=\link_\Delta(F)\) is Cohen--Macaulay if \(F\) meets every \(S_i\), and is contractible otherwise.  Applying \cref{thm:reisner-criterion} to \(L\) in the first case, and using contractibility in the second, give
  \[
    \widetilde H_q(L;\k)=0
    \qquad\text{for all }q<\dim L.
  \]
  A second application of \cref{thm:reisner-criterion} now shows that \(\Delta\) is Cohen--Macaulay.
\end{proof}

\begin{lemma}\label{lem:k-fold-first-wall-retraction}
  With the notation of the preceding proof, suppose that \(F\cap S_i=\varnothing\).  Then \(\mathcal U_F\cap\mathcal A_i\) is a strong deformation retract of both \(\mathcal U_F\) and \(\mathcal U_{>F}\).
\end{lemma}

\begin{proof}
  If \(i\) is a coloop of \(\mathsf M_j\) for some \(j\in[k]\), then it is a nonloop in every initial matroid of \(\mathsf M_j\).  Hence \(\mathcal A_i=N_\R\) and \(\mathcal U_{F\cup\{x_i^{(j)}\}}=\mathcal U_F\).  Since \(F\cap S_i=\varnothing\), this set occurs in the union defining \(\mathcal U_{>F}\).  On the other hand, \(\mathcal U_{>F}\subseteq\mathcal U_F\).  It follows that \(\mathcal U_F\cap\mathcal A_i=\mathcal U_F=\mathcal U_{>F}\).
  We may therefore assume that \(i\) is not a coloop of any \(\mathsf M_j\).

  The homotopy will move in the \(\bar e_i\)-direction and stop when \(i\) first becomes a nonloop in one of the initial matroids.  Fix \(u\in\mathcal U_F\) and choose a lift to \(\R^E\), denoted again by \(u\).  All basis weights below are computed using this lift and the block lifts \(w_j\) supplied by \(\bm w\in\R^{\widetilde E}\).  Replacing \(u\) by \(u+c\mathbf 1_E\) adds \(c\rk(\mathsf M_j)\) to the weight of every basis of \(\mathsf M_j\).  Thus, for every \(j\), the difference between the two maxima below is independent of the lift, and hence so is the stopping time defined from these differences.

  Since each \(\mathsf M_j\) is loopless and \(i\) is not a coloop, each \(\mathsf M_j\) has bases both containing and omitting \(i\).
  For \(j\in[k]\), set
  \[
    \alpha_j(u) = \max_{\substack{D\in\mathcal B(\mathsf M_j)\\i\notin D}} (u-w_j)(D),
    \qquad
    \beta_j(u) = \max_{\substack{D\in\mathcal B(\mathsf M_j)\\i\in D}} (u-w_j)(D).
  \]
  Thus \(\alpha_j(u)\) and \(\beta_j(u)\) are the largest weights of bases omitting and containing \(i\), respectively.  Along the ray \(u+s\bar e_i\), represented on lifts by \(u+se_i\), one has
  \[
    \max_{\substack{D\in\mathcal B(\mathsf M_j)\\i\notin D}}
    (u+se_i-w_j)(D)=\alpha_j(u),
    \qquad
    \max_{\substack{D\in\mathcal B(\mathsf M_j)\\i\in D}}
    (u+se_i-w_j)(D)=\beta_j(u)+s.
  \]
  An element is a nonloop of an initial matroid exactly when it belongs to some maximum-weight basis.  Therefore, \(i\) is a nonloop of \(\operatorname{in}_{u+s\bar e_i-w_j}\mathsf M_j\) exactly when \(\beta_j(u)+s\geq\alpha_j(u)\).  Write \(r_+=\max\{r,0\}\), and define
  \[
    \tau_i(u)
    =\min_{j\in[k]}\bigl(\alpha_j(u)-\beta_j(u)\bigr)_+.
  \]
  The preceding comparison shows that
  \[
    \tau_i(u)
    =\min\{s\geq0:u+s\bar e_i\in\mathcal A_i\}.
  \]

  We next check that no existing nonloop disappears before this first wall.
  There is nothing to prove if \(\tau_i(u)=0\), so suppose that \(\tau_i(u)>0\).  Then \(u\notin\mathcal A_i\), so, for each \(j\), every \((u-w_j)\)-maximizing basis of \(\mathsf M_j\) omits \(i\).  Fix \(j\), and let \(a\) be a nonloop of \(\operatorname{in}_{u-w_j}\mathsf M_j\).  Choose a \((u-w_j)\)-maximizing basis \(D\) containing \(a\).  The basis \(D\) omits \(i\), so its weight remains \(\alpha_j(u)\) along the ray.
  Moreover, for \(0\leq s\leq\tau_i(u)\),
  \[
    \beta_j(u)+s\leq\alpha_j(u).
  \]
  Hence \(D\) remains maximizing, and \(a\) remains a nonloop.  After passing to \(N_\R\), it follows that the segment from \(u\) to \(u+\tau_i(u)\bar e_i\) preserves every nonloop and, in particular, stays inside \(\mathcal U_F\).

  Each maximum above is piecewise linear and continuous on \(\R^E\). Taking differences, positive parts, and a finite minimum shows that the descended function \(\tau_i\) is continuous and piecewise linear on \(N_\R\).

  Returning to \(N_\R\), define
  \[
    H_t(u)=u+t\tau_i(u)\bar e_i,
    \qquad 0\leq t\leq1,
  \]
  for \(u\in\mathcal U_F\).  The preceding paragraphs show that this is a homotopy in \(\mathcal U_F\) and that \(H_1(u)\in\mathcal A_i\).  If \(u\in\mathcal A_i\), then \(\tau_i(u)=0\), so \(H_t(u)=u\) for every \(t\).  Hence \(H\) is a strong deformation retraction of \(\mathcal U_F\) onto \(\mathcal U_F\cap\mathcal A_i\).

  It remains to restrict the same homotopy to \(\mathcal U_{>F}\).  If \(u\in\mathcal U_{>F}\), then \(u\in\mathcal U_{F\cup\{v\}}\) for some \(v\in\widetilde E\setminus F\).  Since the homotopy preserves every nonloop,
  \[
    H_t(u)\in\mathcal U_{F\cup\{v\}}\subseteq\mathcal U_{>F}
    \qquad(0\leq t\leq1).
  \]
  Thus the endpoint lies in \(\mathcal U_{>F}\cap\mathcal A_i\).  Moreover,
  \[
    \mathcal U_{>F}\cap\mathcal A_i
    =\mathcal U_F\cap\mathcal A_i.
  \]
  One inclusion follows from \(\mathcal U_{>F}\subseteq\mathcal U_F\), and the other is \eqref{eq:k-fold-retraction-target}.  The same homotopy therefore gives a strong deformation retraction of \(\mathcal U_{>F}\) onto \(\mathcal U_F\cap\mathcal A_i\).
\end{proof}

\begin{proof}[Proof of \cref{thm:B-CM-Delta}]
  By \cref{lem:k-fold-loop-reduction}, we may assume that every \(\mathsf M_j\) is loopless.  The \(K\)-polynomial identity and multivaluativity are given by \cref{prop:k-fold-k-polynomial-loopless}, and the Cohen--Macaulay and dimension assertions are given by \cref{prop:k-fold-point-complex-cm}.
\end{proof}

\begin{proof}[Proof of \cref{thm:A-positivity-K-poly}]
  Combine \cref{thm:B-CM-Delta} with \cref{lem:theorem-b-implies-a}.
\end{proof}

\section{Chern monomials and diagonal Dilworth bases}\label{sec:Chern-Dilworth}
In this section, we describe an application of the main results proved in the preceding sections.

Let \(\bm{\mathsf M}=(\mathsf M_1,\ldots,\mathsf M_k)\) be a \(k\)-tuple of connected matroids on \(E=[n]\), where \(n\geq2\) and \(k\geq\max\{2,n-1\}\), and set \(r_j=\rk(\mathsf M_j)\). Let \(\rho\) be the polymatroid rank function of \(D(\bm{\mathsf M})\):
\[
  \rho(A)=\min_{A=A_1\sqcup\cdots\sqcup A_s}\sum_{q=1}^s\left(\sum_{j=1}^k\rk_{\mathsf M_j}(A_q)-1\right), \qquad \rho(\varnothing)=0.
\]

\begin{lemma}\label{lem:rank-connected-tuple}
  Under the preceding assumption, \(\rho(E)=\sum_{j=1}^k r_j-1\).
\end{lemma}
\begin{proof}
  Taking \(s=1\) and \(A_1=E\) gives \(\rho(E)\leq\sum_{j=1}^k r_j-1\). For the reverse inequality, let \(E=A_1\sqcup\cdots\sqcup A_s\) be a partition with \(s\geq2\). For every \(j\in[k]\), submodularity gives
  \[
    \sum_{q=1}^s \rk_{\mathsf M_j}(A_q)
    \geq \rk_{\mathsf M_j}(A_1)+\rk_{\mathsf M_j}(E\setminus A_1).
  \]
  In the notation of \cite[Proposition~4.2.1]{oxley2011matroid}, a subset \(T\subseteq E(\mathsf M)\) is a separator if and only if
  \[
    r_{\mathsf M}(T)+r_{\mathsf M}(E(\mathsf M)\setminus T)=r(\mathsf M).
  \]
  Applying this criterion with \(\mathsf M=\mathsf M_j\) and \(T=A_1\) and using the connectedness of \(\mathsf M_j\), we obtain
  \[
    \sum_{q=1}^s \rk_{\mathsf M_j}(A_q)\geq r_j+1.
  \]
  Since \(s\leq n\leq k+1\), every such partition satisfies
  \[
    \sum_{q=1}^s\left(\sum_{j=1}^k\rk_{\mathsf M_j}(A_q)-1\right)
    \geq \sum_{j=1}^k(r_j+1)-s
    \geq \sum_{j=1}^k r_j-1.
  \]
  Therefore the one-part partition attains the minimum value \(\sum_{j=1}^k r_j-1\).
\end{proof}

Recall that
\[
  \mathcal{B}(\rho)=\{b\in \mathbb{Z}^E_{\geq 0}:b(A)\leq \rho(A)\;\forall A\subseteq E,\ b(E)=\rho(E)\}
\]
is the set of integer bases of \(\rho\).

Choose a generic tuple of vectors \(\bm w\) as in \cref{subsec:external-activity-tuple}, and set \(\Delta=\Delta_{\bm w}(\bm{\mathsf{M}})\).
By \cref{rmk:facet-basis-bijection}, the map
\[
  F \mapsto (|F\cap S_i|-1)_{i\in E}
\]
is a bijection from the facets of \(\Delta\) to \(\mathcal{B}(\rho)\). We identify \(F\in \Delta\) with its vertex set as in \cref{def:k-fold-external-activity-complex}. Moreover, for \(F\in \Delta\), let
\[
  F^{(j)}=\{i\in E:(j,i)\in F\}.
\]

\thmDChernMonomials*

\begin{proof}
  If \(l_j>n-r_j\) for some \(j\), then \(c_{l_j}(\mathcal{Q}_{\mathsf M_j})=0\), since \(\mathcal{Q}_{\mathsf M_j}\) has rank \(n-r_j\). On the other hand, \(F^{(j)}\subseteq E\), so
  \[
    |F^{(j)}|\leq n<r_j+l_j
  \]
  for every facet \(F\) of \(\Delta\). Thus both sides of the claimed identity vanish. We may therefore assume that \(l_j\leq n-r_j\) for every \(j\in[k]\). Set \(d_j=n-r_j-l_j\geq0\). By \cref{lem:rank-connected-tuple},
  \[
    |\bm d|=kn-\sum_{j=1}^k r_j-|\bm l|=(k-1)n-\rho(E)=:c.
  \]
  Consider
  \[
    \mathcal{K}_{\bm{\mathsf{M}}}(1-U_1,\ldots,1-U_k)=\chi\left(X_E,\prod_{j=1}^k\lambda_{U_j-1}(\mathcal{Q}_{\mathsf M_j}^\vee)\right).
  \]
  Forgetting the \(H\)-action in \eqref{eq:k-fold-gw-expansion} and setting \(V_j=1-U_j\) gives
  \[
    [U_1^{d_1}\cdots U_k^{d_k}]\mathcal{K}_{\bm{\mathsf{M}}}(1-U_1,\ldots,1-U_k)=\sum_{m\mid\bm X}[m]\Psi,
  \]
  where the sum is over monomials \(m\mid\bm X\) satisfying \(|\operatorname{supp}_j(m)|=d_j\) for every \(j\). Every such monomial has \(|m|=|\bm d|=c\), and hence
  \[
    |m^\vee|=kn-|m|=kn-c=n+\rho(E)=\dim\Delta+1,
  \]
  where the last equality follows from \cref{prop:k-fold-point-complex-cm}. Consequently, if \(m^\vee\in\Delta\), then \(m^\vee\) is a facet of \(\Delta\). The link of a facet consists only of the empty face, so \cref{prop:k-fold-link-coefficient} gives
  \[
    [m]\Psi=
    \begin{cases}
      -\widetilde\chi(\{\varnothing\})=1,
      &m^\vee\in\Delta,\\
      0,&m^\vee\notin\Delta.
    \end{cases}
  \]
  Moreover, if \(F=m^\vee\) is a facet, then \(|\operatorname{supp}_j(m)|=n-|F^{(j)}|\). We conclude that
  \[
    [U_1^{d_1}\cdots U_k^{d_k}]\mathcal{K}_{\bm{\mathsf{M}}}(1-U_1,\ldots,1-U_k)=\#\{F\text{ facet of }\Delta:|F^{(j)}|=n-d_j\;\forall j\in [k]\}.
  \]
  On the other hand, the Hirzebruch--Riemann--Roch theorem gives
  \[
    \chi(X_E,\prod_{j=1}^k \lambda_{U_j-1}(\mathcal{Q}_{\mathsf M_j}^\vee))=\int_{X_E}\prod_{j=1}^k\ch(\lambda_{U_j-1}(\mathcal{Q}_{\mathsf M_j}^\vee))\td(X_E).
  \]
  Since \(\mathcal{Q}_{\mathsf M_j}\) has rank \(n-r_j\),
  \[
    [U^a]\ch(\lambda_{U-1}(\mathcal{Q}_{\mathsf M_j}^\vee))=c_{n-r_j-a}(\mathcal{Q}_{\mathsf M_j})+\text{terms of Chow degree greater than }n-r_j-a.
  \]
  This is \eqref{eq:lambda-chern-leading-term} with
  \(q_j=n-r_j\) and \(i=n-r_j-a\).
  Therefore,
  \[
    \begin{aligned}
      &[U_1^{d_1}\cdots U_k^{d_k}]\mathcal{K}_{\bm{\mathsf{M}}}(1-U_1,\ldots,1-U_k)\\
      &\qquad=\int_{X_E}\prod_{j=1}^k\bigl(c_{n-r_j-d_j}(\mathcal{Q}_{\mathsf M_j})+\text{terms of higher Chow degree}\bigr)\td(X_E).
    \end{aligned}
  \]
  Since \(\sum_{j=1}^{k}(n-r_j-d_j)=|\bm l|=n-1\), only the lowest-degree term from each factor can contribute to the integral. Hence,
  \[
    [U_1^{d_1}\cdots U_k^{d_k}]\mathcal{K}_{\bm{\mathsf{M}}}(1-U_1,\ldots,1-U_k)=\int_{X_E}\prod_{j=1}^k c_{n-r_j-d_j}(\mathcal{Q}_{\mathsf M_j}).
  \]
  Since \(n-r_j-d_j=l_j\) and \(n-d_j=r_j+l_j\) for every \(j\), comparing the two coefficient formulas proves the claimed identity.
\end{proof}

In particular, the theorem realizes these Chern numbers as the multidegree distribution of the facets of \(\Delta\). Although \(\Delta\) depends on the generic tuple \(\bm w\), this distribution does not. Their nonnegativity follows directly from the fan displacement rule, since the Chern classes \(c_i(\mathcal Q_{\mathsf M_j})\) are represented by nonnegative Minkowski weights. Moreover, \cite[Question~1.4]{BergetEurSpinkTseng2023} asks for combinatorial interpretations of products of Schur classes of \(\mathcal{S}_\mathsf{M}^\vee\) and \(\mathcal{Q}_\mathsf{M}\), together with analogous positivity and log-concavity properties. When specialized to \(\mathsf M_1=\cdots=\mathsf M_k=\mathsf M\), the theorem above partially answers the question for products of column Schur classes of \(\mathcal{Q}_\mathsf{M}\), since \(c_{l_j}(\mathcal{Q}_\mathsf{M})=s_{(1^{l_j})}(\mathcal{Q}_\mathsf{M})\).

\appendix
\crefalias{section}{appendix}
\renewcommand{\thetheorem}{\thesection.\arabic{theorem}}
\renewcommand{\thelemma}{\thetheorem}
\renewcommand{\theproposition}{\thetheorem}
\renewcommand{\thecorollary}{\thetheorem}
\renewcommand{\theconjecture}{\thetheorem}
\renewcommand{\theremark}{\thetheorem}
\renewcommand{\thedefinition}{\thetheorem}
\renewcommand{\theexample}{\thetheorem}
\section{An alternative proof of antiample positivity}\label{sec:ample-k-positivity}
In this appendix, we use Grothendieck weights to give an alternative proof of the following \(K\)-theoretic positivity result.

\corCAntiample*

We already deduced this result from \cref{thm:nef-tautological-positivity} in \cref{sec:nef-tautological-positivity}.  It is also a corollary of \cite[Theorem~B]{eur2025vanishingtheoremscombinatorialgeometries}.  We give the alternative proof because it serves as a model for the proof of \cref{thm:B-CM-Delta}.

We first rewrite \cref{thm:C-example-application} using the product rule.  Fix a sufficiently generic lattice vector \(w\in N\) as in \cref{subsec:grothendieck-weights}. To keep notation simple, we define
\[
  \Gamma_w = \{(\sigma,\tau):\sigma,\tau\in \Sigma_E,(\sigma+w)\cap\tau\neq\varnothing,
  \sigma\cap\tau=\{0\}.\}
\]
\begin{lemma}\label{lem:matroid-euler-product-rule}
  For an element \(\xi\in K(X_E)\) with corresponding GW \(g_\xi\), we have
  \[
    \chi_\mathsf{M}(\xi)=\sum_{\substack{(\sigma_\mathcal{F},\sigma_\mathcal{G})\in \Gamma_w\\\sigma_\mathcal{F}\in \Sigma_\mathsf{M}}}(-1)^{\ell(\mathcal{F})+\ell(\mathcal{G})-n+1}g_\xi(\mathcal{G}).
  \]
\end{lemma}
\begin{proof}
  By \eqref{eq:matroid-euler-as-product}, the left-hand side is \((\Delta_\mathsf M g_\xi)(\{0\})\).  Apply \cref{prop:gw-product-rule} with \(d=\dim X_E=n-1\).  The value of \(\Delta_\mathsf M\) is \(1\) precisely on the cones \(\sigma_\mathcal F\) indexed by flags of flats, by \eqref{eq:bergman-gw}, and boundedness is equivalent to \(\sigma_\mathcal F\cap\sigma_\mathcal G=\{0\}\).  The product rule is therefore exactly the displayed sum.
\end{proof}

To obtain positivity, we regroup the summation by the flag \(\mathcal{G}\). For a fixed cone \(\tau\in\Sigma_E\), set
\[
  c_\tau(\mathsf{M},w)=\sum_{\substack{\sigma\in\Sigma_\mathsf{M}\\(\sigma,\tau)\in\Gamma_w}}(-1)^{\dim\sigma+\dim\tau-n+1}.
\]
We suppress \(\mathsf{M},w\) from the notation henceforth.  Then
\[
  \chi_\mathsf{M}(\xi)=\sum_{\tau\in\Sigma_E}c_\tau g_\xi(\tau).
\]

For \(\tau\in\Sigma_E\), set
\[
  P_\tau=(\Sigma_{\mathsf M}+w)\cap\tau.
\]
The signs of the \(c_\tau\) are controlled by the following two complexes.

\begin{definition}\label{def:incidence-complex}
  Let \(\Pi_E\) be the standard permutohedron with normal fan \(\Sigma_E\), and let \(F_\gamma=\face_\gamma\Pi_E\) be the face with normal cone \(\gamma\).  The \emph{incidence complex} of \(\mathsf M\) and \(w\) is
  \[
    \mathcal C_w(\mathsf M)
    =\{F_\gamma\subseteq\Pi_E:P_\gamma\neq\varnothing\}.
  \]
  For \(\tau\in\Sigma_E\) with \(P_\tau\neq\varnothing\), set \(s=\dim\tau\) and label the rays of \(\tau\) by \([s]\).  For \(A\subseteq[s]\), let \(\tau_A\) be the cone generated by the rays in \(A\); since \(\Sigma_E\) is simplicial, every face of \(\tau\) is a unique \(\tau_A\), and \(\dim\tau_A=|A|\).  Set
  \[
    \Delta_\tau
    =\bigl\{[s]\setminus A:A\subseteq[s],\ P_{\tau_A}\neq\varnothing\bigr\}.
  \]
  We omit \(w\) and \(\mathsf M\) from both notations when they are clear.
\end{definition}

The set \(\Delta_\tau\) is a simplicial complex.  Indeed, if \(B=[s]\setminus A\in\Delta_\tau\) and \(B'\subseteq B\), then \(A'=[s]\setminus B'\) contains \(A\), so \(P_{\tau_A}\subseteq P_{\tau_{A'}}\) and hence \(B'\in\Delta_\tau\).  The empty face belongs to \(\Delta_\tau\) because it corresponds to \(A=[s]\) and \(P_\tau\neq\varnothing\).

\begin{proposition}\label{prop:sign-c-tau}
  For a loopless matroid \(\mathsf M\) of rank \(r\) on \(E\) and the fixed generic vector \(w\) above,
  \[
    (-1)^{n+r-\dim\tau}c_\tau(\mathsf M,w)\geq 0
    \qquad\text{for every \(\tau\in\Sigma_E\)}.
  \]
\end{proposition}

\Cref{thm:C-example-application} is a corollary of \cref{prop:sign-c-tau}.
\begin{proof}[Proof of \cref{thm:C-example-application}]
  Every line bundle on \(X_E\) is represented by a torus-invariant Cartier divisor.  Choose such a representative of \(\mathcal L\), and let \(P\) be its lattice polytope.  Recall from \eqref{eq:gw-line-bundle} that Ehrhart reciprocity gives, for \(\tau\in\Sigma_E\),
  \[
    g_{\mathcal L^{-1}}(\tau)
    =(-1)^{n-1-\dim\tau}|\face_\tau(P)^\circ\cap M|.
  \]
  Substituting this into the expression above for \(\chi_\mathsf M\) gives
  \[
    \begin{aligned}
      (-1)^{r-1}\chi_\mathsf M(\mathcal L^{-1})
      & =(-1)^{r-1}\sum_{\tau\in\Sigma_E}
      c_\tau g_{\mathcal L^{-1}}(\tau)     \\
      & =\sum_{\tau\in\Sigma_E}
      (-1)^{n+r-\dim\tau}c_\tau
      |\face_\tau(P)^\circ\cap M|.
    \end{aligned}
  \]
  This summation is nonnegative by \cref{prop:sign-c-tau}.
\end{proof}

\begin{proof}[Proof of \cref{prop:sign-c-tau}]
  For \(\tau\in\Sigma_E\), let \(P_\tau^b\) be the bounded subcomplex of \(P_\tau\).  By genericity of \(w\), the nonempty intersections \((\sigma^\circ+w)\cap\gamma^\circ\), with \(\sigma\in\Sigma_\mathsf{M}\) and \(\gamma\leq\tau\), are precisely the cells of \(P_\tau\), and their dimensions are \(\dim\sigma+\dim\gamma-n+1\).  Such a cell is bounded exactly when \(\sigma\cap\gamma=\{0\}\).  Hence
  \[
    \begin{aligned}
      \chi(P_\tau^b)
      & =\sum_{\gamma\leq\tau}
      \sum_{\substack{\sigma\in\Sigma_\mathsf{M}\\
      (\sigma,\gamma)\in\Gamma_w}}
      (-1)^{\dim\sigma+\dim\gamma-n+1} \\
      & =\sum_{\gamma\leq\tau}c_\gamma.
    \end{aligned}
  \]

  Both \(\Sigma_\mathsf{M}+w\) and \(\tau\) are tropically convex, so \(P_\tau\) is tropically convex and, whenever it is nonempty, contractible by \cite[Theorem~2]{develin2004tropical}.  Moreover, each cell has pointed recession cone \(\sigma\cap\gamma\).  The cellwise deformation-retraction argument in \cite[Proof of Proposition~3.4] {eur2025vanishingtheoremscombinatorialgeometries} therefore gives a deformation retraction \(P_\tau\to P_\tau^b\).  Thus, with \(\chi(\varnothing)=0\),
  \[
    \chi(P_\tau^b)=\mathbf 1_{P_\tau\neq\varnothing}.
  \]
  M\"obius inversion on the Boolean face lattice of \(\tau\) gives
  \begin{equation}\label{eq:c-tau-mobius}
    c_\tau
    =\sum_{\gamma\leq\tau}
    (-1)^{\dim\tau-\dim\gamma}
    \mathbf 1_{P_\gamma\neq\varnothing}.
  \end{equation}

  If \(\dim\tau<n-r\), then genericity forces \(P_\tau=\varnothing\), since \(\dim\Sigma_\mathsf{M}+\dim\tau<n-1\).  The same is then true for every \(P_\gamma\) with \(\gamma\leq\tau\), so the formula above gives \(c_\tau=0\).

  If \(P_\tau=\varnothing\), then the same formula again gives \(c_\tau=0\).  Suppose that \(P_\tau\neq\varnothing\), and use the notation \(s\), \(\tau_A\), and \(\Delta_\tau\) of \cref{def:incidence-complex}.  The formula above gives
  \[
    \begin{aligned}
      \widetilde\chi(\Delta_\tau)
      & =\sum_{\substack{A\subseteq[s]\\
      P_{\tau_A}\neq\varnothing}}
      (-1)^{s-|A|-1} \\
      & =-c_\tau.
    \end{aligned}
  \]

  Consider the tropical initial degeneration
  \[
    Y=\operatorname{ind}_w\mathsf M
  \]
  of \cref{def:tropical-initial-degeneration}.  It is a reduced, pure \((r-1)\)-dimensional union of torus-orbit closures, and it is Cohen--Macaulay by \cref{thm:efl-cm}.  By \cref{lem:efl-orbit-criterion},
  \[
    V(\gamma)\subseteq Y
    \quad\Longleftrightarrow\quad
    P_\gamma=\gamma\cap(w+\Sigma_\mathsf M)\neq\varnothing.
  \]

  On the affine toric chart \(U_\tau\), let \(x_i\) be the coordinate corresponding to the \(i\)-th ray of \(\tau\).  Recall that \(\mathbb G_m=\operatorname{Spec}\k[t^{\pm1}]\).  The description of \(Y\) as a reduced union of orbit closures gives
  \[
    V(\tau_A)\cap U_\tau
    \simeq\operatorname{Spec}\k[x_i:i\notin A]
    \times(\mathbb G_m)^{n-1-s}.
  \]
  Only the orbit closures indexed by faces of \(\tau\) meet \(U_\tau\). Therefore, the definition of \(\Delta_\tau\) gives, scheme-theoretically,
  \[
    Y\cap U_\tau
    \simeq
    \operatorname{Spec}\k[\Delta_\tau]
    \times(\mathbb G_m)^{n-1-s}.
  \]
  The left-hand side is Cohen--Macaulay because it is open in \(Y\). Since Laurent polynomial extension reflects Cohen--Macaulayness, \(\k[\Delta_\tau]\) is Cohen--Macaulay; equivalently, \(\Delta_\tau\) is Cohen--Macaulay over \(\k\) by \cref{def:cm-simplicial-complex}.

  Since \(P_\tau\neq\varnothing\), the open subscheme \(Y\cap U_\tau\) is nonempty and has dimension \(r-1\).  The preceding product decomposition therefore gives
  \[
    r-1
    =\dim\Delta_\tau+1+n-1-s.
  \]
  Thus
  \[
    d_\tau
    =\dim\Delta_\tau
    =s-(n-r)-1.
  \]
  By \cref{thm:reisner-criterion}, the reduced homology of \(\Delta_\tau\) vanishes below degree \(d_\tau\).  Therefore
  \[
    \widetilde\chi(\Delta_\tau)
    =(-1)^{d_\tau}
    \dim_\k\widetilde H_{d_\tau}
    (\Delta_\tau;\k),
  \]
  and hence
  \[
    (-1)^{n+r-\dim\tau}c_\tau
    =\dim_\k\widetilde H_{d_\tau}
    (\Delta_\tau;\k)
    \geq0.
  \]
  This includes \(d_\tau=-1\), with the usual reduced-homology convention for \(\{\varnothing\}\).
\end{proof}

The complexes \(\Delta_\tau\) are not merely an auxiliary device: they are the links of the single complex \(\mathcal C_w(\mathsf M)\), which is itself contractible and Cohen--Macaulay.  None of this is needed for \cref{prop:sign-c-tau}, but it is what suggests \cref{conj:incidence-complex-shellable} below, and it is the model for the treatment of tuples in \cref{sec:k-fold-positivity}.

\begin{proposition}\label{prop:incidence-complex}
  Let \(\mathsf M\) be a loopless matroid of rank \(r\) on \(E\), and let \(w\) be sufficiently generic.  Then
  \begin{enumerate}
    \item \(\mathcal C=\mathcal C_w(\mathsf M)\) is a polytopal subcomplex of
      \(\Pi_E\), pure of dimension \(r-1\);
    \item \(\link_{\mathcal C}(F_\tau)=\Delta_\tau\) for every
      \(\tau\in\Sigma_E\) with \(P_\tau\neq\varnothing\);
    \item \(|\mathcal C|\) is contractible;
    \item \(\mathcal C\) is Cohen--Macaulay over \(\k\).
  \end{enumerate}
\end{proposition}

\begin{proof}
  The normal-cone correspondence gives
  \[
    F_\delta\subseteq F_\gamma
    \quad\Longleftrightarrow\quad
    \delta\supseteq\gamma.
  \]
  If \(\delta\supseteq\gamma\), then \(P_\gamma\subseteq P_\delta\).  Hence \(\mathcal C\) is a polytopal subcomplex of \(\Pi_E\).  The orbit-closure description of \(Y=\operatorname{ind}_w\mathsf M\) in the proof of \cref{prop:sign-c-tau} shows that \(\mathcal C\) indexes the orbit closures contained in \(Y\).  Its maximal faces therefore index the irreducible components of \(Y\), so \(\mathcal C\) is pure of dimension \(r-1\).  This proves (1).

  Recall that, for a face \(F\) of a polytopal complex \(\mathcal P\), the face poset of \(\link_{\mathcal P}(F)\) is naturally identified with
  \[
    \{H\in\mathcal P:F\subseteq H\},
  \]
  where \(F\) corresponds to the empty face, inclusions are preserved, and the face corresponding to \(H\) has dimension \(\dim H-\dim F-1\).  For a simplicial complex, this agrees with \cref{def:cm-simplicial-complex}.

  We now prove (2).  Suppose that \(P_\tau\neq\varnothing\), and use the notation \(s\) and \(\tau_A\) of \cref{def:incidence-complex}.  The faces of \(\Pi_E\) containing \(F_\tau\) are precisely
  \[
    \{F_\gamma:\gamma\leq\tau\}
    =\{F_{\tau_A}:A\subseteq[s]\}.
  \]
  The correspondence
  \[
    F_{\tau_A}\longmapsto[s]\setminus A
  \]
  identifies this face poset with that of the simplex on \([s]\): it sends \(F_\tau\) to the empty face and preserves inclusions because
  \[
    F_{\tau_A}\subseteq F_{\tau_B}
    \quad\Longleftrightarrow\quad
    A\supseteq B
    \quad\Longleftrightarrow\quad
    [s]\setminus A\subseteq[s]\setminus B.
  \]
  Restricting to the faces that belong to \(\mathcal C\) now gives
  \[
    \begin{aligned}
      \link_{\mathcal C}(F_\tau)
      & =\{[s]\setminus A:A\subseteq[s],\ F_{\tau_A}\in\mathcal C\} \\
      & =\{[s]\setminus A:A\subseteq[s],\
      P_{\tau_A}\neq\varnothing\} \\
      & =\Delta_\tau.
    \end{aligned}
  \]

  For (3), we compare \(|\mathcal C|\) with \(\Sigma_{\mathsf M}+w\) using the nerve theorem. For a face \(F\in\mathcal C\), let \(\operatorname{st}^{\circ}_{\mathcal C}(F)\) denote its open star, the union of the relative interiors of all faces of \(\mathcal C\) that contain \(F\). For each vertex \(v\) of \(\mathcal C\), let \(\sigma_v\) be its normal cone in \(\Pi_E\), and abbreviate
  \[
    P_v=P_{\sigma_v}=(\Sigma_{\mathsf M}+w)\cap\sigma_v,
    \qquad
    V_v=\operatorname{st}^{\circ}_{\mathcal C}(v).
  \]
  After subdividing \(\Sigma_{\mathsf M}+w\) by the braid fan, the subcomplexes \(P_v\) form a finite closed cover of \(\Sigma_{\mathsf M}+w\), while the sets \(V_v\) form an open cover of \(|\mathcal C|\). Indeed, the braid fan is complete, and every maximal cone meeting \(\Sigma_{\mathsf M}+w\) corresponds by definition to a vertex of \(\mathcal C\). We will show that these two covers share the same nerve. For a nonempty set of vertices \(S\), let \(F_S\) be the smallest face of \(\Pi_E\) containing \(S\), and set \(\gamma_S=\bigcap_{v\in S}\sigma_v\), its normal cone.  Then
  \[
    \bigcap_{v\in S}P_v=P_{\gamma_S},
    \qquad
    \bigcap_{v\in S}V_v=
    \begin{cases}
      \operatorname{st}^{\circ}_{\mathcal C}(F_S), & F_S\in\mathcal C,    \\
      \varnothing,                                 & F_S\notin\mathcal C.
    \end{cases}
  \]
  Since \(F_S\) has normal cone \(\gamma_S\), the definition of \(\mathcal C\) gives
  \[
    P_{\gamma_S}\neq\varnothing
    \quad\Longleftrightarrow\quad
    F_S\in\mathcal C.
  \]
  Thus the two covers have the same nerve \(\mathcal N\). Each nonempty intersection \(P_{\gamma_S}\) is contractible by tropical convexity, while each nonempty intersection of \(V_v\) is contractible by the nature of open star. Hence they are both good covers. The nerve theorem \cite[Theorem~10.7]{bjorner1995topological} therefore gives
  \[
    \Sigma_{\mathsf M}+w\simeq|\mathcal N|\simeq |\mathcal{C}|.
  \]
  Since \(\Sigma_{\mathsf M}+w\) is contractible, so is \(|\mathcal C|\).

  For (4), the polytopal form of \cref{thm:reisner-criterion} reduces the claim to
  \[
    \widetilde H_i\bigl(\link_{\mathcal C}(F);\k\bigr)=0
    \qquad
    \text{for every face \(F\in\mathcal C\) and every
    \(i<\dim\link_{\mathcal C}(F)\)}.
  \]
  If \(F\) is nonempty, write \(F=F_\tau\).  Then \(P_\tau\neq\varnothing\), and the link identification above gives \(\link_{\mathcal C}(F)=\Delta_\tau\).  The proof of \cref{prop:sign-c-tau} shows that \(\Delta_\tau\) is Cohen--Macaulay, so the required vanishing follows.  For the empty face, the link is \(\mathcal C\) itself, whose reduced homology vanishes because \(\mathcal C\) is contractible.  Hence \(\mathcal C\) is Cohen--Macaulay over \(\k\).
\end{proof}

\begin{remark}
  For \(\dim\tau\geq n-r\), set \(i_\tau(\mathsf M)=\mathbf 1_{P_\tau\neq\varnothing}\).  Under the order reversal \(V(\gamma)\supseteq V(\tau)\Longleftrightarrow\gamma\leq\tau\), \eqref{eq:c-tau-mobius} is the M\"obius-inversion form of the recurrence defining \(c_\tau(\mathsf M)\) in \cite[Proof of Proposition~3.4] {eur2025vanishingtheoremscombinatorialgeometries}.  Thus \(c_\tau(\mathsf M,w)\) agrees with the coefficient denoted \(c_\tau(\mathsf M)\) there; for \(\dim\tau<n-r\), our coefficient vanishes by the argument above.
\end{remark}

\Cref{prop:incidence-complex} suggests the following combinatorial strengthening.

\begin{conjecture}\label{conj:incidence-complex-shellable}
  For every loopless matroid \(\mathsf M\) on \(E\) and every sufficiently generic vector \(w\in N_\R\), the pure polytopal complex \(\mathcal C_w(\mathsf M)\) is shellable.
\end{conjecture}

\section{Proof of the higher product rule}\label{sec:proof-higher-product-rule}

We retain the notation introduced before \cref{thm:equivariant-multi-product-rule}, establish the required auxiliary results, and then prove the theorem.

\begin{lemma}\label{lem:jointly-generic-locus}
  The jointly generic tuples form a dense rational polyhedral open subset of
  \((N_\R)^k\).  Every chamber contains an integral tuple.
\end{lemma}

\begin{proof}
  This is the \(k\)-fold version of
  \cite[Lemma~5.2]{wang2026grothendieckweights}.  For a cone tuple
  \(\bm\sigma\), set
  \[
    D_{\bm\sigma}
    =
    \{\bm v\in(N_\R)^k:C(\bm\sigma;\bm v)\neq\varnothing\}.
  \]
  Equivalently,
  \[
    D_{\bm\sigma}
    =
    \{(p-x_1,\ldots,p-x_k):
    p\in N_\R,\ x_i\in\sigma_i\text{ for }1\leq i\leq k\},
  \]
  so \(D_{\bm\sigma}\) is a rational polyhedral cone.  For a polyhedron
  \(P\), write \(P^\circ\) for its relative interior and
  \(\partial P=P\setminus P^\circ\).

  Define \(U\) to be the complement in \((N_\R)^k\) of
  \[
    \bigcup_{\bm\sigma}\partial D_{\bm\sigma}
    \ \cup\
    \bigcup_{\bm\sigma\ \mathrm{nontransverse}}D_{\bm\sigma}.
  \]
  Here \(\bm\sigma\) is nontransverse if
  \[
    \dim\bigcap_i(N_{\sigma_i})_\R
    >
    \sum_i\dim\sigma_i-(k-1)d.
  \]
  Since \(\Sigma_E\) is finite, both unions are finite and closed.  The first
  union has codimension at least one.  For a tuple occurring in the second
  union,
  \[
    \dim\operatorname{span}(D_{\bm\sigma})
    =
    d+\sum_i\dim\sigma_i
    -\dim\bigcap_i(N_{\sigma_i})_\R
    <kd,
  \]
  so it is contained in a proper linear subspace.  Thus \(U\) is dense,
  open, and rational polyhedral.

  Relative interior commutes with linear images, so
  \[
    \bm v\in D_{\bm\sigma}^\circ
    \quad\Longleftrightarrow\quad
    \bigcap_{i=1}^k(\sigma_i^\circ+v_i)\neq\varnothing.
  \]
  If this holds, then
  \[
    \operatorname{aff}C(\bm\sigma;\bm v)
    =
    \bigcap_i\bigl((N_{\sigma_i})_\R+v_i\bigr),
    \qquad
    \dim C(\bm\sigma;\bm v)
    =
    \dim\bigcap_i(N_{\sigma_i})_\R.
  \]
  It follows directly from the definition of \(U\) that \(U\) is precisely
  the jointly generic locus.  Finally, its chambers are rational open cones,
  so each contains an integral tuple.
\end{proof}

Once \(C(\bm\sigma;\bm v)\) is nonempty, its recession cone is
\(\bigcap_{i=1}^k\sigma_i\).  Consequently,
\[
  C(\bm\sigma;\bm v)\text{ is bounded}
  \quad\Longleftrightarrow\quad
  \bigcap_{i=1}^k\sigma_i=\{0\}.
\]

Recall that \(M(\sigma)=\sigma^\perp\cap M\) is the annihilator lattice of
\(\sigma\).
\begin{lemma}\label{lem:braid-multi-cone-index}
  Let \(\sigma_1,\ldots,\sigma_k\in\Sigma_E\).  If
  \[
    \bigcap_{i=1}^k\operatorname{span}_\R(\sigma_i)
    =\{0\},
  \]
  then
  \[
    M(\sigma_1)+\cdots+M(\sigma_k)=M.
  \]
\end{lemma}

\begin{proof}
  First consider an arbitrary cone \(\sigma=\sigma_{\mathcal F}\), where
  \(\mathcal F\) is the flag
  \[
    F_1\subsetneq\cdots\subsetneq F_\ell.
  \]
  Set \(F_0=\varnothing\) and \(F_{\ell+1}=E\), and let
  \[
    B_j=F_j\setminus F_{j-1}
    \qquad(1\leq j\leq\ell+1)
  \]
  be the blocks of the associated ordered partition.  Since \(\sigma\) is generated by the vectors \(\bar e_{F_j}\), an element \(q=(q_r)_{r\in E}\in M\) lies in \(M(\sigma)\) if and only if
  \[
    \sum_{r\in F_j}q_r=0
    \qquad(1\leq j\leq\ell).
  \]
  Taking successive differences, and using
  \(\sum_{r\in E}q_r=0\), this is equivalent to
  \[
    \sum_{r\in B_j}q_r=0
    \qquad(1\leq j\leq\ell+1).
  \]
  Therefore
  \[
    M(\sigma)
    =
    \bigoplus_{j=1}^{\ell+1}
    \left\langle e_r-e_s:r,s\in B_j\right\rangle_\Z.
  \]

  For each \(i\), let
  \(B_{i,1},\ldots,B_{i,\ell_i+1}\) be the blocks associated with
  \(\sigma_i\), where \(\ell_i=\dim\sigma_i\).
  Let \(G\) be the graph on \(E\) in which \(r\) and \(s\) are adjacent if
  they lie in the same block \(B_{i,j}\) for some \(i,j\).  The preceding
  description gives
  \[
    \sum_{i=1}^k M(\sigma_i)
    =
    \left\langle
    e_r-e_s:\{r,s\}\text{ is an edge of }G
    \right\rangle_\Z.
  \]
  After tensoring with \(\R\), the hypothesis gives
  \[
    \left(\sum_{i=1}^k M(\sigma_i)\right)_\R
    =
    \left(
      \bigcap_{i=1}^k\operatorname{span}_\R(\sigma_i)
    \right)^\perp
    =
    M_\R.
  \]
  We claim that the graph \(G\) is connected.
  If the connected components of \(G\) are \(E_1,\ldots,E_c\),
  \[
    \sum_{i=1}^k M(\sigma_i) =
    \left\{
      q\in\Z^E:
      \sum_{r\in E_b}q_r=0\text{ for }1\leq b\leq c
    \right\}.
  \]
  Indeed, if \(r\) and \(s\) lie in the same component, choose a path
  \(r=r_0,r_1,\ldots,r_h=s\).  Then
  \[
    e_r-e_s
    =
    \sum_{t=0}^{h-1}(e_{r_t}-e_{r_{t+1}}).
  \]
  The sublattice generated by all \(e_r-e_s\) where
  \(r,s\in E_b\), \(1\leq b\leq c\), is exactly the right-hand side.
  Thus \(\sum_i M(\sigma_i)\) has rank \(|E|-c\).  Since
  \(\dim M_\R=|E|-1\), the displayed equality forces \(c=1\).  Thus \(G\)
  is connected, and the right-hand side is precisely \(M\).
\end{proof}

The same statement holds for a product of braid fans, by applying
\cref{lem:braid-multi-cone-index} to each factor.

Let \(\delta_k:X_E\to X_E^k\) denote the small diagonal embedding.

\begin{lemma}
  \label{lem:small-multi-diagonal-degeneration}
  Let \(\bm v\in N^k\) be integral and jointly generic.  Then, in
  \(K_T(X_E^k)\),
  \[
    (\delta_k)_*[\mathcal{O}_{X_E}]
    =
    \sum_{\substack{
        \sigma_1,\ldots,\sigma_k\in\Sigma_E\\
        C(\bm\sigma;\bm v)\neq\varnothing\\
    C(\bm\sigma;\bm v)\text{ bounded}}}
    (-1)^{\sum_{i=1}^k\dim\sigma_i-(k-1)d}
    x_{\sigma_1}\boxtimes\cdots\boxtimes x_{\sigma_k}.
  \]
\end{lemma}

\begin{proof}
  The proof is essentially the same as \cite[Section~5]{wang2026grothendieckweights}; we only sketch the steps here.

  Consider the closure in \(\PP^1\times X_E^k\) of
  \[
    \Gamma
    =
    \bigl\{
      (t,\lambda_{-v_1}(t)x,\ldots,\lambda_{-v_k}(t)x)
      :t\in\mathbb G_m,\ x\in X_E
    \bigr\}.
  \]
  This closure is integral and flat over \(\PP^1\).
  It is invariant under the diagonal torus, so its fibers at \(1\) and
  \(0\) have the same equivariant class.  The fiber at \(1\) is the small
  diagonal.

  We identify the fiber at \(0\) on a product of maximal affine charts.  Let
  \(\mathcal A_i\) be the character basis dual to the ray basis of the
  maximal cone defining the \(i\)-th chart.  The small diagonal is defined
  by the toric ideal of the labeled configuration
  \[
    \bigsqcup_{i=1}^k\mathcal A_i\subseteq M.
  \]
  The family gives this configuration the weights
  \[
    \eta(u)=\langle u,v_i\rangle
    \qquad
    (u\in\mathcal A_i,\ 1\leq i\leq k).
  \]
  Repeating the local initial-ideal calculation of
  \cite[Section~5]{wang2026grothendieckweights}, we find that the cell
  determined by \(p\in N_\R\) has zero set
  \[
    \bigsqcup_{i=1}^k
    (\mathcal A_i\cap M(\sigma_i)),
  \]
  where the \(\sigma_i\) are the unique cones such that
  \[
    p-v_i\in\sigma_i^\circ
    \qquad
    \text{for }1\leq i\leq k.
  \]
  The inclusion of cells is the same as the reverse inclusion of the cones.
  Joint genericity implies that the maximal cells are exactly the vertices
  of the common refinement.  For such a cell,
  \(\sum_i\dim\sigma_i=(k-1)d\), so its zero set consists of exactly
  \[
    \sum_{i=1}^k(d-\dim\sigma_i)=d
  \]
  vectors.  Moreover, since \(\cap_i (\sigma_i+v_i)\) is a point, taking
  affine hulls shows that
  \[
    \bigcap_{i=1}^k\operatorname{span}_\R(\sigma_i)=\{0\}.
  \]
  For every \(i\), the vectors
  \(\mathcal A_i\cap M(\sigma_i)\) form a lattice basis of
  \(M(\sigma_i)\).  By \cref{lem:braid-multi-cone-index},
  \(\sum_i M(\sigma_i)=M\).  Hence the union of these bases generates
  \(M\).  Since it consists of \(d\) vectors, it is a
  \(\Z\)-basis of \(M\).  Thus every maximal cell of the regular subdivision
  is a unimodular simplex. The local initial ideals are thus square-free, and the
  special fiber is the reduced union
  \[
    \bigcup_{C(\bm\sigma;\bm v)\text{ a vertex}}
    \prod_{i=1}^kV(\sigma_i).
  \]

  Write \(Z\) for this special fiber.  Let \(P_{\bm v}\) be the poset of
  tuples \(\bm\sigma\) with \(C(\bm\sigma;\bm v)\neq\varnothing\), ordered
  componentwise, after adjoining a minimum \(\hat 0\) and a maximum
  \(\hat 1\).  Joint genericity identifies \(P_{\bm v}\) with the augmented
  face lattice of the common refinement via
  \(\bm\sigma\mapsto C(\bm\sigma;\bm v)\).  In particular, it is a lattice.
  Let \(L_{\bm v}\subseteq P_{\bm v}\) be the sublattice generated by the
  vertex tuples, which are precisely the atoms indexing the components of
  \(Z\).  Scheme-theoretic intersections of these products of orbit closures
  are given by componentwise joins.  Hence the closed-cover exact sequence
  and the crosscut theorem, in the form used in
  \cite[Section~5]{wang2026grothendieckweights}, give
  \[
    [\mathcal O_Z]
    =
    \sum_{\bm\sigma\in L_{\bm v}\setminus\{\hat0,\hat1\}}
    -\mu_{L_{\bm v}}(\hat0,\bm\sigma)
    x_{\sigma_1}\boxtimes\cdots\boxtimes x_{\sigma_k}.
  \]
  For \(P=C(\bm\sigma;\bm v)\), the interval
  \([\hat0,\bm\sigma]\) in \(L_{\bm v}\) is the vertex-facet lattice of \(P\).  The
  vertex-facet M\"obius formula recalled in the same reference therefore gives
  \[
    \mu_{L_{\bm v}}(\hat0,\bm\sigma)
    =
    \begin{cases}
      (-1)^{\dim P+1}, & P\text{ is bounded},\\
      0, & P\text{ is unbounded}.
    \end{cases}
  \]
  Every bounded \(P\) is the join of its vertices, so its tuple belongs to
  \(L_{\bm v}\).  Thus the lattice condition can be dropped, and it follows
  that
  \[
    (\delta_k)_*x_{\{0\}}
    =
    \sum_{\substack{
        C(\bm\sigma;\bm v)\neq\varnothing\\
    C(\bm\sigma;\bm v)\text{ bounded}}}
    (-1)^{\dim C(\bm\sigma;\bm v)}
    x_{\sigma_1}\boxtimes\cdots\boxtimes x_{\sigma_k}.
  \]
  By \eqref{eq:multi-expected-dimension},
  \(\dim C(\bm\sigma;\bm v)
  =\sum_i\dim\sigma_i-(k-1)d\), proving the formula.
\end{proof}

The same proof applies to a product of braid fans, using the product version
of \cref{lem:braid-multi-cone-index}.

\begin{lemma}\label{lem:quotient-commutes-with-intersection}
  Let \(A,B\) be polyhedra in a real vector space \(V\), and let \(C\) be a
  polyhedral cone contained in
  \(\operatorname{rec}(A)\cap\operatorname{rec}(B)\).  For the quotient map
  \[
    \pi:V\longrightarrow V/\operatorname{span}_\R(C),
  \]
  one has
  \[
    \pi(A\cap B)=\pi(A)\cap\pi(B).
  \]
\end{lemma}

\begin{proof}
  Only the reverse inclusion requires proof.  Let
  \(\bar p\in\pi(A)\cap\pi(B)\), and choose \(a\in A\) and \(b\in B\) mapping
  to \(\bar p\).  Since
  \(a-b\in\operatorname{span}_\R(C)=C-C\), there are \(c_A,c_B\in C\) such
  that
  \[
    a-b=c_B-c_A.
  \]
  Hence
  \[
    a+c_A=b+c_B.
  \]
  The left-hand side belongs to \(A\), and the right-hand side belongs to
  \(B\), because \(C\) is contained in both recession cones.  This common
  point maps to \(\bar p\).
\end{proof}

\begin{proof}[Proof of \cref{thm:equivariant-multi-product-rule}]
  First suppose that \(\gamma=\{0\}\) and that \(\bm v\) is integral.  Let
  \(p_i:X_E^k\to X_E\) be the \(i\)-th projection.  After multiplying the
  identity in \cref{lem:small-multi-diagonal-degeneration} by
  \(\prod_i p_i^*\alpha_i\), the projection formula gives
  \[
    \chi_T\left(
      \left(\prod_{i=1}^k p_i^*\alpha_i\right)
      (\delta_k)_*[\mathcal O_{X_E}]
    \right)
    =
    \chi_T\left(
      \delta_k^*\left(\prod_{i=1}^k p_i^*\alpha_i\right)
    \right)=\chi_T\left(\prod_{i=1}^k \alpha_i\right).
  \]
  For each cone tuple on the right-hand side of that identity, the
  equivariant K\"unneth formula gives
  \[
    \chi_T\left(
      \left(\prod_{i=1}^k p_i^*\alpha_i\right)
      (x_{\sigma_1}\boxtimes\cdots\boxtimes x_{\sigma_k})
    \right)
    =
    \prod_{i=1}^k\chi_T(\alpha_i x_{\sigma_i})
    =
    \prod_{i=1}^k g_i(\sigma_i).
  \]
  Applying \(\chi_T\) to \cref{lem:small-multi-diagonal-degeneration}
  therefore yields
  \[
    \left(\prod_{i=1}^k g_i\right)(\{0\})
    =
    \sum_{\substack{
        \sigma_1,\ldots,\sigma_k\in\Sigma_E\\
        C(\bm\sigma;\bm v)\neq\varnothing\\
    C(\bm\sigma;\bm v)\text{ bounded}}}
    (-1)^{\sum_i\dim\sigma_i-(k-1)d}
    \prod_{i=1}^k g_i(\sigma_i).
  \]

  For a real jointly generic \(\bm v\), choose an integral jointly generic
  tuple \(\bm v'\) in the same chamber.  The two tuples determine the same
  nonempty cells, while boundedness depends only on their recession cones.
  Thus the displayed formula for \(\bm v'\) is also the formula for
  \(\bm v\), proving the origin case.

  Now let \(\gamma\in\Sigma_E\) be arbitrary.  Write \(r=\dim\gamma\), and
  let \(B_1,\ldots,B_s\) be the blocks of its ordered partition.  The orbit
  closure \(V(\gamma)\) is the toric variety of
  \[
    \Star_{\Sigma_E}(\gamma)
    \simeq
    \prod_{j=1}^s\Sigma_{B_j},
  \]
  whose dimension is \(\bar d=d-r\).  Let
  \(\iota_\gamma:V(\gamma)\hookrightarrow X_E\) be the inclusion and set
  \(\bar\alpha_i=\iota_\gamma^*\alpha_i\).  The diagonal identity for the
  effective torus of \(V(\gamma)\) remains valid after pullback to \(T\), so
  the origin-case argument applies to \(T\)-equivariant classes on
  \(V(\gamma)\).  For every
  \(\sigma_i\supseteq\gamma\), the weight induced by \(\bar\alpha_i\) on the
  star fan satisfies
  \[
    \bar g_i\bigl(\pi_\gamma(\sigma_i)\bigr)
    =
    \chi_T\left(V(\gamma),
    \bar\alpha_i[\mathcal O_{V(\sigma_i)}]\right)
    =
    \chi_T(X_E,\alpha_i x_{\sigma_i})
    =g_i(\sigma_i).
  \]
  The projection formula also gives
  \[
    \left(\prod_{i=1}^k\bar g_i\right)(\{0\})
    =
    \chi_T\left(V(\gamma),\prod_{i=1}^k\bar\alpha_i\right)
    =
    \chi_T\left(X_E,\left(\prod_{i=1}^k\alpha_i\right)x_\gamma\right)
    =
    \left(\prod_{i=1}^k g_i\right)(\gamma).
  \]

  Let \(\pi_\gamma:N_\R\to N(\gamma)_\R\), and set \(\bar v_i=\pi_\gamma(v_i)\) and \(\bar\sigma_i=\pi_\gamma(\sigma_i)\).  Since \(\gamma\subseteq\operatorname{rec}(\sigma_i+v_i)=\sigma_i\) for every \(i\), repeated application of \cref{lem:quotient-commutes-with-intersection} gives
  \begin{equation}\label{eq:star-cell-quotient}
    \pi_\gamma\bigl(C(\bm\sigma;\bm v)\bigr)
    =
    \bigcap_{i=1}^k(\bar\sigma_i+\bar v_i)
    =
    C(\bar{\bm\sigma};\bar{\bm v}).
  \end{equation}
  Here the induction is valid because every nonempty partial intersection
  still has \(\gamma\) in its recession cone.  In particular, the cell on the
  star fan is nonempty if and only if the original cell is nonempty.

  We next verify that \(\bar{\bm v}\) is jointly generic.  If these cells are nonempty, joint genericity of \(\bm v\) gives a point of \(C(\bm\sigma;\bm v)\) in all the relative interiors \(\sigma_i^\circ+v_i\).  Its image lies in all \(\bar\sigma_i^\circ+\bar v_i\), since a linear map sends the relative interior of a polyhedron onto the relative interior of its image.  Moreover, the same relative-interior property gives
  \[
    \operatorname{aff} C(\bm\sigma;\bm v)
    =
    \bigcap_{i=1}^k
    \left(\operatorname{span}_\R(\sigma_i)+v_i\right).
  \]
  The direction space on the right contains
  \(\operatorname{span}_\R(\gamma)\), which is the kernel of
  \(\pi_\gamma\).  Therefore
  \[
    \begin{aligned}
      \dim C(\bar{\bm\sigma};\bar{\bm v})
      & =\dim C(\bm\sigma;\bm v)-r                 \\
      & =\sum_{i=1}^k\dim\sigma_i-(k-1)d-r         \\
      & =\sum_{i=1}^k\dim\bar\sigma_i-(k-1)\bar d.
    \end{aligned}
  \]
  Thus \(\bar{\bm v}\) has both properties required for joint genericity on
  the star fan.

  By \eqref{eq:star-cell-quotient},
  \[
    C(\bar{\bm\sigma};\bar{\bm v})\text{ is bounded}
    \quad\Longleftrightarrow\quad
    C(\bm\sigma;\bm v)
    \text{ is bounded modulo }(N_\gamma)_\R.
  \]
  Finally, the exponent in the origin formula for the star fan is
  \[
    \begin{aligned}
      \sum_{i=1}^k\dim\bar\sigma_i-(k-1)\bar d
      & =\sum_{i=1}^k(\dim\sigma_i-r)-(k-1)(d-r) \\
      & =\sum_{i=1}^k\dim\sigma_i-(k-1)d-r.
    \end{aligned}
  \]
  Replacing \(-r\) by \(+r\) does not change its parity.  Applying the
  origin formula to the product of braid fans above, and using the preceding
  identifications of cells and weights, gives exactly the formula in the
  statement.
\end{proof}

\printbibliography

\end{document}